\documentclass[11pt]{article}
\usepackage[utf8]{inputenc}
\usepackage[T1]{fontenc}
\usepackage{microtype}
\usepackage{amsmath}
\usepackage{amsfonts}
\usepackage{amsthm}
\usepackage{dsfont}
\usepackage{amssymb}
\usepackage{txfonts}
\usepackage{version}
\usepackage
[
        a4paper,
        left=2.9cm,
        right=2.9cm,
        top=3cm,
        bottom=3cm,
]
{geometry}

\usepackage{dsfont}
\usepackage{tikz}
\usetikzlibrary{positioning,shadows,arrows}
\usetikzlibrary{intersections,shapes.arrows}
\usepackage{url}
\usepackage{hyperref}
\usepackage{authblk}
\usepackage{comment}
\usepackage{cleveref}
\usepackage{enumerate}
\usepackage{enumitem}
\usepackage{csquotes}

\theoremstyle{plain}
\newtheorem{prop}{Proposition}[section]
\newtheorem{lem}[prop]{Lemma}
\newtheorem*{lem*}{Lemma}
\newtheorem{thm}[prop]{Theorem}
\newtheorem{corollary}[prop]{Corollary}
\theoremstyle{remark}
\newtheorem{rmk}[prop]{Remark}

\theoremstyle{definition}

\newtheorem{ex}[prop]{Example}

\makeatletter
\def\mysequence#1{\expandafter\@mysequence\csname c@#1\endcsname}
\def\@mysequence#1{%
  \ifcase#1\or (Reg) \or (Ent)\or (KP)\or (KP$_\sigma$) \or (FL)\else\@ctrerr\fi}
\makeatother

\makeatletter
\newcommand{\mylabel}[2]{#2\def\@currentlabel{#2}\label{#1}}
\makeatother

\newenvironment{ackno}%
    {
    \paragraph{Acknowledgements:} 
    }

\newcommand{\reals}{\mathbb{R}}

\newcommand{\diff}{\mathrm{d}}
\newcommand{\dif}{\mathrm{d}}

\newcommand{\dto}{\rightsquigarrow}

\newcommand{\LC}{\mathcal{L}}

\newcommand{\MC}{\mathcal{M}}
\newcommand{\C}{C}
\newcommand{\X}{\mathcal{X}}
\newcommand{\UC}{\mathcal{U}}

\newcommand{\Y}{\mathcal{Y}}

\newcommand{\AThree}{(A3)}
\newcommand{\AThreeStar}{(A3$^*$)}
\newcommand{\BOne}{(B1)}
\newcommand{\BTwo}{(B2)}
\newcommand{\BTwoStar}{(B2$^*$)}
\newcommand{\probQ}{\nu}
\newcommand{\Rd}{\reals^d}
\newcommand{\model}{\mathcal{M}}

\renewcommand{\AThree}{\textbf{(JW)}}
\renewcommand{\AThreeStar}{\textbf{(JW)$^*$}}

\renewcommand{\BOne}{\textbf{(OP)}}
\renewcommand{\BTwo}{\textbf{(Sup)}}
\renewcommand{\BTwoStar}{\textbf{(Sup)$^*$}}

\newcommand{\norm}[1]{\left\|#1\right\|}
\usepackage{mathtools}

\usepackage[backend=bibtex,style=authoryear,natbib=true, url=false,isbn=false, doi=false,maxnames=3,maxbibnames=5, eprint=false]{biblatex} 
\renewbibmacro{in:}{}

\usepackage{color}
\definecolor{darkraspberry}{rgb}{0.53, 0.15, 0.34}
\definecolor{britishracinggreen}{rgb}{0.0, 0.26, 0.15}
\definecolor{greenalt}{rgb}{0.24, 0.5,0} 
\definecolor{burntumber}{rgb}{0.54, 0.2, 0.14}
\definecolor{royalblue}{RGB}{0,78,156}

\providecommand{\reals}{\mathbb{R}}

\providecommand{\NN}{\mathbb{N}}
\providecommand{\RR}{\mathbb{R}}

\providecommand{\oh}{\mathrm{o}}

\providecommand{\eqd}{\stackrel{d}{=}}
\providecommand{\dto}{\rightsquigarrow}
\providecommand{\pto}{\overset{\textup{P}}{\to}}
\providecommand{\ptoStar}{\xrightarrow{\textup{P}^*}}
\providecommand{\ptoOut}{\ptoStar}
\providecommand{\pOut}{{\prob^{*}}}
\providecommand{\as}{\xrightarrow{\textup{a.s.}}}

\providecommand{\eps}{\varepsilon} 
 
\providecommand{\diff}{\mathrm{d}}
\providecommand{\prob}{\mathrm{P}}
\providecommand{\Cov}[1]{\textup{Cov}[#1]}
\providecommand{\Prob}{\mathbb{P}}

\providecommand{\EE}{\mathbb{E}}
\providecommand{\Gproc}{\mathbb{G}}
\providecommand{\argmin}{\operatornamewithlimits{\arg\min}}

\providecommand{\supp}{\operatorname{supp}}

\providecommand{\Law}{\mathcal{L}}
\providecommand{\AC}{\mathcal{A}}
\providecommand{\NC}{\mathcal{N}}

\providecommand{\BL}{\textup{BL}}

\providecommand{\FC}{\mathcal{F}}
\providecommand{\HC}{\mathcal{H}}

\providecommand{\GC}{\mathcal{G}}
\providecommand{\PC}{\mathcal{P}}

\providecommand{\XC}{\mathcal{X}}

\providecommand{\YC}{\mathcal{Y}}

\providecommand{\SO}{\textup{SO}}

\providecommand{\closure}{\textup{Cl}}

\providecommand{\Id}{\textup{Id}}
\providecommand{\diam}{\textup{diam}}
\providecommand{\dwasser}{{{W}_2}}

\providecommand{\Range}{\textup{Range}}
\providecommand{\interior}[1]{\textup{int}(#1)}
\DeclareMathOperator*{\argmax}{arg\,max}

\providecommand{\linfty}{\ell^\infty}
\providecommand{\Xib}{X_i^{b}}
\newcommand{\Xboot}[1]{X_{#1}^{b}}
\providecommand{\Yib}{Y_i^{b}}
\providecommand{\muboot}{\mu_{n,k}^{b}}
\providecommand{\nuboot}{\nu_{n,k}^{b}}
\providecommand{\cboot}{c^{b}_{n,k}}

\providecommand{\ctildeboot}{\tilde{c}^{b}_{n,k}}

\makeatletter
\newcommand{\subjclass}[1]{%
  \let\@oldtitle\@title%
  \gdef\@title{\@oldtitle\footnotetext{\emph{MSC 2020 subject classification.} #1}}%
}
\newcommand{\keywords}[1]{%
  \let\@@oldtitle\@title%
  \gdef\@title{\@@oldtitle\footnotetext{\emph{Key words and phrases.} #1.}}%
}
\newcommand\footnoteref[1]{\protected@xdef\@thefnmark{\ref{#1}}\@footnotemark}
\makeatother
\numberwithin{equation}{section}
\DeclareUnicodeCharacter{001C}{AAAAAAA}

\date{\today\vspace{-0.29cm}}

\begin{document}
\pagenumbering{arabic}
\title{Empirical Optimal Transport under Estimated Costs:\\ Distributional Limits and Statistical Applications}
\author[$\dagger$]{Shayan Hundrieser}
\author[$\dagger$]{Gilles Mordant}
\author[$\dagger$]{Christoph A. Weitkamp}
\author[$\dagger$ $\ddagger$ $\ast$]{Axel Munk\vspace{-0.25cm}}
\affil[$\dagger$]{\footnotesize Institute for Mathematical Stochastics, University G\"ottingen, Goldschmidtstra\ss e 7, 37077 G\"ottingen}
\affil[$\ddagger$]{Max Planck Institute for Multidisciplinary Sciences,  Am Fa\ss berg 11, 37077 G\"ottingen}
\affil[$\ast$]{Cluster of Excellence ”Multiscale Bioimaging: from Molecular Machines to Networks of Excitable Cells”(MBExC), University Medical Center, Robert-Koch-Straße 40, 37075 Göttingen\vspace{-0.20cm}}

\keywords{Optimal Transport, central limit theorem, stability analysis, curse of dimensionality, empirical process, bootstrap}
\subjclass{Primary: 60B12, 60F05, 60G15, 62E20, 62F40; Secondary: 90C08, 90C31.}
\maketitle
\begin{abstract}
  \noindent Optimal transport (OT) based data analysis is often faced with the issue that the underlying cost function is (partially) unknown. This paper is concerned with the derivation of distributional limits for the empirical OT value when the cost function and the measures are estimated from data. For statistical inference purposes, but also from the viewpoint of a stability analysis, understanding the fluctuation of such quantities is paramount. Our results find direct application in the problem of goodness-of-fit testing for group families, in machine learning applications where invariant transport costs arise, in the problem of estimating the distance between mixtures of distributions, and for the analysis of empirical sliced OT quantities. 
  
  \noindent The established distributional limits assume either weak convergence of the cost process in uniform norm or that the cost is determined by an optimization problem of the OT value over a fixed parameter space. For the first setting we rely on careful lower and upper bounds for the OT value in terms of the measures and the cost in conjunction with a Skorokhod representation. The second setting is based on a functional delta method for the OT value process over the parameter space. The proof techniques might be of independent interest.
\end{abstract}



\section{Introduction}
\label{sec: Intro}

Statistically sound methods for data analysis relying on the optimal transport (OT) theory (see e.g., \citet{rachev1998mass,villani2008optimal,santambrogio2015optimal}) have won acclaim in recent years. Exemplarily, we mention fitting of generative adversarial networks \citep{arjovsky2017wasserstein}, novel notions of multivariate quantiles  \citep{chernozhukov2017monge,hallin2021distribution} and dependence \citep{nies2021transport,mordant2022measuring,wiesel2022measuring} or tools for causal inference \citep{torous2021optimal}. %

Recall that for Polish spaces $\XC$ and $\YC$ and a  continuous cost function $c\colon\X\times\Y \to \reals$, the OT value between two (Borel) probability measures $\mu \in \PC(\XC)$ and $\nu \in \PC(\YC)$ is defined as 
\begin{align}
\label{eq: OT}
OT(\mu, \nu,c) \coloneqq  \inf_{\pi \in \Pi(\mu, \nu )} \int_{\X\times \Y} c(x,y) \,\diff \pi(x,y), 
\end{align}
where $ \Pi(\mu, \nu )$ denotes the set of couplings of $\mu$ and $\nu$. 
Under mild assumptions \eqref{eq: OT} also admits a dual formulation (see, e.g., \citealt{santambrogio2015optimal}), 
\begin{align}
  \label{eq: OTdual}
  OT(\mu, \nu,c) = \sup_{f \in C(\XC)}\int_\XC f^{cc}(x)\, \dif \mu(x) + \int_\YC f^{c}(y)\,\dif\nu(y),%
  \end{align} 
  where $C(\XC)$ stands for the set of real-valued, continuous functions on $\XC$. Further, $f^c(y)\coloneqq \inf_{x\in \XC} c(x,y) - f(x)$ and $f^{cc}(x) \coloneqq \inf_{y\in \YC}c(x,y) - f^c(y)$ denote cost-transformations of $f$ and $f^{c}$ under $c$, respectively; also often referred to as $c$-transformations.

If $\XC = \YC$ and the cost function $c=d_\XC^p$ is the $p$-th power ($p\geq 1$) of a metric $d_\XC$ on $\XC$ the OT value gives rise to the $p$-Wasserstein distance 
$$W_p(\mu, \nu) \coloneqq (OT(\mu, \nu,d^p_\XC))^{1/p},$$
which defines a metric on the space of probability measures with $p$-th moments \citep[Chapter 6]{villani2008optimal}. This metric is particularly useful for many data analysis tasks due to its potential awareness of the ``inner geometry'' of $\XC$. For instance, interpreting (normalized) images, or more precisely the corresponding pixel locations and intensities, as probability measures, it has been argued that the distance induced by OT corresponds to the natural expectations of what appears close or far away for the human eye \citep{rubner2000earth}. Meanwhile, there is a plenitude of real world showcases where OT based distances (and their associated transport plans) prove useful for applications e.g., in cell biology \citep{tameling2021colocalization}, genetics \citep{evans2012phylogenetic,schiebinger2019optimal}, protein structure analysis \citep{gellert2019substrate, weitkamp2022distribution} or fingerprint analysis \citep{sommerfeld2018inference}, to mention but a few. In these works, the cost function is a given known quantity which is determined by the concrete application, e.g., a tree distance on the space of phylogenetic trees as in \citet{evans2012phylogenetic}. 

However, despite the various successful applications hinted at above, there are situations in which the underlying cost naturally depends on the measures. In certain problems, e.g., Wasserstein based goodness-of-fit testing under group families \citep{hallin2021multivariate} or Wasserstein Procrustes analysis \citep{grave2019unsupervised}, it is central that the underlying OT problem is invariant with respect to certain transformations. This can only be realized by measure-dependent costs. 
Moreover, for sliced OT \citep{bonneel2015sliced}, the Wasserstein distance between multiple one-dimensional projections of measures is computed. Taking the maximum over all directions gives rise to the max-sliced Wasserstein distance \citep{deshpande2019max} which can be viewed in the framework of OT with measure-dependent costs since maximizing directions are determined by the underlying measures. 
Motivated by these considerations, we provide in this work a general framework for the statistical analysis of empirical OT problems under costs that are dependent on the underlying measures. 

Adopting this statistical point of view, we assume that we do not have access to the measures $\mu$ and $\nu$ but only to independent samples $\{X_i\}_{i=1}^n\sim\mu^{\otimes n}$ and  $\{Y_i\}_{i=1}^{m}\sim \nu^{\otimes m}$ with $n, m\in \NN$. Upon defining the empirical measures $\mu_n \coloneqq \frac{1}{n} \sum_{i = 1}^{n} \delta_{X_i}$ and $\nu_m \coloneqq \frac{1}{m} \sum_{i = 1}^{m} \delta_{Y_i}$ and given a random cost function\footnote{Here, $c_{n,m}$ is either a direct estimator for $c$ or chosen via an OT-related optimization problem over a parameter class.} $c_{n,m}$ such that $OT(\mu_n, \nu_m,c_{n,m})$ estimates the quantity $OT(\mu, \nu,c) $, our main focus is on characterizing for $n, m\rightarrow \infty$ with $m/(n+m)\rightarrow \lambda\in (0,1)$ the limit distribution of 
\begin{align}
  \label{eq:LimitLawOT}\sqrt{\frac{nm}{n+m}}\Big( OT(\mu_n, \nu_m, c_{n,m}) - OT(\mu, \nu, c) \Big).
\end{align}
This is of particular interest for asymptotic tests about the relation between $\mu$ and $\nu$ for unknown~$c$ based on the OT value. Further, this enables the derivation of confidence intervals. As it is practically more relevant, we mainly focus on the scenario where both measures $\mu$ and $\nu$ are unknown. However, we stress that our theory also provides distributional limits for the one-sample case, i.e., when only $\mu$ is estimated from data while $\nu$ is assumed to be known (see \Cref{rmk:CommentsOTWeaklyConvergingCosts} and \Cref{rmk:CommentsOnOTProcessResults}). Moreover, although we mostly focus on empirical measures to estimate the underlying measures, our theory also enables the derivation of distributional limits for alternative measure estimators, provided that the corresponding distributional limits can be determined. 

For a fixed cost function, i.e.,  for $c_{n,m}\equiv c$ for some $c\in C(\XC\times \YC)$, already various works derived limit distribution results for the empirical OT quantity in \eqref{eq:LimitLawOT}. 
A specific situation arises for probability measures on $\RR$ with $c_p(x,y) = |x-y|^p$ for $p\geq 1$ \citep{munk1998nonparametric,del1999central,del2005asymptotics, mason2016weighted, del2019central2} where the OT plan can be represented via a quantile coupling. For this setting, quantile process theory \citep{csorgo1993weighted} in combination with integrability conditions on the underlying densities have been exploited to derive distributional limits. 

Moreover, on general Euclidean spaces $\RR^d$ with $d \geq 1$ and $p$-th power costs $c_p(x,y) =\norm{x-y}^p$ with $p>1$ it has been shown by \citet{del2019central, delBarrio2021GeneralCosts} for probability measures 
$\mu, \nu$ with connected support and finite $2p$-th moments for $n, m\rightarrow \infty$ with $m/(n+m)\rightarrow \lambda\in (0,1)$ that 
\begin{align}\label{eq:CLT_delbarrio}
  \sqrt{\frac{nm}{n+m}}\left(OT(\mu_n, \nu_m,c_p) - \EE\left[OT(\mu_n, \nu_m,c_p)\right] \right)\dto \NC(0, \sigma_{\mu, \nu}^2),
\end{align}
where $\sigma_{\mu, \nu}^2>0$ if and only if $\mu \neq \nu$. Here and throughout, ``$\rightsquigarrow$'' denotes weak convergence in the sense of Hoffman-J{\o}rgensen (see \citealt[Chapter 1.3]{van1996weak}). Their proof is based on an $L^2$-linearization technique of the OT value and relies on the Efron-Stein inequality. In general, the centering quantity $\EE[OT(\mu_n, \nu_m,c_p)]$ in \eqref{eq:CLT_delbarrio} cannot be replaced by its population quantity $OT(\mu, \nu,c_p)$ which hinders further statistical inference purposes. Indeed, for identical absolutely continuous probability measures $\mu= \nu$  on $\RR^d$ with sufficiently many moments it follows for $d>2p$ by \citet{fournier2015rate,weed2019sharp} that
\begin{align*}
  \EE\left[OT(\mu_n, \nu_m, c_p)\right]\asymp \min(n, m)^{-p/d}.
\end{align*}
 Moreover, for different measures $\mu\neq \nu$ on $\RR^d$ which are absolutely continuous and sub-Weibull it has been shown for $d \geq  5$ by \citet{Manole21} that 
 \begin{align*}
  \EE\left[OT(\mu_n, \nu_m,c_p)\right] - OT(\mu, \nu,c_p)\asymp  \min(n, m)^{-\min(p,2)/d}.
 \end{align*}
These rates are also minimax optimal (up to logarithmic factors) over appropriate collections of identical measures $\mu = \nu$ \citep{singh2018minimax} as well as different measures $\mu \neq \nu$ \citep{Manole21}. In particular, this demonstrates that estimation of the OT value suffers from the curse of dimensionality and showcases that it is in general for $d\geq 5$, due to the dominance of the bias, not possible to replace  $\EE[OT(\mu_n, \nu_m,c_p)]$ with $OT(\mu, \nu,c_p)$ in \eqref{eq:CLT_delbarrio}. 

Nevertheless, according to the recently discovered \emph{lower complexity adaptation principle} for empirical OT \citep{hundrieser2022empirical}, fast convergence rates are still achieved if one of the population measures, $\mu$ or $\nu$, is supported on a sufficiently low dimensional domain. Based on this observation, \citet{Hundrieser2022Limits} proved for compactly supported  $\mu, \nu$ on $\RR^d$, with $\mu$ supported on a finite set or a smooth submanifold of dimension $\tilde d< 2 \min(p,2)$ using the functional delta method \citep{Roemisch04}, 
\begin{align}\label{eq:LimitLawOT_LCA}
  \sqrt{\frac{nm}{n+m}}\Big( OT(\mu_n, \nu_m,c_p) - OT(\mu, \nu,c_p) \Big)\dto \sup_{f\in S_{\!c_p}(\mu, \nu)} \sqrt{\lambda}\Gproc^\mu(f^{c_pc_p}) + \sqrt{1-\lambda}\Gproc^\nu(f^{c_p}),
  \end{align}
where $S_{\!c_p}(\mu, \nu)$ is the set of optimizers of \eqref{eq: OTdual} and $\Gproc^\mu, \Gproc^\nu$ denote $\mu$-, $\nu$-Brownian bridges, i.e., centered Gaussian processes with covariance structure characterized by 
\begin{align}\label{eq:covarianceStructureEmpProc}
  \Cov{\Gproc^\mu(f^{cc}), \Gproc^\mu(g^{cc})} = \int f^{cc} g^{cc} \dif \mu - \int f^{cc}  \dif \mu\int g^{cc} \dif \mu \quad \text{ for } f,g\in C(\XC)
\end{align}
and likewise for $\Gproc^\nu$. The asymptotic theory laid out in \eqref{eq:LimitLawOT_LCA} also provides a unified framework for distributional limits of the empirical OT value under discrete population measures \citep{sommerfeld2018inference,tameling18} and the semi-discrete setting \citep{del2022central}.

The central contribution of this work is to extend such distributional limits from \eqref{eq:LimitLawOT_LCA} to settings where the cost function is not fixed and additionally may depend on the underlying measures.  We focus on the following two special instances.
\begin{enumerate}
		\item[(A)] The cost estimator $c_{n,m}$, centered by its population counterpart $c$ and suitably rescaled, weakly converges in $C(\XC\times \YC)$ to a tight limit, i.e., $\sqrt{nm/(n+m)}(c_{n,m} -c)\dto \Gproc^c$ in $C(\XC\times \YC)$.
    \item[(B)] There exists a collection $\{c_\theta\}_{\theta\in \Theta}$ of costs such that for any $\mu\in \PC(\XC), \nu \in \PC(\YC)$ the corresponding cost function $c_{\mu, \nu}\coloneqq c_\theta$ is selected according to an optimization problem of the OT value over $\Theta$, i.e., either $\theta \in \arg\max_{\theta \in \Theta} OT(\mu, \nu, c_\theta)$ or $\theta \in \arg\min_{\theta \in \Theta} OT(\mu, \nu, c_\theta)$.  
\end{enumerate}

These two settings are natural and treat a wide spectrum of problems. Furthermore, they are strongly related. It is noteworthy that setting (B) could be treated in the framework of (A) by estimating the optimal $\theta$. However, this approach requires the existence of a unique population cost function and weak convergence of the cost process as a random element in $C(\XC\times \YC)$. Since we are only interested in the empirical infimal or supremal OT value it is instead more natural to rely on an alternative approach which does not require uniqueness of the population cost function or weak convergence of the cost process. 
 
For setting (A) we allow the cost function to be estimated from the given data and thus capture the asymptotic dependency between the cost estimator and the empirical measures. In particular, this enables an analysis of the empirical OT cost when the cost estimator is parametrized by a plug-in estimator, e.g., a maximum likelihood procedure. 
Notably, setting (A) also allows the cost function to be estimated from independent data. Overall, this setting covers many scenarios with ``extrinsically estimated costs''. We refer to Sections \ref{sec: OSGoF} and \ref{sec: SketchWas} for examples. For setting (B) the motivation slightly differs. Here, the selected cost function depends on the OT problem itself and often brings invariance of the OT problem with respect to a class of transformation parametrized by $\Theta$. One could describe this as OT with ``intrinsically estimated costs''. Examples of this setting are provided in Sections \ref{sec: OTwInv} and \ref{sec: Sliced}. 

Under suitable assumptions we show in \Cref{thm:AbstractMainResult} for setting (A) that
\begin{align*}
  \sqrt{\frac{nm}{n+m}}\Big(OT( \mu_n, \nu_m, c_{n,m})-OT( \mu, \nu, c)\Big) \!\dto\!\! &\inf_{\pi \in \Pi_{c}^\star(\mu,\nu)}\pi(\Gproc^c) + \!\!\sup_{f \in S\!_c(\mu,\nu)} \!\!\!\sqrt{\lambda}\Gproc^\mu(f^{cc})\! + \!\!\sqrt{1\!-\!\lambda}\Gproc^\nu(f^c),
\end{align*}
where $\Pi^\star_c(\mu, \nu)$ represents the set of optimizers for \eqref{eq: OT} for $\mu,\nu$ with costs $c$ and $\pi(\Gproc^c)\coloneqq \int \Gproc^c \diff \pi$. 
For setting (B) we only state below the distributional limit for supremal costs; a similar distributional limit also occurs for infimal costs (\Cref{thm:OTProcessInf}). 
Upon defining the set $S_{\!+}(\Theta, \mu, \nu) = \argmax_{\theta\in \Theta}OT(\mu, \nu,c_\theta)$ of maximizers we show in \Cref{thm:OTProcessSup} that
\begin{align*}
\sqrt{\frac{nm}{n+m}}\left(\sup_{\theta\in \Theta} OT( \mu_n, \nu_m,c_\theta)- \sup_{\theta\in \Theta} OT(\mu, \nu,c_\theta)\right)\!\dto\!\! \sup_{\theta\in S_{\!+}(\Theta, \mu, \nu)} \sup_{f_\theta\in S_{\!c_\theta}(\mu, \nu)}\!\!  \sqrt{\lambda}\Gproc^\mu(f_\theta^{c_\theta c_\theta})\!+\!\!  \sqrt{1\!-\!\lambda}\Gproc^\nu(f_\theta^{c_\theta}).
\end{align*}

In addition to these distributional limits we show for both settings (A) and (B) consistency of a bootstrap principle. This is of practical importance since quantiles of the respective distributional limits are difficult to express explicitly due to their dependency on the collection of primal and dual optimizers for population measures and cost.

Our proof technique for the distributional limit under setting (A) differs from previous approaches and might be of interest in its own right. More precisely, due to the estimation of the cost function, we cannot rely on any of the techniques from the references mentioned above. Instead, we derive certain lower and upper bounds on the OT value which fulfill appropriate (semi-)continuity properties. In~conjunction with a Skorokhod representation for the empirical process jointly with the cost process, this enables us to prove that the law of the empirical OT value with estimated costs is asymptotically stochastically dominated from above and below by the asserted limit distribution. %

For the analysis of setting (B) we show under suitable assumptions on the cost family $\{c_\theta\}_{\theta\in \Theta}$ and the underlying probability measures, that the empirical OT process $\sqrt{n}(OT(\mu_n, \nu_m,c_\theta)- OT(\mu, \nu,c_\theta))_{\theta \in \Theta}$ weakly converges in $C(\Theta)$ to a tight random variable. We prove this result by invoking the functional delta method in conjunction with a general result on Hadamard directional differentiability for extremal-type functionals uniformly over a compact parameter space (see \Cref{subsec:HadDiff}). The latter can be viewed as an extension of \citet[Lemma S.4.9]{fang2019inference} to processes over $\Theta$ and relies on Dini's theorem (\citealt[Corollary 1]{toma1997strong}). Central for this differentiability result is a certain continuity condition among the sets of maximizing elements for varying parameter. For the OT process it is fulfilled, e.g., if for every $\theta\in \Theta$ the set of dual optimizer $S_{\!c_\theta}(\mu, \nu)$ is unique (up to constant shift). A similar assumption has been imposed by \citet{xi2022distributional} for weak convergence of the empirical sliced OT process, which can be viewed as a special instance of our results for general OT processes, see Section~\ref{sec: Sliced}. The distributional limits for the empirical infimal and supremal OT value over $\theta\in \Theta$ then follow by another application of the functional delta method. 

\paragraph*{Outline}
We begin our exposition by deriving in \Cref{subsec:preliminaries} an appropriate dual formulation of the OT value which proves useful for our subsequent considerations. We then proceed with our main contributions, distributional limits for the empirical OT value under weakly converging costs in \Cref{subsec:DistributionalLimits} as well as for the empirical OT value under extremal-type costs in \Cref{subsec:OTProcessResults}. These asymptotic results are complemented with consistency results of bootstrap resampling schemes in \Cref{subsec:BootstrapPrinciple}. We discuss our assumptions for the distributional limits and the bootstrap principles in \Cref{subsec:Assumptions} and provide sufficient conditions for their validity. Statistical applications of our theory are provided in \Cref{sec:Applications}, where we also derive a deterministic (first-order) stability result for the OT cost under joint perturbations of measures and cost function. 
In \Cref{subsec:RegEl} we explicitly construct functionals which enables us to ``elevate'' the regularity of cost estimators to that of their population counterparts. We employ them in the proofs of our main results which are stated in Section~\ref{sec: ProofMain}. All remaining proofs as well as auxiliary results and lemmata are relegated to the Appendices. 

\paragraph*{Notation and probability spaces}

Given a set $T$ denote by $\ell^\infty(T)$ the Banach space of bounded functionals on $T$ equipped with uniform norm $\| \varphi \| \coloneqq \sup_{t\in T}|\varphi(t)|$. Moreover, if $T$ is equipped with a topology $\tau$ denote by $C(T)$ the Banach space of real valued, bounded, continuous functions on $T$ equipped with uniform norm. If $d_T$ denotes a metric on $T$, then we define by $C_u(T,d_T)$, or $C_u(T)$ when the metric $d_T$ is clear from context, the space of real-valued, bounded, uniformly continuous functions on $(T,d)$. Endowed with the uniform norm, it is a Banach space as well. A real-valued function class $\FC$ on $\XC$ is always be equipped with uniform norm. This specifies the Banach space $C_u(\FC)$ which is a closed  subset of the Banach space $\ell^\infty(\FC)$. Moreover, for $\eps>0$ the covering number $\NC(\eps, T, d)$ denotes the minimal number of sets with diameter $2\eps$ to cover $T$, and we write $x\lesssim y$ when there exists a constant $C>0$ with $x\le C y$. 

For a topological space $\XC$ the set $\PC(\XC)$ denotes the collection of Borel probability measures on $\XC$. Integration $\int f \diff \mu$ of a real-valued Borel measurable function $f\colon \XC\rightarrow \RR$ with respect to $\mu \in \PC(\XC)$ is abbreviated by $\mu(f)$ or $\mu f$. Further, we denote by $f_\# \mu$ the pushforward of $\mu$ under~$f$. We define all random variables on the same probability space $(\Omega, \AC, \prob)$. We further assume a product structure of that space to define samples and the random weights of the bootstrap, i.e., $\Omega= \Omega_0 \times \Omega_1 \times \ldots $ and $\prob = \prob_0 \otimes \prob_1 \times \ldots$ so that the samples only depend on $(\Omega_0,\prob_0)$, the weights of the first bootstrap replicate on $(\Omega_1,\prob_1)$ and so on. The law of a random variable $X$ is denoted by $\LC(X).$ We finally assume that there exist infinite sequences of measurable maps $X_1, X_2, \ldots$ from $(\Omega_0,\prob_0)$ to $\X$, respectively, and that samples of cardinality $n$ are obtained from the infinite sequence by projection of the first $n$ coordinates. Outer probability measures are denoted by $\pOut$ (see \cite[Chapter 1.2]{van1996weak}).
Denoting by $\mathrm{BL}_1$ the set of real-valued functions on a metric space $(T, d_T)$ which are bounded by one in uniform norm and such that  $\lvert f(x)- f(y) \rvert \le d_T(x,y)$ for any $x,y\in  T$, we define the bounded Lipschitz metric between two probability measures $\mu, \nu$ as $d_{BL}(\mu,\nu ):= \sup_{f \in \mathrm{BL}_1}\left \lvert \mu(f) - \nu(f)\right\rvert$. 
For a set $A$ and a function $f$, we write $f(A)\coloneqq\{f(a) \ \vert\ a \in A\}$. For two subsets $A, B$ of a vector space,  $A+B := \{a + b\ \vert \ a \in A , b\in B  \}$.
%




\section{Main Results}\label{sec:MainResults}
\subsection{Preliminaries}\label{subsec:preliminaries}

For our theory on distributional limits for the empirical OT value under estimated cost functions we consider throughout compact Polish spaces $\XC$ and $\YC$. %
 Given a continuous cost function $c\in C(\XC\times \YC)$ and  probability measures $\mu \in \PC(\XC), \nu\in \PC(\YC)$ there always exist optimizers to both primal and dual problem \citep[Theorems 4.1 and 5.10]{villani2008optimal}. 

According to  \citet[Remark 1.13]{villani2021topics}, dual optimizers can always be selected from the function class 
\begin{align}\label{eq:C-concaveFunctionsDef}
  \HC_c := \left\{h : \X\to \reals \ \Big\vert \ \exists g \colon \Y \to [-\norm{c}_\infty, \norm{c}_\infty], h(\cdot)=\inf_{y \in \Y} c(\cdot,y) -g(y)\right\},
\end{align}
which yields for any $\mu\in \PC(\XC), \nu \in \PC(\YC)$ the alternative dual representation of the OT value, 
\begin{align}\label{eq:DualFormulationSpecificC}
	OT(\mu, \nu,c) =  \sup_{h\in \HC_c} \mu(h^{cc}) + \nu(h^c).%
\end{align}
The function class $\HC_c$ is uniformly bounded and each element exhibits the same modulus of continuity as $c$, hence it is compact in $C(\XC)$ by the Theorem of Arzel\`a-Ascoli. 
Formula \eqref{eq:DualFormulationSpecificC} was exploited by \citet{Hundrieser2022Limits} for distributional limits of the empirical OT value under a fixed cost function. 

 For our purposes, we require a dual formulation over a fixed function class which holds for more than a single cost function and to circumvent potential measurability issues we seek a function class which is compact in $C(\XC)$ (cf. \Cref{lem:measurability}). To this end, let $B>0$ and consider a concave modulus of continuity $w\colon \RR_+ \rightarrow \RR_+$. Then, for a continuous metric $d_\XC$ on $\XC$ we define the compact function class $\FC(B,w)\subseteq C(\XC)$, 
\begin{align}\label{eq:defF}
 \FC(B,w)\coloneqq \left\{ f\colon \XC\rightarrow \RR \;\Big|\; \norm{f}_\infty \leq 2B, \;|f(x) - f(x')| \leq w(d_\XC(x,x')) \;\text{ for all } x,x'\in \XC \right\},
\end{align}
which will be utilized for a dual representation of the OT value under suitable costs. 

\begin{lem}[Dual formulation]
  \label{lem:RelFCandConcave}
  Let $c\in C(\XC\times \YC)$ with $\norm{c}_\infty \leq B$ and $|c(x,y)-c(x',y)| \leq w\left(d_\XC(x,x')\right)$ for all $x,x'\in \XC, y\in \YC$. Then, for $\FC\coloneqq \FC(B, w)$ the following inclusions hold
\begin{align*}%
  \HC_c \subseteq \FC^{cc}\subseteq  \HC_c  + [-2B,2B]\quad \text{ and }\quad \HC_c^{c} \subseteq \FC^{c} \subseteq  \HC_c^{c}  + [-2B,2B].
\end{align*}
  Further,  for arbitrary probability measures $\mu \in \PC(\XC)$ and $\nu \in \PC(\YC)$ it follows that 
  \begin{align}\label{eq:OTDualNice}
  		OT(\mu, \nu,c) = \sup_{f\in \FC} \mu(f^{cc}) + \nu(f^c)
  \end{align}
 and the set of dual optimizers  $S\!_c(\mu,\nu)$ of \eqref{eq:OTDualNice}, referred to as \emph{Kantorovich potentials}, is non-empty.
  \end{lem}

The proof of \Cref{lem:RelFCandConcave} is deferred to \Cref{subsubsec:RelFCandConcaveProof}. Overall, \Cref{lem:RelFCandConcave} justifies the use of the function class $\FC=\FC(B,w)$ for a dual OT formulation and enables us to state conditions of distributional limits in terms of $\FC$ instead of potentially varying collections of functions. 

\subsection{Distributional Limits under Weakly Converging Costs}\label{subsec:DistributionalLimits}

For the distributional limits in all the statements below, we consider independent and identically distributed random variables $\{X_i\}_{i=1}^n\sim\mu^{\otimes n}$ and independent $\{Y_i\}_{i=1}^{m}\sim \nu^{\otimes m}$ defined on the probability space put forward in the introduction. Based on these samples, we define empirical measures $ \mu_n\coloneqq \frac{1}{n}\sum_{i=1}^{n} \delta_{X_i}$ and $ \nu_m\coloneqq \frac{1}{m}\sum_{i=1}^{m} \delta_{Y_i}$. All the subsequent asymptotic results are to be understood for $n,m\to \infty$ with  $m/(n+m) \to \lambda \in (0,1)$, which we do not recall each time for space considerations. 

Our main result on the limit law for the empirical OT value under weakly converging costs is given as follows for the two-sample case. The one-sample case is discussed in Remark \ref{rmk:CommentsOTWeaklyConvergingCosts}\ref{rem:WeaklyConvergingCosts_OneSample}.  

\label{sec: OTestCosts}
\begin{thm}[OT under weakly converging costs]\label{thm:AbstractMainResult}
Let $c\in C(\XC\times \YC)$ and consider an estimator $ c_{n,m} \in C(\XC\times \YC)$ for $c$ such that $c_{n,m}(x,y)$ is measurable for each $(x,y) \in \XC\times \YC$. Let $w\colon \RR_+\rightarrow \RR_+$ be a concave modulus of continuity for $c$ with $w(\delta)>0$ for $\delta>0$ such that $|c(x,y) - c(x',y)|\leq w(d_\XC(x,x'))$ for all $x,x'\in \XC, y\in \YC$. Assume for $\mu\in \PC(\XC), \nu \in \PC(\YC)$ the following. 
\begin{enumerate}[label={\AThree}] 
   \item \label{ass:AThree} For the function class $\FC= \FC(2\norm{c}_\infty+1, 2w)$ from \eqref{eq:defF} joint weak convergence occurs,  \begin{align*}
   \sqrt{\frac{nm}{n+m}} \begin{pmatrix}
\mu_n - \mu\\
\nu_m - \nu\\
 c_{n,m} - c
\end{pmatrix} 
\dto
\begin{pmatrix}
\sqrt{\lambda}\;\Gproc^\mu\\
\sqrt{1-\lambda}\;\Gproc^\nu\\
\Gproc^c
\end{pmatrix}\quad \text{ in }\ell^\infty(\FC^{cc})\times \ell^\infty(\FC^c)\times C(\X\times \Y),
\end{align*}
where $(\Gproc^\mu,\Gproc^\nu,\Gproc^c)$ is a tight random variable and $\Gproc^\mu, \Gproc^\nu$ have covariance structure as in~\eqref{eq:covarianceStructureEmpProc}. 
\end{enumerate}
Further, suppose either one of the following two assumptions. 
\begin{enumerate}[label={\BOne}] 
    \item \label{ass:BOne} There exists a unique OT plan $\pi\in \Pi_{c}^\star(\mu,\nu)$ between $\mu$ and $\nu$ for the cost function $c$. 
\end{enumerate}
\begin{enumerate}[label={\BTwo}] 
    \item \label{ass:BTwo}  The empirical processes $\Gproc_n^{\mu}\coloneqq \sqrt{n}(\mu_n - \mu)$ and $\Gproc_m^{\nu}\coloneqq \sqrt{m}(\nu_m - \nu)$ fulfill the convergence $\sup_{f\in \FC}\Gproc_n^{\mu}(f^{ c_{n,m}c_{n,m}} - f^{cc})\ptoOut 0$ and $\sup_{f\in \FC}\Gproc_m^{\nu}(f^{ c_{n,m}} - f^{c})\ptoOut 0.$
\end{enumerate}
Then, it follows that 
\begin{align*}
     \sqrt{\frac{nm}{n+m}}\Big(OT( \mu_n, \nu_m, c_{n,m})-OT( \mu, \nu,c)\Big) \dto &\inf_{\pi \in \Pi_{c}^\star(\mu,\nu)}\pi(\Gproc^c) + \sup_{f \in S\!_c(\mu,\nu)} \!\!\!\sqrt{\lambda}\;\Gproc^\mu(f^{cc})\! + \!\sqrt{1-\lambda}\;\Gproc^\nu(f^c).
    \end{align*}
\end{thm}

A key insight of \Cref{thm:AbstractMainResult} is that the limit distribution for the estimated OT value can be decomposed into two terms: the fluctuation of the cost estimators evaluated at the collection of OT plans and the Kantorovich potentials evaluated at the limit of the empirical process. Under uniqueness of primal and dual optimizers for the population OT problem we obtain the following. 

\begin{corollary}[OT under weakly converging costs and uniqueness]\label{cor:AbstractMainResult_Uniqueness}
  In the setting of \Cref{thm:AbstractMainResult} assume \ref{ass:AThree}and \ref{ass:BOne}, and %
    suppose that the set of Kantorovich potentials $S\!_c(\mu,\nu)$ for $\mu, \nu$ with cost function $c$ is unique (up to a constant shift)\footnote{By this we mean, for any $f, g\in S\!_c(\mu,\nu)$ the difference $f - g$ is constant on $\supp(\mu)$.}.
   Then, for $\pi\in \Pi_c^\star(\mu, \nu)$ and $f\in S\!_c(\mu, \nu)$, it follows that 
  \begin{align}\label{eq:LimitUnderUniqueness}
   \sqrt{\frac{nm}{n+m}}\Big(OT( \mu_n, \nu_m,c_{n,m})-OT( \mu, \nu,c)\Big) \dto \pi(\Gproc^c) + \sqrt{\lambda}\Gproc^\mu(f^{cc}) +  \sqrt{1-\lambda}\Gproc^\nu(f^c).
   \end{align}
   In particular, if $(\Gproc^\mu, \Gproc^\nu, \Gproc^c)$ is a jointly centered Gaussian process in $\ell^\infty(\FC^{cc})\times \ell^\infty(\FC^c)\times C(\X\times \Y)$, the weak limit in \eqref{eq:LimitUnderUniqueness} is centered normal.
   \end{corollary}
The proof of \Cref{thm:AbstractMainResult} is deferred to \Cref{subsubsec:AbstractMainResultProof} and relies on careful lower and upper bounds for the empirical OT value due to the primal \eqref{eq: OT} and dual formulation \eqref{eq:OTDualNice}, as well as arguments from empirical process theory. In the course of this, a key argument is the application of \Cref{lem:RelFCandConcave} for $c_{n,m}$ and $c$. Notably, we do not demand that the cost estimator $c_{n,m}$ is suitably bounded or exhibits a similar modulus of continuity as $c$ itself. Instead, we 
construct by \Cref{thm:RegElModulus} an alternative cost estimator $\overline c_{n,m}$ such that the conditions, $\|\overline c_{n,m}\|_\infty\leq 2\norm{c}_\infty+1$ as well as $|\overline c_{n,m}(x,y) - \overline c_{n,m}(x',y)|\leq 2w(d_\XC(x,x'))$ for all $x,x'\in \XC, y\in \YC$, are fulfilled deterministically and $\sqrt{nm/(n+m)}\|\overline c_{n,m} - c_{n,m}\|_\infty\pto 0$. The latter implies by \Cref{lem:LowerUpperBound} that
$$\sqrt{\frac{nm}{n+m}}\Big(OT( \mu_n, \nu_m,\overline c_{n,m}) - OT( \mu_n, \nu_m,c_{n,m})\Big) \leq \sqrt{\frac{nm}{n+m}}\|\overline c_{n,m} - c_{n,m}\|_\infty\pto 0.$$
It thus suffices to show the assertion for $\overline c_{n,m}$ where the dual formulation from \Cref{lem:RelFCandConcave} involving the function class $\FC(2\norm{c}_\infty+1, 2w)$ is available. We call $\overline c_{n,m}$ a \emph{regularity elevation} of~$c_{n,m}$; details on different kinds of regularity elevations are given in \Cref{subsec:RegEl}. The notion of regularity elevations also proves to be useful for showing the validity of condition \ref{ass:BTwo} as outlined in \Cref{subsec:AssumSup}.

\begin{rmk}
\label{rmk:CommentsOTWeaklyConvergingCosts}
We like to comment on a few aspects of the derived distributional limits. 
\begin{enumerate}[label=$(\roman*)$]
  \item \label{rem:WeaklyConvergingCosts_AssumptionDiscussion} The assumptions of \Cref{thm:AbstractMainResult} and sufficient conditions for their validity are discussed in Sections \ref{subsec:AssumJW} -- \ref{subsec:AssumSup}. 
  Effectively, \ref{ass:AThree} delimits the theory to settings of low dimensionality. In~such settings \ref{ass:BTwo} is often also valid as long as the population cost is sufficiently regular. 
\item \label{rem:WeaklyConvergingCosts_OneSample} Our proof technique for Theorem \ref{thm:AbstractMainResult} and \Cref{cor:AbstractMainResult_Uniqueness} also asserts distributional limits for the one-sample setting, i.e., when  $\mu$ is estimated by $\mu_n$ and $\nu$ is assumed to be known. For this setting,  \ref{ass:AThree} reduces to the condition \begin{align*}
  \sqrt{n} \begin{pmatrix}
\mu_n - \mu\\
c_{n} - c
\end{pmatrix} 
\dto
\begin{pmatrix}
\Gproc^\mu\\
\Gproc^c
\end{pmatrix}\quad \text{ in }\ell^\infty(\FC^{cc})\times C(\X\times \Y).
\end{align*}
Moreover, in \ref{ass:BTwo} we only require that $\sup_{f\in \FC}\Gproc_n^{\mu}(f^{ c_{n} c_{n}} - f^{cc})\ptoOut 0$. Then,  \begin{align*}
 \sqrt{n}\Big(OT( \mu_n, \nu,c_{n})-OT( \mu, \nu,c)\Big) \dto \inf_{\pi \in \Pi_{c}^\star(\mu,\nu)}\pi(\Gproc^c) + \sup_{f \in S\!_c(\mu,\nu)}  \Gproc^\mu(f^{cc}).
\end{align*}
  \item \label{rem:WeaklyConvergingCosts_FixedCost} In case of a \emph{fixed} cost function, i.e., when selecting $ c_n = c$, the conditions of \Cref{thm:AbstractMainResult} reduce to $\FC^{cc}$ being $\mu$-Donsker and $\FC^{c}$ being $\nu$-Donsker. Further, by \Cref{lem:RelFCandConcave} this is equivalent to $\HC_c$ and $\HC_c^c$ being Donsker for $\mu$ and $\nu$ \citep[Theorem 2.10.1 and Example 2.10.7]{van1996weak}, respectively, matching conditions (C) and (S2) of Theorem 2.1 in  \citet{Hundrieser2022Limits} which imply that $$\sqrt{\frac{nm}{n+m}}\Big(OT( \mu_n,  \nu_m,c) - OT(\mu, \nu,c)\Big)\dto \sup_{f \in S\!_c(\mu,\nu)}  \sqrt{\lambda}\Gproc^\mu(f^{cc}) +  \sqrt{1-\lambda}\Gproc^\nu(f^c).$$
  \item \label{rem:WeaklyConvergingCosts_GeneralMeasureEstimator}Our proof technique also yields distributional limits for the estimated OT value when instead of empirical measures $\mu_n$ and $\nu_m$ one considers measurable estimators $\tilde \mu_n\in \PC(\XC)$, $\tilde \nu_m\in \PC(\YC)$, respectively, that fulfill $\tilde \mu_n \dto \mu$ and $\tilde \nu_m \dto \nu$ in probability. This would mean to replace the empirical measures $\mu_n$ and $\nu_m$ in Assumptions \ref{ass:AThree} and \ref{ass:BTwo} by $\tilde \mu_n$ and $\tilde \nu_m$, respectively. In addition, instead of the scaling rate $\sqrt{nm/(n+m)}$ our proof technique theory also permits a different scaling rate $a_{n,m}$ which diverges to infinity for $n,m\rightarrow \infty$.
  \item\label{rem:WeaklyConvergingCosts_StabilityDiffInterpretation} In \Cref{prop:gateaux} we prove that the OT value is Gateaux differentiable in $(\mu, \nu,c)$ for admissible directions in $(\Delta^\mu, \Delta^\nu, \Delta^c) \in (\PC(\XC) - \mu)\times (\PC(\YC) - \nu)\times C(\XC\times \YC)$ with derivative, $$(\Delta^\mu, \Delta^\nu, \Delta^c) \mapsto \inf_{\pi\in \Pi^\star_c(\mu, \nu)}\pi(\Delta^c)+\sup_{f \in S\!_c(\mu,\nu)} \Delta^\mu(f^{cc}) +\Delta^\nu(f^c).$$
  Hence, the asymptotic distribution described in \Cref{thm:AbstractMainResult} may also be interpreted as a derivative of the OT value with respect to the triple $(\mu, \nu,c)$ evaluated at the limit process. 
  Proving \Cref{thm:AbstractMainResult} via an application of the functional delta method would amount to showing Hadamard directional differentiability of the OT value \citep{Roemisch04}. However, this turns out be a challenging issue without imposing additional assumptions on the measure and cost estimators, see \Cref{rmk:OnHadamardDifferentiability}.
  \item \label{rem:WeaklyConvergingCosts_VarianceNormalLimit}In case of a centered normal limit in \eqref{eq:LimitUnderUniqueness} the limit variance is given by \begin{align*}
    &\operatorname{Var}\big(\pi(\Gproc^c)\big)+ \lambda\operatorname{Var}_{X\sim \mu}\big(f^{cc}(X)\big)+ (1-\lambda)\operatorname{Var}_{Y\sim \mu}\big(f^{c}(Y)\big) \\
   & + 2 \sqrt{\lambda}\operatorname{Cov}\big(\pi (\Gproc^c), \Gproc^\mu(f^{cc})\big) + 2\sqrt{1-\lambda}\operatorname{Cov}\big(\pi(\Gproc^c), \Gproc^\nu(f^{c})\big),%
\end{align*}
	where we used that the random variables $X_1, \dots, X_n$ and $Y_1, \dots, Y_n$ are independent. 
In particular, the limit law degenerates if both Kantorovich potentials $(f^{cc}, f^c)$ are $(\mu, \nu)$-almost surely constant and $c_{n,m}$ converges to $c$ with a faster rate than $(nm/(n+m))^{-1/2}$, uniformly on the support of the OT plan $\pi$. For a sharp characterization of the occurrence of almost surely constant Kantorovich potentials we refer to Section 4 of \citet{Hundrieser2022Limits} where the authors showcase that for most cost functions of practical interest a.s.\ constancy typically does not occur if the underlying measures are different.
\end{enumerate}
\end{rmk}

\subsection{Distributional Limits under Extremal-Type Costs}\label{subsec:OTProcessResults}

As noted in the introduction, could the empirical infimal or supremal OT value over a fixed collection of cost functions also be analyzed using the previously described framework. However, as part of this approach, we would require the existence of a single underlying population cost function as well as weak convergence of the cost estimator.
To broaden the scope of our theory, we follow in this subsection a different route to derive limiting distributions where such conditions are not required. More precisely, we first prove a uniform distributional limit for the empirical OT process indexed over the collection of cost functions before relying on a delta method to characterize the distributional limits for the respective infimal and supremal statistics. 

For the subsequent assertions we again adhere to the sampling convention provided at the beginning of \Cref{subsec:DistributionalLimits}. The one-sample case is discussed in \Cref{rmk:CommentsOnOTProcessResults}\ref{rem:OTProcessResults_OneSample}.

\begin{thm}[OT process uniformly over compact $\Theta$]\label{thm:OTProcessCty}
Let $\Theta$ be a compact Polish space and consider a continuous map $c\colon \Theta \rightarrow C(\XC\times \YC), \theta\mapsto c_\theta$. 
Let $w\colon \RR_+\rightarrow \RR_+$ be a modulus of continuity such that $\sup_{\theta\in \Theta}|c_\theta(x,y) - c_\theta(x',y)|\leq w(d_\XC(x,x'))$ for all $x,x'\in \XC, y\in \YC$. 
Assume for $\mu \in \PC(\XC), \nu \in \PC(\YC)$ the following. 
\begin{enumerate}[label=\textbf{(Don)}]
	\item \label{ass:Don} For the function class $\FC= \FC(\sup_{\theta \in \Theta}\norm{c_\theta}_\infty, w)$ from \eqref{eq:defF} the collection $\bigcup_{\theta\in \Theta} \FC^{c_\theta c_\theta}$ is $\mu$-Donsker and $\bigcup_{\theta\in \Theta} \FC^{c_\theta}$ is $\nu$-Donsker. 
\end{enumerate}
\begin{enumerate}[label=\textbf{(KP)}]
		\item \label{ass:KPU} For any $\theta\in \Theta$, the set of Kantorovich potentials $S_{\!c_\theta}(\mu, \nu)\subseteq \FC$ for the OT problem between $\mu$ and $\nu$ and cost $c_\theta$ is unique (up to a constant shift). 
\end{enumerate}
Then, upon selecting $f_\theta\in S_{\!c_\theta}(\mu, \nu)$ for any $\theta\in\Theta$, it follows that 
	\begin{align*}
		\sqrt{\frac{nm}{n+m}}\Big( OT( \mu_n, \nu_m,c_\theta) - OT(\mu, \nu, c_\theta)\Big)_{\theta\in \Theta} \dto \left(\sqrt{\lambda}\Gproc^\mu(f_\theta^{c_\theta c_\theta})+  \sqrt{1-\lambda}\Gproc^\nu(f_\theta^{c_\theta})\right)_{\theta\in \Theta}\quad \text{ in } C(\Theta).
	\end{align*}	
	\vspace{-0.4cm}
\end{thm}

The proof of \Cref{thm:OTProcessCty} is based on Hadamard directional differentiability of the OT cost process, which follows from a general sensitivity analysis for extremal-type functions uniformly over a compact parameter space (\Cref{subsec:HadDiff}). The assertion for the empirical OT process then follows by invoking the functional delta method \citep{Roemisch04}; the proof is deferred to Section~\ref{subsubsec:OTProcessCtyProof}. 
        
From the above result, given any functional $\Phi\colon C(\Theta)\rightarrow \RR$ that is Hadamard directionally differentiable at the function $OT(\mu, \nu,c(\cdot)) \in C(\Theta)$, \Cref{thm:OTProcessCty}  yields by another application of the functional delta method, the distributional limit 
\begin{multline*}
  \sqrt{\frac{nm}{n+m}}\Big( \Phi\big(( OT( \mu_n, \nu_m,c_\theta))_{\theta\in \Theta}\big) - \Phi\big(( OT( \mu, \nu,c_\theta))_{\theta\in \Theta}\big)\Big) \\
  \dto D^{H}_{OT(\mu, \nu,c(\cdot))}\Phi\Big(\big(\sqrt{\lambda}\Gproc^\mu(f_\theta^{c_\theta c_\theta})+  \sqrt{1-\lambda}\Gproc^\nu(f_\theta^{c_\theta})\big)_{\theta\in \Theta}\Big).
\end{multline*}
Here, $D^{H}_{OT(\mu, \nu,c(\cdot))}\Phi$ denotes the directional Hadamard derivative of $\Phi$. This enables the derivation of the limit distribution for the infimal mapping  using \citet[Lemma S.4.9]{fang2019inference} (see also \citealt[Corollary 2.3]{carcamo2020directional}).  
  
 \begin{thm}[OT infimum over compact $\Theta$]\label{thm:OTProcessInf}
 	Consider the setting of \Cref{thm:OTProcessCty}. Then, %
	upon selecting $f_\theta\in S_{\!c_\theta}(\mu, \nu)$ for any $\theta\in\Theta$, it follows that 
 	$$\sqrt{\frac{nm}{n+m}}\left(\inf_{\theta\in \Theta} OT( \mu_n, \nu_m,c_\theta)- \inf_{\theta\in \Theta} OT( \mu, \nu,c_\theta)\right)\dto \inf_{\theta\in S_{\!-}(\Theta, \mu, \nu)} \sqrt{\lambda}\Gproc^\mu(f_\theta^{c_\theta c_\theta})+ \sqrt{1-\lambda}\Gproc^\nu(f_\theta^{c_\theta}),$$
 	where $S_{\!-}(\Theta, \mu, \nu) = \argmin_{\theta\in \Theta}OT(\mu, \nu,c_\theta)$ denotes the set of minimizers of $OT(\mu, \nu,c_\theta)$ over $\Theta$. 
 \end{thm}
  
In case only \ref{ass:Don} holds, one can still infer the limit law for the empirical supremal OT value. 

\begin{thm}[OT supremum over compact $\Theta$]\label{thm:OTProcessSup}
Consider the setting of \Cref{thm:OTProcessCty} and only  assume \ref{ass:Don}. Then, it follows that %
$$\sqrt{\frac{nm}{n+m}}\left(\sup_{\theta\in \Theta} OT( \mu_n, \nu_m,c_\theta)- \sup_{\theta\in \Theta} OT(\mu, \nu, c_\theta)\right)\dto \sup_{\substack{\theta\in S_{\!+}(\Theta, \mu, \nu)\\ f_\theta\in S_{\!c_\theta}(\mu, \nu)}}  \sqrt{\lambda}\Gproc^\mu(f_\theta^{c_\theta c_\theta})+  \sqrt{1-\lambda}\Gproc^\nu(f_\theta^{c_\theta}),$$
where $S_{\!+}(\Theta, \mu, \nu) = \argmax_{\theta\in \Theta}OT(\mu, \nu,c_\theta)$ denotes the set of maximizers of $OT(\mu, \nu,c_\theta)$ over $\Theta$.
\end{thm}

The proofs of Theorems \ref{thm:OTProcessInf} and \ref{thm:OTProcessSup} are documented in Sections~\ref{subsubsec:OTProcessInfProof} and \ref{subsubsec:OTProcessSupProof}, respectively. Moreover, in some contexts the compactness assumption on $\Theta$ might be too restrictive. The following result provides an extension to non-compact spaces $\Theta$ and focuses on the infimal statistic; an analogue statement also holds for the supremal statistic. Its proof is deferred to \Cref{subsubsec:ProofOfCorollaryNonCompactStuff}. %

\begin{prop}[OT infimum over general $\Theta$]\label{cor:OTInfimumNonCompactTheta}
  Let $\Theta$ be a Polish space and consider a continuous map $c\colon \Theta\rightarrow C(\XC\times \YC)$. Let $\mu\in \PC(\XC), \nu \in \PC(\YC)$ and suppose there is a compact set $K\subseteq \Theta$ such that $S_{\!-}(\Theta, \mu, \nu) \subseteq K$,  there is a sequence of minimizers $\theta_{n,m}\in S_{\!-}(\Theta, \mu_n, \nu_m)$ with $\lim_{n,m\rightarrow \infty}\pOut (\theta_{n,m} \not\in K) = 0$, and that the assumptions of \Cref{thm:OTProcessInf} hold with $\Theta$ replaced by~$K$. 
  Then, the assertion of \Cref{thm:OTProcessInf} on the empirical infimal OT value over $\Theta$ remains valid. 
\end{prop}

\begin{rmk}
\label{rmk:CommentsOnOTProcessResults}
A few comments are in order concerning the weak limits for the empirical OT cost process as well as the respective infimal and supremal statistic. 
\begin{enumerate}[label=$(\roman*)$]
  \item\label{rem:OTProcessResults_EquicontinutiyCondition} 
  In the setting of \Cref{thm:OTProcessCty} the parameter space $\Theta$ is compact and $c\colon \Theta \rightarrow C(\XC\times \YC)$ is continuous, therefore the range $c(\Theta)$ is also compact in $C(\XC\times \YC)$. In particular, by the Theorem of Arzel\`a-Ascoli, we conclude that $\sup_{\theta \in \Theta}\norm{c_\theta}_\infty<\infty$ and there exists a suitable modulus of continuity for all cost functions uniformly on $\Theta$. 
  \item \label{rem:OTProcessResults_AssumtionDiscussion} Both assumptions of \Cref{thm:OTProcessCty} and sufficient conditions are discussed in Sections \ref{subsec:AssumDon} and \ref{subsec:AssumKP}. Assumption \ref{ass:Don} appears natural in order to control the empirical OT process uniformly over $\Theta$, whereas \ref{ass:KPU} is to ensure that the limit process is supported in $C(\Theta)$ and stays tight. Our proof technique suggests that \ref{ass:KPU} can be slightly lifted, but not much. For instance, one could demand that Kantorovich potentials $S_{\!c_\theta}(\mu, \nu)$ which attain the supremum in the derivative can be approximated by Kantorovich potentials $S_{\!c_{\theta'}}(\mu, \nu)$  for $\theta'$ in the immediate vicinity of $\theta$, as required in \Cref{lem:LSC_con}(i). In particular, if $\Theta \coloneqq \{\theta_1, \dots, \theta_K\}$ is a finite set equipped with discrete topology, then \ref{ass:KPU} can be omitted.
  \item \label{rem:OTProcessResults_OneSample} The results also extend to the one-sample setting, i.e., when  $\mu$ is estimated by $\mu_n$ and $\nu$ is assumed to be known. 
  For the one-sample version of \Cref{thm:OTProcessCty} it suffices to assume in \ref{ass:Don} that the function class $\cup_{\theta\in \Theta} \FC^{c_\theta c_\theta}$ is $\mu$-Donsker in conjunction with \ref{ass:KPU}. Upon selecting $f_\theta\in S_{\!c_\theta}(\mu, \nu)$ for any $\theta\in \Theta$, the  limit distribution is then given for $n \rightarrow \infty$ by 
  \begin{align*}
    \sqrt{n}\Big(OT(\mu_n, \nu,c_\theta)- OT(\mu, \nu,c_\theta)\Big)_{\theta\in \Theta} \dto \left( \Gproc^\mu(f^{c_\theta c_\theta}_\theta) \right)_{\theta\in \Theta} \quad \text{ in } C(\Theta).
  \end{align*}
    Under identical assumptions, the one-sample analogue of \Cref{thm:OTProcessInf} is available. For the validity of the one-sample result in \Cref{thm:OTProcessSup} it suffices that $\cup_{\theta\in \Theta} \FC^{c_\theta c_\theta}$ is $\mu$-Donsker.
  \item \label{rem:NormalityDegeneracy}\label{rem:OTProcessResults_NormalityDegeneracy} The obtained weak limits highlight an intimate dependency of limit distributions to the collection of Kantorovich potentials. 
  In \Cref{thm:OTProcessCty} the limit process is centered Gaussian due to Assumption \ref{ass:KPU}. %
	For fixed $\theta\in \Theta$ the limiting random variables degenerates to a Dirac measure at zero if the respective Kantorovich potentials are $(\mu, \nu)$-almost surely constant. Moreover, the limit distribution in \Cref{thm:OTProcessInf} is also centered normal if Kantorovich potentials $(f^{c_\theta c_\theta}, f^{c_\theta})$ for $\mu, \nu$ and $c_\theta$ coincide (up to a constant shift) on $\supp(\mu)\times \supp(\nu)$ for any $\theta\in S_{\!-}(\Theta, \mu,\nu)$. Under analogous assumptions for $\theta\in S_{\!+}(\Theta, \mu,\nu)$ the limit distribution in \Cref{thm:OTProcessSup} is centered normal. In particular, assuming \ref{ass:KPU}, this condition is fulfilled if $S_{\!-}(\Theta, \mu, \nu)$ or $S_{\!+}(\Theta, \mu,\nu)$ consist of a singleton. The resulting limit distributions degenerate if Kantorovich potentials are $(\mu, \nu)$-almost surely constant. A sharp characterization of almost surely constant potentials is detailed in Section 4 of \citet{Hundrieser2022Limits}. 
\end{enumerate}
\end{rmk}

\subsection{Bootstrap Principle for Optimal Transport Costs}\label{subsec:BootstrapPrinciple}

Since the limit distributions in Theorems \ref{thm:AbstractMainResult}, \ref{thm:OTProcessInf} and \ref{thm:OTProcessSup} involve the set of Kantorovich potentials (and OT plans), under non-unique optimizers there is little hope for an explicit, closed-form description of the quantiles for these distributions,  which is required for further practical purposes. To circumvent this issue we suggest the use of a $k$-out-of-$n$ bootstrap procedure with $k = \oh(n)$ whose consistency is shown in this subsection. 

For simplicity, we state the subsequent results for equal sample sizes, i.e., $n=m$ as well as bootstrap samples of equal size $k=\oh(n)$. Under differing sample sizes $n\neq m$ one would select  bootstrap samples of size $k=\oh(n),l=\oh(m)$ such that $l/(l+k)\approx m/(n+m)$. 
Below, we always consider the same bootstrap approach that we now introduce. For the two sequences of i.i.d.\ random variables $\{X_i\}_{i=1}^n\sim\mu^{\otimes n}$,  
$\{Y_i\}_{i=1}^n\sim\nu^{\otimes n}$, with respective empirical measures $ \mu_n,  \nu_n$, consider another sequence of  i.i.d.\ bootstrap random variables $\{\Xib\}_{i=1}^k\sim\mu_n^{\otimes k}$, 
$\{\Yib\}_{i=1}^k\sim\nu_n^{\otimes k}$ and define the bootstrap empirical measures $\muboot \coloneqq \frac{1}{k}\sum_{i = 1}^{k} \delta_{\Xib}$ and $\nuboot \coloneqq \frac{1}{k}\sum_{i = 1}^{k} \delta_{\Yib}$. Moreover, we write in the subsequent statement $c_{n}$ for the cost estimator and  $\cboot$ for the bootstrap cost estimator. 

\begin{prop}[Bootstrap for OT under weakly converging costs]
\label{thm: BootConst}
In the setting of \Cref{thm:AbstractMainResult}, assume $\textit{\ref{ass:AThree}}$ and either $\textit{\ref{ass:BOne}}$ or $\textit{\ref{ass:BTwo}}$. Let $\cboot\in C(\XC\times \YC)$ be the bootstrap cost estimator such that $\cboot(x,y)$ is measurable for all $(x,y) \in \XC\times \YC$.
Further, assume the following. 
\begin{enumerate}[label={\AThreeStar}] 
   \item \label{ass:AThreeStar} The bootstrap empirical processes are conditionally on $X_1, \dots, X_n, Y_1, \dots Y_n$ consistent in the space $\ell^\infty(\FC^{cc})\times \ell^\infty(\FC^c)\times C(\X\times \Y)$ for $n,k \rightarrow \infty$ with $k=\oh(n)$, i.e., 
   \begin{align*}
       d_{BL}\left( \Law\left(\sqrt{k} \begin{pmatrix}
\muboot - \mu_n\\
\nuboot - \nu_n\\
 \cboot - c_n
\end{pmatrix} \middle| X_1, \dots, X_n, Y_1, \dots Y_n\right),\Law\left( \sqrt{n} \begin{pmatrix}
\mu_n - \mu\\
\nu_n - \nu\\
 c_n - c
\end{pmatrix}\right) \right)\ptoOut 0. 
   \end{align*}
\end{enumerate}
In case of setting \textit{\ref{ass:BTwo}} additionally assume the following. 
\begin{enumerate}[label={\BTwoStar}] 
    \item \label{ass:BTwoStar} The unconditional bootstrap empirical processes $\Gproc_{n,k}^{\mu}\coloneqq  \sqrt{k}(\muboot- \mu)$ and $\Gproc_{n,k}^{\nu} \coloneqq \sqrt{k}(\nuboot -~\nu)$ fulfill $\sup_{f\in \FC}\Gproc_{n,k}^{\mu}(f^{ \cboot \cboot} - f^{cc})\ptoOut 0$,  $\sup_{f\in \FC}\Gproc_{n,k}^{\nu}(f^{ \cboot} - f^{c})\ptoOut 0$ for $n,k \rightarrow \infty$,  $k=\oh(n)$. 
\end{enumerate}
Then, it follows for $n,k\rightarrow\infty$ with $k = \oh(n)$ that \begin{align*}
     d_{BL}\Big( \Law\left(\sqrt{k}\left( OT( \muboot, \nuboot, \cboot) - OT(\mu_n, \nu_n,c_n)\right)\middle|X_1, \dots, X_n, Y_1, \dots, Y_n\right), \\ \Law\left(\sqrt{n}\left( OT( \mu_{n}, \nu_{n},c_{n}) - OT(\mu, \nu,c)\right)\right) \Big) \ptoStar 0. 
\end{align*}
\end{prop}

Despite not relying on the functional delta method for the derivation of the limit distribution of the empirical OT value under weakly converging costs, we obtain a similar bootstrap principle as \citet[Proposition 2]{dumbgen1993nondifferentiable} by employing an equivalent formulation for bootstrap consistency \citep{bucher2019note} in conjunction with the use of a Skorokhod representation. The full proof is provided in \Cref{subsubsec:BootConstProof}.

\begin{rmk}
  When employing the functional delta method, the $n$-out-of-$n$ bootstrap is not consistent if the Hadamard directional derivative is not linear \citep[Proposition 1]{dumbgen1993nondifferentiable}. Although \Cref{thm: BootConst} does not build on a differentiability result, we show in \Cref{subsec:stability} that the OT functional is Gateaux directional differentiability with a derivative that is non-linear if primal or dual optimizers are non-unique.  Since Gateaux directional differentiability is implied by Hadamard directional differentiability, this suggests in the regime of non-unique optimizers the inconsistency of the naive $n$-out-of-$n$ bootstrap for the empirical OT cost under weakly converging costs. 
\end{rmk}

 Verification of the bootstrap consistency in the settings of Theorems \ref{thm:OTProcessCty}--\ref{thm:OTProcessSup} is straightforward. It is a direct consequence of consistency of the  $k$-out-of-$n$ bootstrap empirical processes with $k=\oh(n)$ \citep[Theorem 3.6.13]{van1996weak} and the functional delta method for the bootstrap \citep[Proposition 2]{dumbgen1993nondifferentiable}. Hence, we omit the proof of the following proposition.
\begin{prop}[Bootstrap for OT process, supremum, and infimum]
	Let	$\XC, \YC$ be compact Polish spaces and $\Theta$ a compact topological space. Consider a continuous map $c\colon \Theta \rightarrow C(\XC\times \YC), \theta\mapsto c_\theta$, let $\mu\in \PC(\XC), \nu\in \PC(\YC)$ and assume \ref{ass:Don}. 
\begin{enumerate}[label=$(\roman*)$]
	\item (OT process in $C(\Theta)$) Then, under \ref{ass:KPU}, it follows for $n,k \rightarrow \infty$ with $k\leq n$~that 
	\begin{align*}
     d_{BL}\bigg( \Law\left(\sqrt{k}\left( OT( \muboot, \nuboot, c_{\theta}) - OT(\mu_n, \nu_n,c_{\theta})\right)_{\theta\in \Theta}\middle|X_1, \dots, X_n, Y_1, \dots, Y_n\right), \\ \Law\left(\sqrt{n}\left( OT( \mu_{n}, \nu_{n},c_{\theta}) - OT(\mu, \nu,c_{\theta})\right)\right)_{\theta\in \Theta}\bigg)\ptoStar 0. 
\end{align*}
	\item (OT infimum over $\Theta$)  Then, under \ref{ass:KPU}, it follows for $n,k\rightarrow\infty$ with $k\leq \oh(n)$~that  
	\begin{align*}
     d_{BL}\bigg( \Law\left(\sqrt{k}\left( \inf_{\theta\in \Theta}OT( \muboot, \nuboot,c_\theta) - \inf_{\theta\in \Theta}OT(\mu_n, \nu_n,c_\theta)\right)\middle|X_1, \dots, X_n, Y_1, \dots, Y_n\right), \\ \Law\left(\sqrt{n}\left(\inf_{\theta\in \Theta} OT( \mu_{n}, \nu_{n}, { c_{\theta}}) - \inf_{\theta\in \Theta}OT(\mu, \nu, { c_{\theta}})\right)\right)\bigg) \ptoStar 0.
	\end{align*}
	\item (OT supremum over $\Theta$)  Then, it follows for $n,k \rightarrow \infty$ with $k=\oh(n)$ that 
	\begin{align*}
     d_{BL}\bigg( \Law\left(\sqrt{k}\left( \sup_{\theta\in \Theta}OT( \muboot, \nuboot,c_\theta) - \sup_{\theta\in \Theta}OT(\mu_n, \nu_n,c_\theta)\right)\middle|X_1, \dots, X_n, Y_1, \dots, Y_n\right), \\ \Law\left(\sqrt{n}\left(\sup_{\theta\in \Theta} OT( \mu_{n}, \nu_{n},c_\theta) - \sup_{\theta\in \Theta}OT(\mu, \nu,c_\theta)\right)\right)\bigg)  \ptoStar 0.
	\end{align*}
	\end{enumerate}
\end{prop}
 Notably, we also obtain consistency of the $n$-out-of-$n$ bootstrap for setting $(i)$ since \ref{ass:KPU} implies linearity of the Hadamard directional derivative.



\section{Discussion of Assumptions}\label{subsec:Assumptions}
In this section we discuss the assumptions on the distributional limits and the bootstrap consistency. We also provide sufficient conditions for their validity. %
All the proofs are deferred to \Cref{app:SufficientCriteriaProofs}. 

\subsection{Assumptions \ref{ass:AThree} and \ref{ass:AThreeStar}: Joint Weak Convergence}\label{subsec:AssumJW}

For the empirical OT value under estimated costs we demand in  \ref{ass:AThree} and \ref{ass:AThreeStar} weak convergence of the empirical processes in $\linfty(\FC^{cc})$ and $\linfty(\FC^{c})$, where $\FC = \FC(2\norm{c}_\infty+1,2w)$ is selected as in \Cref{thm:AbstractMainResult}. This requires $\FC^{cc}$ and $\FC^{c}$ to be $\mu$- and $\nu$-Donsker, respectively.  Moreover, we demand weak convergence of the estimated cost function in $C(\XC\times \YC)$ to ensure that any sequence of OT plans for $ \mu_n,  \nu_m$ and $ c_{n,m}$ tends towards an OT plan in $\Pi^\star_c(\mu, \nu)$. Finally, we stress the necessity of \emph{joint} weak convergence in \ref{ass:AThree} and \ref{ass:AThreeStar} as the limit distribution is determined by the random variable $(\Gproc^\mu,\Gproc^\nu,\Gproc^c)$ and thus characterized by their dependency. 

Even though apparently unavoidable, these conditions are somewhat restrictive and delimit the theory to low dimensional settings. This is to be expected as estimation of the OT value (under population costs) suffers from the curse of dimensionality (\cite{Manole21}), leading to slow convergence rates when both population measures $\mu, \nu$ exhibit high-dimensional support. However, in view of the recently discovered \emph{lower complexity adaptation} principle \citep{hundrieser2022empirical}, it suffices that one measure, $\mu$ or $\nu$, is supported on a low dimensional space. The following proposition provides bounds on the covering numbers (see the notation section for a definition) of $\FC^{c}$ and $\FC^{cc}$ under uniform norm which leads to a universal Donsker property for both function classes.  %

\begin{prop}[Universal Donsker property]
\label{prop:DonskerProperty}
Let $c\in C(\XC\times \YC)$ be a continuous cost function with $\norm{c}_\infty\leq 1$. Assume one of the three settings. 
\begin{enumerate}[label=(\roman*)]
    \item $\XC=\{x_1, \dots, x_N\}$ is a finite space (and no additional assumption on $c$). 
    \item There exists a pseudo metric\footnote{A non-negative function $d\colon \MC\times \MC\rightarrow \RR_+$ on a set $\MC$ is a pseudo-metric if the three conditions $d(x,x) = 0$, $d(x,y) = d(y,x)$ and $d(x,y) \leq d(x,z)+d(z,y)$ are fulfilled for any $x,y,z\in \MC$.} $\tilde d_{\XC}$ on $\XC$ such that $\NC(\eps, \XC,\tilde  d_{\XC}) \lesssim \eps^{-\beta}$ for $\eps>0$ sufficiently small and some $\beta\in (0,2)$ and $c(\cdot,y)$ is $1$-Lipschitz under $\tilde d_{\XC}$ for all $y\in \YC$.  
    \item $\XC = \bigcup_{i = 1}^{I}\zeta_i(\UC_i)$ for $I\in \NN$ compact, convex subsets $\UC_i\subseteq \reals^{d_i}$, $d_i \leq 3$ with non-empty interior and maps $\zeta_i \colon \UC_i \rightarrow \XC$ such that for each $i\in \{1, \dots, I\}$ the function $c(\zeta_i(\cdot),y)$ is $(\gamma_i, 1)$-H\"older\footnote{\label{foot:Holder} A function $f\colon \UC\rightarrow \RR$ on a convex set $\UC\subseteq \RR^d$ with non-empty interior is $(\gamma,\Lambda)$-H\"older with modulus $\Lambda\geq 0$ and $\gamma\in (0,1]$ if $\norm{f}_\infty<\Lambda$ and $|f(x)-f(y)|\leq \Lambda\norm{x-y}^\gamma$ for any $x,y\in \UC$. Further, $f$ is called $(\gamma,\Lambda)$-H\"older for $\gamma\in (1,2]$ if every partial derivative of $f$ is $(\gamma - 1, \Lambda)$-H\"older. If $\UC$ is not open, we assume the existence of an extension $\tilde f$ of $f$ onto an open convex set containing $\UC$ such that $\tilde f$ is $(\gamma, \Lambda)$-H\"older thereon, cf. \citet{hundrieser2022empirical}.} on $\UC_i$  for some $\gamma_i \in (d_i/2,2]$ for all $y \in \YC$.
    \end{enumerate}
Let $B\geq 0$ and consider a modulus of continuity $w\colon \RR_+\to \RR_+$ with respect to a metric $d_\XC$ on $\XC$. Then, for each setting there exists some $\alpha < 2$ such that for $\eps>0$ sufficiently small, \begin{align*}
  \log\NC(\eps, \FC^{c}, \norm{\cdot}_\infty) =\log\NC(\eps, \FC^{c c}, \norm{\cdot}_\infty) \lesssim \eps^{-\alpha} \quad \text{ for } \FC = \FC(B, w),\end{align*}
where the hidden constant depends for (i) on $N$, for (ii) on $\NC(\eps, \XC, \tilde d_{\XC})$, and for (iii) on $(\zeta_i, \UC_i)_{i = 1}^{I}$. In particular, the function classes $\FC^{c}$ and $\FC^{cc}$ are universal Donsker. 
\end{prop}
The bounds for the covering numbers stated in the above proposition are essential for the weak convergence of the empirical processes $\sqrt{n}(\mu_n-\mu)$ and $\sqrt{m}(\nu_m-\nu)$ and represent an important tool for verifying \ref{ass:AThree}. In order to clarify the assumptions of \Cref{prop:DonskerProperty}, we showcase them in a simple example. We additionally refer to \citet[Section 3]{hundrieser2022empirical} and \citet[Section 5]{Hundrieser2022Limits} for more illustrative examples.  
\begin{ex}
Suppose that $\XC$ and $\YC$ are compact subsets of $\RR^3$ and let $c:\RR{^3\times\RR{^3}}\to\RR$ be twice continuously differentiable. By 
enlarging $\XC$ to a compact, convex set and since $c$ can be rescaled such that $c(\cdot,y)$ is $(2,1)$-H\"older on $\XC$, \Cref{prop:DonskerProperty} is applicable in this setting.
\end{ex}  

To state sufficient conditions for \ref{ass:AThree} and \ref{ass:AThreeStar} we assume that the population cost as well as the empirical and bootstrap estimators are determined by %
the underlying measures via a Hadamard directionally differentiable functional. For simplicity, we consider in the subsequent proposition random variables $\{X_i\}_{i = 1}^{n}\sim \mu^{\otimes n}, \{Y_i\}_{i =1}^{n}\sim \nu^{\otimes n}$ of identical sample size~$n$ with empirical measures $\mu_n,\nu_n$, and bootstrap samples $\{\Xib\}_{i=1}^{k}\sim \mu_n^{\otimes k}$, $\{\Yib\}_{i=1}^{k}\sim \nu_n^{\otimes k}$ of size $k = k(n) = \oh(n)$ with corresponding bootstrap empirical measures $\muboot, \nuboot$. 

\begin{prop}[Joint weak convergence]
\label{prop: A3StarEx_NonPar}
Let $\FC_\XC$,  $\FC_\YC$ be bounded function classes on $\XC$ and $\YC$, respectively, and assume there is a functional $\Phi_c\colon \PC(\XC)\times \PC(\YC) \subseteq \linfty(\FC_\XC)\times \linfty(\FC_\YC)\rightarrow C(\XC\times \YC)$ such that, for all $n,k\in \NN$, \begin{align*}
    c = \Phi_c(\mu, \nu), \quad  c_n = \Phi_c( \mu_n,  \nu_n), \quad \text{ and } \quad   \cboot = \Phi_c( \muboot,  \nuboot). 
\end{align*}
If $\Phi_c$ is Hadamard directionally differentiable at $(\mu, \nu)$ tangentially to $\PC(\XC)\times \PC(\YC)$, and if $\FC_\XC\cup\FC^{cc}$ is $\mu$-Donsker while $\FC_\YC\cup\FC^{c}$ is $\nu$-Donsker, then both \ref{ass:AThree} and \ref{ass:AThreeStar} are fulfilled. 
\end{prop}

\begin{rmk} 
  We like to point out that if the functional $\Phi_c$ is additionally continuous with respect to the topology induced by weak convergence on $\PC(\XC)\times \PC(\YC)$, it follows  that $c_{n}(x,y)$ and $\cboot(x,y)$ are measurable for each $(x,y) \in \XC\times \YC$ and, due to compactness of $\XC$ and $\YC$, measurable in $C(\XC\times \YC)$. 
\end{rmk}

\subsection{Assumption \ref{ass:BOne}: Uniqueness of Optimal Transport Plans}\label{subsec:AssumOTUniquePlan}
The subject of uniqueness of OT plans between probability measures and a given cost function is of long-standing interest and has been addressed by various authors. General conditions for continuous settings were stated in  \citet{gangbo1996geometry} and \citet{levin1999abstract}, building on previous works. The subject has since been covered in depth in Chapters 9 and 10 of the reference textbook by \citet{villani2008optimal}; further advances have been made since.

To guarantee the uniqueness of the OT plan, many works resort to the so-called \emph{Twist condition} which demands for differentiable costs the injectivity of the map $y \rightarrow\nabla_x  c(x,y)$ for all $x\in \XC$. The following proposition formalizes a uniqueness criterion based on this condition and should fulfill the reader's needs for many practical applications.  The result can be deduced from Theorem 10.28 and  Remark 10.33 in \citet{villani2008optimal}.

\begin{prop}
Assume that $\XC, \YC$ are compact Polish spaces where $\XC\subseteq \reals^d$ is a Euclidean subset with non-empty interior and $\mu$ is absolutely continuous  with respect to the Lebesgue measure. Further, assume that $c$ is locally Lipschitz on $ \XC\times \YC$, that  $c(\cdot, y)$ is differentiable on $\interior{\XC}$ for each $y \in \YC$ and that $y\mapsto\nabla_x c(x,y)$ is injective for each $x\in \XC$.  Then, the OT plan $\pi^\star\in \Pi(\mu, \nu)$  is unique. 
\end{prop}

Even though in certain cases weaker conditions can yield uniqueness \citep{ahmad2011optimal}, these more general conditions are typically considerably more difficult to verify. %
Nevertheless, unless the cost function exhibits some kind of symmetry %
or is constant in some region, uniqueness of OT plans is often to be expected. Indeed,  for fixed measures there is a residual set of cost functions such that for any such costs the OT plan is unique \citep{mccann2016intrinsic}. 

In finite discrete settings, i.e., when both underlying measures are supported on finitely many points, results on the uniqueness of OT plans are mostly based on the theory of finite-dimensional linear programs and we refer to \citet[Section~6]{klatt2022limit} for a detailed account. Among others, they provide sufficient conditions for uniqueness of OT plans which solely depend on the cost function and the support points but are independent of the weights of the measures. For Euclidean based costs their condition is fulfilled for Lebesgue-almost every arrangement of support points of $\mu$ and $\nu$, it is however violated if the support points obey some regular or repetitive pattern.

\subsection{Assumptions \ref{ass:BTwo} and \ref{ass:BTwoStar}: Control of Supremum over Empirical Process}\label{subsec:AssumSup}

Our assumptions on the suprema of the empirical processes ensure that the fluctuation on the set of feasible dual potentials caused by estimation of the cost function is asymptotically negligible. Let us also point out that the suprema in \ref{ass:BTwo} and \ref{ass:BTwoStar} are (Borel) measurable by Lemma \ref{lem:measurability} and \ref{lem:JointWeakConvergence}. This implies that the convergence in outer probability occurs, in fact, in probability. Indeed, following along the proof of  \Cref{lem:JointWeakConvergence} and due to measurability of $c_{n,m}$ it follows for fixed $f\in \FC= \FC(2\norm{c}_\infty+1, 2w)$ that both maps $\omega\mapsto \Gproc^\mu_n(f^{cc})$ and $\omega\mapsto \Gproc^\mu_n(f^{ c_{n,m}  c_{n,m}})$ are measurable. In conjunction with $\Gproc^\mu_n((\cdot)^{cc}), \Gproc^\mu_n((\cdot)^{ c_{n,m}  c_{n,m}})\in C_u(\FC, \norm{\cdot}_\infty)$ by \Cref{lem:measurability} and compactness of $(\FC, \norm{\cdot}_\infty)$ the measurability of the $\Gproc^\mu_n((\cdot)^{cc} - (\cdot)^{ c_{n,m}  c_{n,m}})$ as well as its supremum follow.

In the following, we derive sufficient conditions for the validity of  Assumption \ref{ass:BTwo} (as well as Assumption \ref{ass:BTwoStar}). Based on  empirical process theory, in order to suitably control the suprema
\[\sup\nolimits_{f\in \FC}\Gproc_n^{\mu}(f^{c_{n,m}, c_{n,m}} - f^{cc}) \quad \text{ and } \quad \sup\nolimits_{f\in \FC}\Gproc_n^{\nu}(f^{c_{n,m}} - f^{c})\]
a canonical route would be to impose metric entropy bounds for $\FC^{c_{n,m}, c_{n,m}}\cup \FC^{c c}$ and $\FC^{c_{n,m}}\cup \FC^{c}$. 
Such bounds, however, would impose certain regularity requirements on the cost estimator $c_{n,m}$. Hence, in order not to narrow our scope concerning cost estimators, we employ the same ideas as in \Cref{subsec:DistributionalLimits} and approximate the cost estimator $c_{n,m}$ by a more regular cost estimator $\tilde c_{n,m}$. 
The subsequent result formalizes these considerations for our context. Its proof relies on techniques developed by \citet{wellner2007empirical} for empirical processes indexed over estimated function classes. 

\begin{prop}\label{prop:AbstractB2_CoveringNumbers}
	Let $\XC, \YC$ be compact Polish spaces and consider a continuous cost function $c$.
	\begin{enumerate}
		\item[(i)] Assume \ref{ass:AThree} for random elements $c_{n,m}\in C(\XC\times \YC)$ and take random elements $\tilde c_{n,m}\in C(\XC\times \YC)$ with $\sqrt{nm/(n+m)}\|c_{n,m} - \tilde c_{n,m}\|_\infty \smash{\pto} 0$ for $n,m=m(n)\rightarrow \infty$ and $m/(n+m)\rightarrow \lambda\in (0,1)$ such that for $\eps>0$ sufficiently small,
	\begin{align}\label{eq:EntropyCondition}
    \log\NC(\eps, \FC^{c c}, \norm{\cdot}_\infty) + \sup_{n\in \NN}\log\NC(\eps, \FC^{\tilde c_{n,m} \tilde c_{n,m}}, \norm{\cdot}_\infty) 
  \lesssim \eps^{-\alpha} \quad \text{ with } \alpha<2.
\end{align}
		Then Assumption \ref{ass:BTwo} is fulfilled. 
		\item[(ii)] Assume  \ref{ass:AThree} and \ref{ass:AThreeStar} for random elements $\cboot\in C(\XC\times \YC)$ and let $\ctildeboot\in C(\XC\times \YC)$ be random elements with $\sqrt{k}\|\cboot - \ctildeboot\|_\infty \smash{\pto} 0$ for $n, k = k(n) \rightarrow \infty$ and $k = \oh(n)$ such that for $\eps>0$ sufficiently small, 	\begin{align}\label{eq:EntropyConditionBS}
      \log\NC(\eps, \FC^{c c}, \norm{\cdot}_\infty) + \sup_{n\in \NN}\log\NC(\eps, \FC^{\tilde c_{n,k}^b \tilde c_{n,k}^b}, \norm{\cdot}_\infty)
      \lesssim \eps^{-\alpha} \quad \text{ with } \alpha<2.
\end{align}
	 Then Assumption \ref{ass:BTwoStar} is fulfilled. 
	\end{enumerate}
\end{prop}

As a straightforward corollary of \Cref{prop:AbstractB2_CoveringNumbers} we find that \ref{ass:BTwo} and \ref{ass:BTwoStar} are fulfilled if the cost estimators $c_{n,m}$ and $\cboot$ fulfill certain deterministic regularity conditions once $n,m,k$ are sufficiently large. In the large sample regime we then choose $\tilde c_{n,m}\coloneqq c_{n,m}$ and $\ctildeboot\coloneqq\cboot$. 
\begin{corollary}\label{cor:SufficientConditionsSupDeterministic}
Let $\XC, \YC$ be compact Polish spaces, consider a continuous cost function $c$. 
Assume \ref{ass:AThree} for  $c_{n}$ (and \ref{ass:AThreeStar} for  $\cboot$) and that  $c$,  $c_{n}$ (and $\cboot$) each fulfill one of the three conditions of \Cref{prop:DonskerProperty} for $n\geq N$, $k\geq K$ with random variables $N,K\in \NN$ .
Then, \ref{ass:BTwo} (and \ref{ass:BTwoStar}) hold.	
\end{corollary}
Hence, if the population cost $c$ and the estimators $c_n, c_{n,k}$ are determined by some parameter $\theta\in \Theta$ and estimators $\theta_n, \theta_{n,k}$, such that the regularity properties of \Cref{prop:DonskerProperty} are met uniformly in an open neighborhood of $\theta$ and if the estimators are consistent, then \Cref{cor:SufficientConditionsSupDeterministic} asserts the validity of Assumptions \ref{ass:BTwo} and \ref{ass:BTwoStar}. 

  Moreover, under mild additional assumptions on the space $\XC$ and the cost function $c$, %
  we can state a functional $\Psi\colon C(\XC\times \YC)\rightarrow C(\XC\times \YC)$ such that $\tilde  c_n\coloneqq \Psi(c_n)$ fulfills the entropy bound \eqref{eq:EntropyCondition} while satisfying $\sqrt{n}\norm{\tilde  c_{n,m} - c_{n,m}}_\infty\!{\pto}\, 0$ for $n \rightarrow \infty$. 
  We call such a functional $\Psi$ a \emph{regularity elevation functional} since it lifts the degree of regularity of the cost estimator. Details on regularity elevations are deferred to \Cref{subsec:RegEl}.
  
\begin{corollary}\label{cor:SufficientConditionsSupConditions}
    Let $\XC, \YC$ be compact Polish spaces and consider a continuous cost.
    Assume \ref{ass:AThree} (and \ref{ass:AThreeStar}).
    Suppose that $c$ fulfills one of the three conditions of \Cref{prop:DonskerProperty}. Under (ii) or (iii) further assume the subsequent condition (ii)' or (iii)', respectively. 
    \begin{enumerate}
      \item[(ii)']  The weak limit $\Gproc^c$ is almost surely continuous with respect to $(\XC,\tilde d_\XC)\times \YC$.
      \item[(iii)'] For each $i\in \{1, \dots I\}$ the set $\UC_i\subseteq \RR^{d_i}$ is convex and compact, the map $\zeta_i\colon \UC_i\rightarrow \zeta_i(\UC_i)$ is a homeomorphism,  and $c_i\coloneqq c(\zeta_i(\cdot), \cdot)\colon \UC_i\times \YC\rightarrow \RR$ is continuously differentiable in $u$ on $\UC_i\times \YC$, i.e., the derivative $\nabla_u c_i\colon \interior{\UC_i}\times \YC\rightarrow \RR^d$ can be continuously extended to $\UC_i\times \YC$.  Further, there exists a continuous partition of unity\footnote{A collection $\{\eta_i\}_{i = 1}^I$ is a continuous partition of unity  if  $\eta_i \in C(\XC)$, $\eta_i \geq 0$ for each $i$ and  $\sum_{i =1}^{I} \eta_i\equiv 1$ on $\XC$.} $\{\eta_i\}_{i = 1}^I$ on $\XC$ with $\supp(\eta_i)\subseteq \zeta_i(\UC_i)$.  
    \end{enumerate}
    Then, Assumption \ref{ass:BTwo} (and \ref{ass:BTwoStar}) is fulfilled. 
\end{corollary}

\subsection{Assumption \ref{ass:Don}: Donsker Property Uniformly over $\Theta$}\label{subsec:AssumDon}

For the distributional limits  by \citet{Hundrieser2022Limits} on the empirical OT value under a fixed cost function $c$, the authors effectively assume that the function classes $\FC^{cc}$ and $\FC^c$ are $\mu$- and $\nu$-Donsker, respectively (\Cref{rmk:CommentsOTWeaklyConvergingCosts}\ref{rem:WeaklyConvergingCosts_FixedCost}). Hence, for the uniform convergence result from \Cref{thm:OTProcessCty} it is natural that we demand the $\mu$- and $\nu$-Donsker property for the unions $\cup_{\theta\in \Theta}\FC^{c_\theta c_\theta}$ and $\cup_{\theta\in \Theta}\FC^{c_\theta}$ for $\FC=\FC(\sup_{\theta \in \Theta}\norm{c}_\infty, w)$. The validity of this condition can be ensured under assumptions on the complexity of the domain $\XC$ in conjunction with regularity conditions imposed on the cost function. 

\begin{prop}[Universal Donsker property over $\Theta$]
\label{prop:DonskerProperty_ParameterClass}
Let $\XC, \YC$ be compact Polish spaces and let $(\Theta,d_\Theta)$ be a metric space such that $\log\NC(\eps, \Theta, d_\Theta)\lesssim \eps^{-\alpha}$ for $\alpha<2$. Suppose that $c\colon (\Theta, d_\Theta)\rightarrow C(\XC\times \YC), \theta \mapsto c_\theta$ is $1$-Lipschitz and assume $\sup_{\theta\in \Theta}\norm{c_\theta}_\infty \leq 1$. Consider one of the three settings. 
\begin{enumerate}[label=(\roman*)]
    \item $\XC=\{x_1, \dots, x_N\}$ is a finite space (and no additional assumption on $c$). 
    \item For any $\theta\in \Theta$ there exists a pseudo metric $\tilde d_{\theta,\XC}$ on $\XC$ such that $\sup_{\theta\in \Theta}\NC(\eps, \XC, \tilde d_{\theta,\XC}) \lesssim \eps^{-\beta}$ for $\beta<2$ and $c_\theta(\cdot,y)$ is $1$-Lipschitz under $\tilde d_{\theta,\XC}$ for all $y\in \YC$. 
    \item $\XC = \bigcup_{i = 1}^{I}\zeta_i(\UC_i)$ for $I\in \NN$ compact, convex subsets $\UC_i\subseteq \reals^{d_i}$, $d_i \leq 3$ with non-empty interior and maps $\zeta_i \colon \UC_i \rightarrow \XC$ so that for each $i\in \{1, \dots, I\}$ the function $c_\theta(\zeta_i(\cdot),y)$ is $(\gamma_i, 1)$-H\"older on $\UC_i$ (recall footnote \ref{foot:Holder}) for some $\gamma_i \in (d_i/2,2]$ for all $y \in \YC$, $\theta \in \Theta$. 
\end{enumerate}
Then, for each setting, there exists some $\alpha<2$ such that \begin{align*}
  \log\NC(\eps, \cup_{\theta\in \Theta}\FC^{c_\theta c_\theta}, \norm{\cdot}_\infty) \lesssim \eps^{-\alpha} \quad \text{ and } \quad \log\NC(\eps, \cup_{\theta\in \Theta}\FC^{c_\theta}, \norm{\cdot}_\infty) \lesssim \eps^{-\alpha}.\end{align*}
In particular, $\cup_{\theta\in \Theta}\FC^{c_\theta c_\theta}$, $\cup_{\theta\in \Theta}\FC^{c_\theta}$ are universal Donsker, and Assumption \ref{ass:Don} is fulfilled.
\end{prop}
The proof of  \Cref{prop:DonskerProperty_ParameterClass} is a simple consequence of \Cref{prop:DonskerProperty} in combination with the subsequent lemma whose proof is deferred to \Cref{subsubsec:ParameterBoundMetricEntropyProof}. 

\begin{lem}\label{lem:ParameterBoundMetricEntropy}
	Let $\XC, \YC$ be compact Polish spaces and let $(\Theta,d_\Theta)$ be a metric space. Suppose $c\colon (\Theta, d_\Theta)\rightarrow C(\XC\times \YC)$ is $1$-Lipschitz. Then, it follows for any $\eps>0$ that \begin{align*}
   \max\Big( \NC\big(\eps, \cup_{\theta\in \Theta} \FC^{c_\theta}, \norm{\cdot}_\infty\big),  \NC\big(\eps, \cup_{\theta\in \Theta} \FC^{c_\theta c_\theta}, \norm{\cdot}_\infty\big)\Big)  \leq  \NC\left(\frac{\eps}{4}, \Theta, d_\Theta\right) \sup_{\theta\in \Theta}  \NC\left(\frac{\eps}{2}, \FC^{c_\theta c_\theta}, \norm{\cdot}_\infty\right).
\end{align*}
\end{lem}

\subsection{Assumption \ref{ass:KPU}: Uniqueness of Kantorovich Potentials}\label{subsec:AssumKP}

The uniform weak limit of the empirical OT process from \Cref{thm:OTProcessCty} demonstrates a close relation to the collection of Kantorovich potentials. In particular, for the limit to be supported on $C(\Theta)$ a certain continuity property on the Kantorovich potentials $S_{\!c_\theta}(\mu, \nu)$ with respect to $\theta$ is required. Assumption \ref{ass:KPU} on the uniqueness of Kantorovich potentials represents a sufficient condition to ensure this property.

The recent work by \citet{Staudt2022Unique} thoroughly analyzes the topic of uniqueness in Kantorovich potentials and highlights that it is often expected. 
More precisely, for differentiable costs and assuming that one probability measure is supported on the closure of a connected open set on a smooth manifold, Kantorovich potentials are unique. As Example 3 in their work showcases, uniqueness also occurs under continuous costs if one measure is discrete while the other has connected support. In case both measures have disconnected support, then uniqueness can still be guaranteed if potentials on restricted OT sub-problems are unique and if there exists, in the language of \citeauthor{Staudt2022Unique}, a \emph{non-degenerate} OT plan, meaning that all connected components of both measures are linked via that OT plan. The existence of such OT plans can be guaranteed under mild conditions on the underlying measures (see \eqref{eq:nondegenerateMeasures}) and intuitively demands that the OT problem cannot be divided into distinct sub-problems. 

The following statement is a simple consequence of the theory of \citet{Staudt2022Unique}, which we have included for ease of reference. 
\begin{prop}
\label{prop:uniquenessOTPotentials}
  Let  $c\colon \RR^d\times \RR^d\rightarrow \RR$ be a differentiable cost function. Consider probability measures $\mu, \nu\in\PC(\RR^d)$ with compact support and suppose $\supp(\mu)= \bigcup_{i \in I}\XC_i$  and $\supp(\nu) =\bigcup_{j \in J} \YC_j$ for finitely many disjoint sets. Assume each set  $\XC_i$ is either (i) the closure of a connected open set, (ii) the closure of a connected open set in a smooth compact submanifold of $\RR^d$, or (iii) a single point. Further, if $\min(|I|, |J|) \geq 2$, suppose for all non-empty, proper $I'\subset I$ and $J'\subset J$ that 
  \begin{align}\label{eq:nondegenerateMeasures}
    \sum_{i\in I'}\mu(X_i) \neq \sum_{j \in J'}\nu(Y_j),
  \end{align}
  Then, Kantorovich potentials for $\mu, \nu$ and $c$ are unique (up to a constant shift). 
\end{prop}

The proposition follows by verifying the conditions of Theorem 1 in \citet{Staudt2022Unique}. Indeed, continuity of Kantorovich potentials on $\supp(\nu)$ follows due to the compactness assumption and continuity of the cost function, uniqueness of Kantorovich potentials on subproblems follows from the assumptions on the cost and the sets $\XC_i$ \citep[Corollary~2]{Staudt2022Unique}, and existence of non-degenerate plans follows via \citep[Lemma 6]{Staudt2022Unique} due to \eqref{eq:nondegenerateMeasures}.


\section{Applications}\label{sec:Applications}

In this section we employ our theory from \Cref{sec:MainResults} to obtain novel insights about various OT related topics. All proofs for this section are deferred to \Cref{sec:ProofsApplications}.

\subsection{Optimal transport based One-Sample Goodness-of-Fit-Testing}
\label{sec: OSGoF}
 \citet{hallin2021multivariate} proposed to use the Wasserstein distance between a sample measure and a reference measure for goodness-of-fit testing under group actions. In the following, we briefly recall the setting for compactly supported measures. Let $\probQ_0 \in \PC(\RR^d)$ be compactly supported, define $\YC$ as the convex hull of $\supp(\probQ_0)$, and let $G_\Theta = \{ g_\vartheta : \vartheta \in \Theta \}$ be a group %
 of measurable transformations $g_\vartheta : \reals^d \to \reals^d$ that is parametrized by $\vartheta\in\Theta\subseteq \RR^k$ for $k \in \NN$. Further, assume that the map $x\mapsto g_\vartheta(x)$ is continuous for every $\vartheta\in\Theta$ and that the mappings $\vartheta\mapsto g_\vartheta$ and $g_\vartheta\mapsto (g_\vartheta)_\#\nu_0$ are bijective (this implies the identifiability of the model parameter). \cite{hallin2021multivariate} consider the subsequent testing problem:

\begin{displayquote}
  {\it{
    Let $G_\Theta$ be a group  and define $\model = \{ {g_{\vartheta}}_\# \nu_0 : g_\vartheta \in G_\Theta \}$. Given an i.i.d.\ sample $\{X_i\}_{i=1}^n$ from some unknown $\mu \in \PC(\XC)$ with $\X\subset\Rd$ compact, the aim is to test 
    \begin{equation}
      \label{eq:hyp:groupI}
      \mathcal{H}^*_0 :\mu \in \model \quad
      \text{ against } \quad
      \mathcal{H}^*_1 :\mu\notin\model.
    \end{equation}
    }}
\end{displayquote}
Note that the parameter $\vartheta^*$ under $\mathcal{H}_0$,  such that $(g_{\vartheta^*})_\# \probQ_0 = \mu $, is unknown. To construct a test for the above hypothesis, which is for instance of particular interest in the analysis of location-scale families, the authors propose to rely on the (2-)Wasserstein distance, i.e., they propose a test based on an empirical version of
  \begin{equation*}\label{eq:def of population quantity} OT\left(\mu, \nu_0, \norm{g_{\vartheta^*}^{-1}(\cdot)-\cdot}^2 \right) = \inf_{\pi\in\Pi(\mu,\probQ_0)}\int\! \norm{g_{\vartheta^*}^{-1}(x)-y}^2\,\diff\pi(x,y).\end{equation*}
 For this purpose, the unknown measure $\mu$ is replaced by  $\mu_n$ and the cost function $c(x,y)=\|g_{\vartheta^*}^{-1}(x)-y\|^2$ by $c_n(x,y)=\|g_{\vartheta_n}^{-1}(x)-y\|^2$, where $\vartheta_n\in\Theta$ denotes a suitable estimator for $\vartheta^*$. Thus, the proposed test statistic is given as
 \begin{equation}
 \label{eq:def of test statistic} 
 OT\left(\mu_n, \nu_0, \norm{g_{\vartheta_n}^{-1}(\cdot)-\cdot}^2 \right) = \inf_{\pi\in\Pi(\mu_n,\probQ_0)}\int\! \norm{g_{\vartheta_n}^{-1}(x)-y}^2\,\diff\pi(x,y),
 \end{equation}
 which amounts to solving an OT problem with an estimated cost function. Hence, we can apply our theory to derive the limiting distribution of 
 \begin{equation}\label{eq:limitingStat}\sqrt{n}\left( OT\left(\mu_n, \nu_0, \norm{g_{\vartheta_n}^{-1}(\cdot)-\cdot}^2 \right) -  OT\left(\mu, \nu_0, \norm{g_{\vartheta^*}^{-1}(\cdot)-\cdot}^2 \right)\right)\end{equation}
under the null hypothesis $H_0^*$ in \eqref{eq:hyp:groupI} (see \Cref{rem:GroupFam} for a discussion). In addition, we are able to extend this to testing whether $H_0^*$ holds approximately, which is often preferable in practice (see, e.g., \cite{munk1998nonparametric,dette1998validation,dette2019detecting}). For this purpose, we fix an estimation procedure for $\vartheta^*$, i.e., we choose a specific estimator $\vartheta_n$ (taking values in $\Theta$) for estimating $\vartheta^*$ and denote its population quantity by $\vartheta^o\in\Theta$ (under $\HC_0^*$ we assume $\vartheta^*=\vartheta^o$). Then, we consider the subsequent testing problem:

\begin{displayquote}
  {\it{Let $G_\Theta$ be a group. Given an i.i.d.\ sample $\{X_i\}_{i=1}^n$ from some unknown $\mu \in \PC(\XC)$ with $\X\subset\Rd$ compact, the aim is to test for some prespecified $\Delta>0$ the hypothesis}}
  \begin{equation}
    \label{eq:hyp:groupII}
    \mathcal{H}_0 : \dwasser(\mu,(g_{\vartheta^o})_\#\nu_0) \leq\Delta\quad
    \textit{ versus } \quad
    \mathcal{H}_1 :\dwasser(\mu,(g_{\vartheta^o})_\#\nu_0) >\Delta.
  \end{equation}
\end{displayquote}

In order to construct a test for the above problem, we have to derive the distributional limits of \eqref{eq:limitingStat} under the assumption that $\mu\notin\model$. To this end, we  employ the theory from Sections \ref{sec:MainResults} and~\ref{subsec:Assumptions}. The first step for the derivation of distributional limits of \eqref{eq:limitingStat} is to establish  H\"older regularity  (recall \Cref{foot:Holder}) for costs induced by $C_\Theta:\Theta\rightarrow C(\XC\times \YC), \vartheta\mapsto ((x,y) \mapsto\|g_{\vartheta}^{-1}(x)-y\|^2)$ near~$\vartheta^o$. 

\begin{lem}\label{lem:PropRegularityOneSampleGOFCost}
  Let $\XC,\YC\subseteq \RR^d$ be compact and denote by $C(\XC,\Rd)$ the space of continuous functions from $\XC$ to $\RR^d$. Assume that $K_\Theta:\Theta\subseteq \RR^k\to C(\XC,\RR^d), ~\vartheta \mapsto (x\mapsto g_{\vartheta}^{-1}(x))$ is continuous near $\vartheta^o$.
  Then, there is an open (w.r.t. relative topology) neighborhood $U\subseteq \Theta$ of $\vartheta^o$ and some $\Lambda\geq 0$ such that for any $x\in\X$ and $\vartheta\in U$ the cost function $C_\Theta(\vartheta)(x,\cdot)\coloneqq\|g^{-1}_{\vartheta}(x)-\cdot\|^2$ is $(2, \Lambda)$-H\"older on~$\YC$. 
\end{lem}

Next, we verify that Hadamard differentiability of $K_\Theta$ at $\vartheta^o$ implies Hadamard differentiability of the cost parametrizing map $C_\Theta:\Theta\rightarrow C(\XC\times \YC), \vartheta\mapsto ((x,y) \mapsto\|g_{\vartheta}^{-1}(x)-y\|^2)$ at $\vartheta^o$. To this end, we additionally impose the following assumption.
\begin{enumerate}[label=\textit{\textbf{(G)}}]
  \item \label{assum: StabParamGroup} \it{For $\vartheta^o$ there exists $m_{\vartheta^o} > 0$ such that for all $\vartheta' \in \Theta$ in some neighborhood of $\vartheta^o$, 
  \begin{align*}
    \sup_{x \in \RR^d} \frac{\norm{g_{\vartheta'}^{-1}(x) - g_{\vartheta^o}^{-1}(x) }}{1 + \norm{g_{\vartheta^o}^{-1}(x)}}
    \le
    m_{\vartheta^o} \norm{\vartheta' - \vartheta^o}.
  \end{align*}}
\end{enumerate}
This condition is fulfilled, e.g., for location-scale families and affine transformations. A global version of the above assumption, i.e., where the condition is to be fulfilled for any $\vartheta$ and not only $\vartheta^0$, has been used by \cite{hallin2021multivariate} to ensure the consistency of their goodness-of-fit test described above. 

\begin{lem}
  \label{prop: cstGoF}
  Assume that the function $K_\Theta:\Theta\to C(\XC,\RR^d), ~\vartheta \mapsto (x\mapsto g_{\vartheta}^{-1}(x))$ is Hadamard differentiable at $\vartheta^o$ tangentially to $\Theta$, i.e., for any sequence $(\vartheta^o + t_nh_n)_{n \in \NN}\subseteq \Theta$ such that $t_n\searrow 0$ and $h_n\to h \in \RR^k$ as $n\to \infty$, 
\[ 
  \lim_{n \to \infty}
\left\lVert \frac{K_\Theta(\vartheta^o + t_n h_n)-K_\Theta(\vartheta^o)}{t_n} - D^H_{|\vartheta^o}{K_\Theta}(h) \right\rVert_{\infty} =0,
\] where $D^H_{|\vartheta^o}{K_\Theta}(h)\colon \XC\rightarrow \RR^d$ is a continuous function. 
Then, if Assumption~\ref{assum: StabParamGroup} is satisfied, $C_\Theta$ is Hadamard differentiable at $\vartheta^o$ tangentially to $\Theta$ with derivative $D^H_{|\vartheta^o}C_{\Theta}(h)\in C(\XC\times \YC)$ given by 
\[
  D^H_{|\vartheta^o}C_{\Theta}(h) \colon \XC\times \YC\rightarrow \RR, \quad (x,y)\mapsto  2 \Big\langle D^H_{\vartheta^o}{K_\Theta}(h)(x),g_{\vartheta^o}^{-1}(x) - y\Big\rangle . 
\]
Moreover, if 
$
\sqrt{n} (\vartheta_n - \vartheta^o) \dto \Gproc^{\vartheta}$ for $n\rightarrow \infty$, we obtain for $c_n\coloneqq C_\Theta(\vartheta_n)$ and $c\coloneqq C_\Theta(\vartheta)$ that
\begin{equation}
\label{eq: LimCostGoF}
\sqrt{n}( c_n -c) \dto \mathbb{G}^{c}\coloneqq \left(2 \Big\langle D^H_{\vartheta^o}{K_\Theta}(\Gproc^{\vartheta})(x),g_{\vartheta^o}^{-1}(x) - y\Big\rangle \right)_{(x,y) \in \XC\times \YC} \qquad \text{ in } C(\X\times \Y).
\end{equation}
\end{lem}

Under the conditions in the proposition above, our main result from \Cref{thm:AbstractMainResult} yields a (typically) non-degenerate limiting distributions for the statistic ${OT}\left(\mu_n, \nu, c_n \right)$ under the assumption that $\mu\notin\model$. In particular, this allows us to construct an asymptotic level $\alpha$ test for the null hypotheses given in \eqref{eq:hyp:groupII} (see \cite{munk1998nonparametric} for the precise construction). 

\begin{prop}\label{prop:gouplimit}
  Let $\nu_0\in \PC(\RR^d)$ for $d\le 3$ be compactly supported and define $\Y$ as the convex hull of $\supp(\nu_0)$, and let $\XC \subseteq \RR^d$ be compact. Assume that \ref{assum: StabParamGroup} holds, and suppose that $G_\Theta:\Theta\to C(\XC, \RR^d), ~\vartheta \mapsto (x\mapsto g_{\vartheta}^{-1}(x))$ is  continuous near $\vartheta^o$ and Hadamard differentiable at $\vartheta^o$. Define for $\Lambda\geq 0$ from \Cref{prop: cstGoF} the function class  $$\FC\coloneqq \left\{f \colon \YC \rightarrow \RR \;\Big|\; \norm{f}_\infty \leq \Lambda+1, |f(y) - f(y')| \leq 2\Lambda\norm{y-y'} \text{ for all } y,y' \in \YC\right\}.$$
Then, the function class $\FC^{C_{\Theta}(\vartheta^o)}$ on $\XC$ is universal Donsker. Moreover, for i.i.d. random variables $\{X_i\}_{i=1}^n\sim\mu^{\otimes n}$ consider a measurable estimator $\vartheta_n$ and suppose for $n\to\infty$ joint weak convergence, 
\begin{align}
\label{eq: JointMeasParam}
\sqrt{n} \begin{pmatrix}
\mu_n - \mu\\
\vartheta_n - \vartheta^o
\end{pmatrix} 
\dto
\begin{pmatrix}
\Gproc^\mu\\
\Gproc^{\vartheta}
\end{pmatrix}\quad \text{ in }\ell^\infty(\FC^{C_{\Theta}(\vartheta^o)})\times \RR^k.%
\end{align}
Then, for $c_n\coloneqq C_\Theta(\vartheta_n)$ and $c\coloneqq C_\Theta(\vartheta)$ and by denoting the limit from~\eqref{eq: LimCostGoF} as $\Gproc^{c}$, it~follows that
\[
\sqrt{n}
\left(  
OT\left(\mu_n, \nu_0,c_n \right)  -  OT\left(\mu, \nu_0, c\right) 
\right)
\dto \inf_{\pi \in \Pi_{c}^\star(\mu, \nu_0)} \pi ( \Gproc^{c} ) + \sup_{f \in S_{\!c}(\mu, \nu_0)} \Gproc^\mu(f^{c}),
\]
where $S_{\!c}(\mu, \nu_0)$ represents the set of optimizers for $\sup_{f\in \FC}\mu(f^{c})+ \nu_0(f^{cc})$.  
\end{prop}

\begin{rmk}\label{rem:GroupFam} A few comments on the distributional limits are in order. 
\begin{enumerate}[label=$(\roman*)$]
\item Given that the function class $\FC^{C_\Theta(\vartheta^o)}$ is universal Donsker and thus $\mu$-Donsker, and assuming that $\sqrt{n}(\vartheta_n - \vartheta^o)$ converges in distribution, the requirement of joint convergence as required in \eqref{eq: JointMeasParam} is very mild. Indeed, if $\sqrt{n}(\vartheta_n - \vartheta^o)$ can be expressed asymptotically in terms of a suitable linear functional of an empirical process, i.e., if it admits an asymptotic influence function $\psi\in L^2(\mu)$ (cf. \citealt[p.~58]{vaart_1998}), joint convergence follows since the union $\FC^{cc}\cup\{\psi\}$ is $\mu$-Donsker.
\item We like to point out that \Cref{prop:gouplimit} also remains valid if $\mu\in\model$. 
However, under this assumption it follows that $(g_{\vartheta^o}^{-1})_{\#}\mu = \nu_0$ which implies that the corresponding OT plan between $\mu$ and $\nu_0$ is given by $\pi = (\Id, g_{\vartheta^o}^{-1}(\cdot))_{\#}\mu$. Hence, by \eqref{eq: LimCostGoF} the process $\Gproc^c$ vanishes along the support of $\pi$ and the first term in the limit degenerates. Further, if the support of $\nu_0$ is connected since then Kantorovich potentials are unique up to a constant shift \citep[Corollary 2]{Staudt2022Unique} and a.s.\ constant \citep[Corollary 4.6(i)]{Hundrieser2022Limits}. Consequently, for this setting the corresponding limit distribution is degenerate. In~contrast, if $\nu_0$ has disconnected support, non-constant Kantorovich potentials exist \citep[Lemma 11]{Staudt2022Unique} which results in a non-degenerate limit. 
\item The elements presented for the one-sample case can also be generalized to the case where both empirical measures undergo a transformation, either separately or jointly. One might think of choosing the Mahalanobis distance $(x-y)^\top\Sigma^{-1}(x-y)$ as a cost function where $\Sigma^{-1}$ has to be estimated and could, e.g., be a diagonal matrix. As the OT value is not invariant with respect to affine transformations, rescaling the variables would ensure that no component has an overwhelming impact on the cost function compared to the other components. 
\end{enumerate}
\end{rmk}

\subsection{Optimal Transport with Embedded Invariances}
\label{sec: OTwInv}
In a similar spirit to the previous section, another strand of the literature  \citep{alvarez2019towards,grave2019unsupervised} aims at making OT invariant to a class of transformation $\mathcal{T}$, with $\tau\colon \RR^d\rightarrow \RR^d$ continuously differentiable for each $\tau\in \mathcal{T}$, by considering 
\begin{equation}
\label{eq: OTInv}
\inf_{\tau \in \mathcal{T}} OT\left(\mu, \nu, \lVert \cdot - \tau(\cdot) \rVert^2\right) = \inf_{\tau \in \mathcal{T}} \inf_{\pi \in \Pi(\mu, \nu)} \int_{\X\times \Y} \lVert x- \tau(y) \rVert^2 \diff \pi (x,y).
\end{equation}
This distance is useful in many contexts, among which the word embedding problem or protein alignment. If the class of transformations considered is the set of rotations, analyses relying on that distance is coined Wasserstein--Procrustes Analysis \citep{grave2019unsupervised,jin2021two}.  \Cref{thm:OTProcessInf} provides the required tools for statistical inference for the empirical version of the optimization problem in \eqref{eq: OTInv}.

\begin{prop}\label{prop:OTInvariancesLimit}
  Consider a set of transformations $(\mathcal{T},d_\mathcal{T})$ that is a compact metric space with $\log \NC(\eps,\mathcal{T},d_\mathcal{T})\lesssim \eps^{-\alpha}$ for $\alpha<2$. 
  Let $\XC,\YC\subseteq \RR^d$ be compact subsets and assume that the functional $c:(\mathcal{T},d_\mathcal{T})\rightarrow C(\XC\times \YC), \tau \mapsto  c_\tau$, with $c_\tau(x,y)= \lVert x- \tau(y) \rVert^2$, is $L$-Lipschitz for some $L\geq 0$. 
  Further, assume for $\XC$ and $\{c_t\}_{t\in \mathcal{T}}$ any of the settings from Proposition~\ref{prop:DonskerProperty_ParameterClass} and take $\mu\in \PC(\XC), \nu\in \PC(\YC)$ such that the support of $\mu$ or $\nu$ is the closure of a connected open set in $\RR^d$. 
  Then, for  $\{X_i\}_{i = 1}^{n}\sim \mu^{\otimes n}$ and $\{Y_i\}_{i = 1}^{m}\sim \nu^{\otimes m}$, respectively, with $n,m \to \infty$ such that $m/(n+m) \to \lambda \in (0,1)$, it holds that 
\[
\sqrt{\frac{nm}{n+m}}\Big( \inf_{\tau \in \mathcal{T}} OT( \mu_n, \nu_m,c_\tau) - \inf_{\tau\in \mathcal{T}} OT(\mu, \nu,c_\tau)\Big) \dto \inf_{\tau \in S_{\!-}(\mathcal{T}, \mu, \nu)} \sqrt{\lambda}\; \Gproc^\mu(f_\tau^{c_\tau c_\tau})  + \sqrt{1-\lambda}\; \Gproc^\nu(f_\tau^{c_\tau}),
\]
where $f_\tau\in S_{\!c_\tau}(\mu, \nu)$ denotes a Kantorovich potential between $\mu$ and $\nu$ and cost $c_\tau$ for $\tau\in S_{\!-}(\mathcal{T}, \mu, \nu)$. 
\end{prop}

As previously noted, one can relax the requirement that $\mathcal{T}$ is compact to the assumption that the sequence of estimated optimal transformation $\tau_n$ is contained within a compact set with probability tending to one (\Cref{cor:OTInfimumNonCompactTheta}).
In the setting of $\mathcal{T}$ consisting of diffeomorphisms we have by Lemma 1 of \cite{nies2021transport}
$$\inf_{\tau \in \mathcal{T}} OT\left(\mu, \nu, \lVert \cdot - \tau(\cdot) \rVert^2\right) = \inf_{\tau \in \mathcal{T}} OT\left(\mu, (\tau^{-1})_{\#}\nu, \lVert \cdot - \cdot \rVert^2\right) = \inf_{\tau \in \mathcal{T}}W_2^2\left(\mu, (\tau^{-1})_{\#}\nu\right),$$ 
for which convergence of empirical minimizers $\tau_n$ can be verified for various settings using results by \citet{bernton2019parameter}.

\begin{rmk}[Wasserstein--Procrustes]
The above proposition can be applied under mild regularity assumptions on the measures to the special orthogonal group $\mathcal{T}\coloneqq\SO(d)$  for $d\leq 3$.  Indeed, upon choosing $\XC, \YC$ as compact, convex sets of $\RR^d$ setting $(iii)$ of \Cref{prop:DonskerProperty_ParameterClass} is fulfilled, asserting \ref{ass:Don}. Moreover, if the support of $\mu$ or $\nu$ is the closure of a connected open set in $\RR^d$, then \ref{ass:KPU} holds  and the distributional limits of \Cref{thm:OTProcessInf} follow. 
\end{rmk}

\subsection{Sketched Wasserstein Distance for Mixture Distributions}
\label{sec: SketchWas}

Recently, \citet{delon2020wasserstein} and \citet{bing2022sketched} investigated a distance between (Gaussian) mixtures distributions. These distributions are ubiquitous in statistics and machine learning, see \citet{MixtRev} and the references therein.
One way of understanding that distance is to start from the Wasserstein distance between discrete measures but instead of using a cost function between points, one replaces the points by distributions and one must thus choose a cost between distributions. 
Before formally defining that concept, recall that, for a set of distributions $\mathcal{A}:= (\mathrm{A}_1, \ldots , \mathrm{A}_K )$ of finite cardinality $K$, a mixture $r$  is a convex combination of components from $\mathcal{A}$ given by a vector $\alpha \in \Delta_K$, i.e., 
$
r = \sum_{i=1}^K \alpha_i \mathrm{A}_i, 
$
 where $\Delta_K$ is the probability simplex in $\reals^K$. 
Given a distance  $d\colon \mathcal{A}\times \mathcal{A}\rightarrow \RR_+$ between mixture components of $\mathcal{A}$, the aforementioned authors define the \textit{Sketched Wasserstein distance} between two mixture distributions with weights $\alpha$  and $\beta$ as
\[
W(\alpha, \beta, d) \coloneqq \inf_{\pi \in \Pi(\alpha, \beta)}\sum_{k,\ell=1}^K \pi_{k,\ell} d(\mathrm{A}_k, \mathrm{A}_\ell),
\] 
where the infimum is taken over elements of the set of couplings
\[
\Pi(\alpha, \beta)= \left\{ \pi \in \Delta_{K\times K} \,\Bigg \vert 
\begin{array}{c}\sum_{\ell=1}^K \pi_{k,\ell} = \alpha_k,\quad    \text{for all } k \in \{1, \dots, K\}  \\
   \sum_{k=1}^K \pi_{k,\ell} = \beta_\ell, \quad  \text{for all } \ell \in \{1, \dots, K\}
\end{array}
\right\}.
\]
Understanding the fluctuations of an estimator for this distance can be achieved using the theory developed in the present paper. This is formalized in the following proposition.
\begin{prop}\label{pro:sketchyWassy}
Let $({\alpha}_n, 
 {\beta}_n$, 
${d_n})\in \Delta_K\times \Delta_K\times \RR^{K^2}_+$ be measurable estimators for $\alpha, \beta$,  $d$, respectively. Further, for a positive sequence $(a_n)_{n\in \NN}$ with $\lim_{n \rightarrow \infty} a_n = \infty$, assume for $n\to \infty$ that  
\begin{align}\label{eq:jointweakConvergenceSketchedWasserstein}
  a_n \begin{pmatrix}
    {\alpha}_n -\alpha\\[0.02cm]
     {\beta}_n -\beta\\[0.02cm]
    {d}_n - d
    \end{pmatrix}= 
  a_n
  \begin{pmatrix}
  ({\alpha}_{n,k} -\alpha_k)_{k = 1}^{K}\\[0.02cm]
   ({\beta}_{n,k} -\beta_k)_{k = 1}^{K}\\[0.02cm]
  ({d}_n(\mathrm{A}_k, \mathrm{A}_\ell) - d(\mathrm{A}_k, \mathrm{A}_\ell))_{l,k= 1}^{K}
  \end{pmatrix} \dto 
  \begin{pmatrix}
  \Gproc^\alpha\\[0.02cm]
  \Gproc^\beta\\[0.02cm]
  \Gproc^d
  \end{pmatrix} \quad \text{in } \RR^{2K + K^2},%
\end{align}
where $(\Gproc^\alpha,\Gproc^\beta, \Gproc^d)$ represents a tight (possibly non-Gaussian) random variable on  $\RR^{2K + K^2}$. Then, 
\[
a_n \Big(W({\alpha}_n, {\beta}_n,{d}_n) - W(\alpha, \beta,d) \Big) \dto \inf_{ \pi \in \Pi^\star_d(\alpha,\,\beta)} \langle \pi,\Gproc^d\rangle + \sup_{f \in S_{\!d}(\alpha, \beta)} \langle f^{dd}, \Gproc^\alpha \rangle + \langle f^d, \Gproc^\beta \rangle.
\]
\end{prop}

The proof follows along the same approach as for showing \Cref{thm:AbstractMainResult} and is therefore omitted, see \Cref{rmk:CommentsOTWeaklyConvergingCosts} \ref{rem:WeaklyConvergingCosts_GeneralMeasureEstimator}. 
 In this context, the requirement of weak convergence for the measure estimators $(\alpha_n, \beta_n)\dto (\alpha, \beta)$ in probability follows from our assumption in \eqref{eq:jointweakConvergenceSketchedWasserstein} since the population measures and its estimators are supported on finitely many points. 

In  \citet{bing2022sketched}, they obtain distributional limits in the case where the asymptotic fluctuation of the cost is negligible in comparison to the estimated measures. Their results are recovered by \Cref{pro:sketchyWassy}, which in addition covers the setting where the cost is estimated on the same data and converges at the same rate. Finally, we stress that the case of Gaussian mixtures, a particularly relevant one in applications, is also covered by our theory. Nonetheless, the lack of distributional results for  estimators of the mixture parameters in that case still hinders further developments and would be of interest for further research.

\subsection{Sliced Optimal Transport}
\label{sec: Sliced}
Our theory from \Cref{subsec:OTProcessResults} also enables the analysis of sliced OT quantities and complement or extend available results from the literature \citep{goldfeld2022statistical,manole2022minimax, xi2022distributional,xu2022central}. In the following, we formalize this statement. For two Borel probability measures $\mu,\nu\in\PC(\RR^d)$ the \emph{average-sliced} and \emph{max-sliced Wasserstein distances of order $1\leq p<\infty$} are defined, respectively, as
\begin{equation*}
	 	\underline{W}_p(\mu,\nu)\coloneqq\left(\int_{\mathbb{S}^{d-1}}\! OT(\mathfrak{p}^\theta_\#\mu, \mathfrak{p}^\theta_\#\nu, |\cdot-\cdot|^p)\,\diff\sigma(\theta) \right)^\frac{1}{p}~\text{ and }~\overline{W}_p(\mu,\nu)\coloneqq\max_{\theta\in\mathbb{S}^{d-1}}\left( OT(\mathfrak{p}^\theta_\#\mu, \mathfrak{p}^\theta_\#\nu, |\cdot - \cdot|^p) \right)^\frac{1}{p}\!\!,
\end{equation*}
 where $\mathfrak{p}^\theta:\RR^d\to\RR$ is the projection map $x\mapsto \theta^Tx$ and $\sigma$ represents the uniform distribution
 on the unit sphere $\mathbb{S}^{d-1}$. Note  by Lemma 1 in \citet{nies2021transport} for any $\theta \in \mathbb{S}^{d-1}$ that
 \begin{align*}
 	OT(\mathfrak{p}^\theta_\#\mu, \mathfrak{p}^\theta_\#\nu, |\cdot - \cdot|^p) = 	OT(\mu, \nu, |\mathfrak{p}^\theta(\cdot) - \mathfrak{p}^\theta(\cdot)|^p), %
 \end{align*}
 which enables to view the sliced Wasserstein quantities in the framework of \Cref{subsec:OTProcessResults} and asserts by Theorems \ref{thm:OTProcessCty}--\ref{thm:OTProcessSup} the following result. 
\begin{prop}\label{prop:sliced OT}
Let $p\geq 1$, $d\geq 2$, and define for $\theta\in \mathbb{S}^{d-1}$ the cost   $c_\theta:\RR^d\times\RR^d\to\RR$, $(x,y)\mapsto |\mathfrak{p}^\theta(\cdot) - \mathfrak{p}^\theta(\cdot)|^p$. Further, take compactly supported probability measures $\mu, \nu \in \PC(\RR^d)$ with empirical measures $\mu_n, \nu_m$, respectively. For all assertions, we let $n,m\rightarrow \infty$ with $m/(n+m)\rightarrow \lambda\in (0,1)$. 
\begin{enumerate}[label=$(\roman*)$]
	\item\label{prop:sliced process}Assume that the set of Kantorovich potentials $S_{\!c_\theta}(\mu, \nu)$ is unique (up to a constant shift) for any $\theta\in \mathbb{S}^{d-1}$. Then, it follows upon selecting $f_\theta\in S_{\!c_\theta}(\mu, \nu)$  for any $\theta\in\mathbb{S}^{d-1}$ that
$$\!\!\!\!\!\sqrt{\frac{nm}{n+m}}\Big(OT(\mu_n, \nu, c_\theta) - OT(\mu, \nu, c_\theta)\Big)_{\theta\in \mathbb{S}^{d-1}} \!\!\dto \Big(\sqrt{\lambda}\Gproc^\mu(f_\theta^{c_\theta c_\theta})+  \sqrt{1-\lambda}\Gproc^\nu(f_\theta^{c_\theta})\Big)_{\theta\in \mathbb{S}^{d-1}} \;\text{ in } C(\mathbb{S}^{d-1})$$

	\item\label{prop:sliced OT_AvSliced}Assume the same as in \ref{prop:sliced process}. Then, it follows that
\begin{align*}
		\sqrt{\frac{nm}{n+m}}\left( \underline{W}^p_p( \mu_n, \nu_m)- \underline{W}^p_p(\mu, \nu)\right)&\dto \int_{\mathbb{S}^{d-1}}\!\sqrt{\lambda}\Gproc^\mu(f_\theta^{c_\theta c_\theta})+  \sqrt{1-\lambda}\Gproc^\nu(f_\theta^{c_\theta})\,\diff\theta.\;\;\quad 
	\end{align*}
	\item \label{prop:sliced OT_MaxSliced}Without imposing the assumption on uniqueness of Kantorovich potentials, it follows that 
	$$\sqrt{\frac{nm}{n+m}}\left( \overline{W}_p^p( \mu_n, \nu_m)- \overline{W}_p^p(\mu, \nu)\right)\dto \sup_{\substack{\theta\in S_{\!+}(\mathbb{S}^{d-1}\!, \,\mu, \nu)\\ f_\theta\in S_{\!c_\theta}(\mu, \nu)}} \sqrt{\lambda}\Gproc^\mu(f_\theta^{c_\theta c_\theta})+  \sqrt{1-\lambda}\Gproc^\nu(f_\theta^{c_\theta}).$$
\end{enumerate}
\end{prop}

Comparing \Cref{prop:sliced OT} to the literature for $p>1$, 
results in \citet{goldfeld2022statistical} and \citet{xi2022distributional} are recovered under slightly weaker assumptions. For the analysis of both types of empirical sliced Wasserstein distances \citet{goldfeld2022statistical} require the underlying measures to have compact, convex support. Moreover, for the uniform central limit theorem by \citet{xi2022distributional} of the sliced OT process, they assume for each $u \in \mathbb{S}^{d-1}$ that one of the projected measures has compact, connected support. These conditions are sufficient for the uniqueness of Kantorovich potentials, but it can also be guaranteed for measures with disconnected support (cf. \Cref{prop:uniquenessOTPotentials} and more generally \citealt{Staudt2022Unique}).
\Cref{prop:sliced OT}$(ii)$ also complements results by \citet{manole2022minimax} on the trimmed sliced Wasserstein distance as we do not require the existence of a density but the underlying measures to be compactly supported. 

For the special case $p = 1$, unlike in our results, distributional limits by \citet{goldfeld2022statistical, xu2022central} for the average- and max-sliced Wasserstein distance do not require uniqueness of the Kantorovich potentials. Further, their theory remains valid for non-compactly supported measures by imposing suitable moment-conditions. Crucial to their approach is the special characterization of the $1$-Wasserstein distance as an integral probability metric over Lipschitz functions \citep[Remark 6.5]{villani2008optimal}, a property which we do not exploit in our general theory. 
Still, under uniqueness of Kantorovich potentials, which occurs, e.g., if one measure is discrete while the other has connected support  and is absolutely continuous \citep[Example 3]{Staudt2022Unique}, \Cref{prop:sliced OT}$(i)$ asserts weak convergence for the sliced OT process in $C(\mathbb{S}^{d-1})$.

\subsection{Stability analysis of Optimal Transport}\label{subsec:stability}
In addition to statistical applications, our theory for the empirical OT value under weakly converging costs enables a deterministic stability analysis of the OT problem \eqref{eq: OT} under joint perturbations of the costs and the measures, which may be of independent interest, e.g., from the viewpoint of optimization. More precisely, we prove in the following Gateaux differentiability of the OT value in $(\mu,\nu,c)\in\PC(\XC)\times\PC(\YC)\times C(\XC\times\YC)$ for all admissible directions.
This extends well-known stability results for finite-dimensional linear programs \citep[Theorem 3.1]{gal1997advances} which covers the OT problem for probability measures supported on finitely many points. 

\begin{prop}\label{prop:gateaux}
	Let $\mu\in\PC(\X)$, $\nu \in \PC(\YC)$ and $c\in C(\XC\times \YC)$ be fixed. Define for $t>0$ sufficiently small the quantities $\mu_t= \mu +t\Delta^\mu$ and $\nu_t= \nu +t\Delta^\nu$, where $\Delta^\mu\in(\PC(\XC)-\mu)$ and $\Delta^\nu\in(\PC(\YC)-\nu)$, respectively. Further, let $c_t = c+t\Delta^{c}$ for some $\Delta^c\in C(\XC\times\XC)$. Then, it follows that
	\[\lim_{t \searrow 0}\frac{1}{t}\left( OT(\mu_t, \nu_t, c_t)-OT(\mu, \nu, c)\right)=\inf_{\pi\in \Pi^\star_c(\mu, \nu)}\pi(\Delta^c)+\sup_{f \in S\!_c(\mu,\nu)} \Delta^\mu(f^{cc}) +\Delta^\nu(f^c).\]
\end{prop}

\begin{rmk}[On Hadamard directional differentiability]\label{rmk:OnHadamardDifferentiability}
Since the set of admissible directions $(\PC(\XC)- \mu)\times (\PC(\YC)-\nu)\times C(\XC\times \YC)$ is not a normed vector space, we are in general unable to infer Hadamard directional differentiability by additionally proving Lipschitzianity of the OT problem with respect to the measures $\mu,\nu$ and the cost function $c$. 

Invoking the same proof strategy as in \Cref{prop:gateaux} would require us to show for any sequence $(\mu_n, \nu_n, c_n)=(\mu + t_n\Delta_n^\mu, \nu+ t_n\Delta_n^\nu, c+ t_n\Delta_n^c)\in \PC(\XC)\times \PC(\YC)\times C(\XC\times \YC)$ with $t_n\searrow 0$ and $(\Delta_n^\mu, \Delta_n^\nu, \Delta_n^c)\rightarrow (\Delta^\mu, \Delta^\nu, \Delta^c)$ in $\ell^\infty(\FC^{cc})\times \ell^\infty(\FC^{c})\times C(\XC\times \YC)$ that 
\begin{align}\label{eq:openproblem}
  \sup_{f\in \FC}\left|\Delta^\mu_n(f^{ c_nc_n} - f^{cc}) + \Delta^\nu_n(f^{ c_n} - f^{c})\right|\rightarrow 0.
\end{align}
Showing this remains a challenge and would enable us to omit conditions \ref{ass:BTwo} and \ref{ass:BTwoStar} in the formulations of Theorems \ref{thm:AbstractMainResult} and \ref{thm: BootConst}, respectively. Another challenge in such an attempt is that any such sequence $(\mu_n, \nu_n)$ does not necessarily converge weakly for $n\rightarrow \infty$ to $(\mu, \nu)$, which is relevant for our proof, since the topology induced by $\ell^\infty(\FC^{cc})\times \ell^\infty(\FC^{c})$ may be too weak. 

Though it is likely possible to show Hadamard directional differentiability of the OT problem jointly in the measures and the cost by selecting a sufficiently strong norm that metrizes weak convergence of measures, the functional delta method would inevitably require the empirical process to weakly converge in this norm and impose additional conditions.  A similar trade-off for the choice of the norm is natural and known in the literature (cf. \citealt[p.76]{dudley1990nonlinear}; \cite{jourdain2021central}).
\end{rmk}

\section{Regularity Elevation Functionals}\label{subsec:RegEl}

In this section, we construct regularity elevation maps, i.e., continuous maps $\Psi\colon C(\XC\times \YC)\rightarrow C(\XC\times \YC)$ such that for measurable estimators $c_n$ with $\sqrt{n}(c_n - c) \dto \Gproc^c$ for $n\rightarrow \infty$, it follows that 
\begin{align}\label{eq:ConditionRegElMaps}
  (i)\; \sqrt{n}\big( c_n - \Psi(c_n)\big)\pto 0 \quad \text{ and } \quad (ii)\;\, \Psi(c_n) \text{ fulfills certain regularity properties.} 
\end{align}
Based on Lipschitzianity of the OT value with respect to the cost function (\Cref{lem:LowerUpperBound}), condition $(i)$ allows us to substitute a cost estimator with one that enjoys certain regularity properties, effectively "elevating" its level of regularity. Such maps prove useful in our work at two particular instances. For one, it enables us to assume in the proof of \Cref{thm:AbstractMainResult} that cost estimators are suitably bounded and exhibit the same modulus of continuity as the population cost function (cf. \Cref{thm:RegElModulus}). This represents an important step to rely on \Cref{lem:RelFCandConcave}. Moreover, the notion of regularity elevations also represents a useful tool to prove \Cref{cor:SufficientConditionsSupConditions}, for which we employ \Cref{prop:AbstractB2_CoveringNumbers} and set $\tilde c_n \coloneqq \Psi(c_n)$ for a suitable regularity elevation map. Insofar, these maps serve as an effective tool for the theoretical analysis of distributional limits. 

The subsequent result provides a first set of conditions to ensure that condition $(i)$ of \eqref{eq:ConditionRegElMaps} is met. Its proof as well as the proof of all subsequent results of this section are detailed in \Cref{ap: RegElev}. 

\begin{prop}\label{thm:RegElevationAbstract} Let $\XC,\YC$ be compact Polish spaces and let $c_n\in C(\XC\times \YC)$ be a (Borel measurable) random sequence such that $a_n(c_n - c)\dto \LC$ in $C(\XC\times \YC)$ for some $c\in C(\XC\times \YC)$ and $(a_n)_{n\in \NN}$ such that $a_n \rightarrow \infty$ for $n\rightarrow \infty$. 
Let $U\subseteq C(\XC\times \YC)$ be a linear subspace such that $\LC$ is a.s.\ contained in $U$. 
Then, if $\Psi\colon C(\XC\times \YC)\rightarrow C(\XC\times \YC)$ is continuous near $c$, Hadamard directionally differentiable at $f$ with a derivative such that $D^H_{c}\Psi|_{U}= \Id_U$ and $\Psi(c) = c$, it follows for $n\rightarrow \infty$ that $$a_n\big(c_n - \Psi(c_n)\big)\smash{\pto} 0 \quad \text{for } n \rightarrow \infty.$$ 
\end{prop}

Notably, in case $\Psi$ is Hadamard differentiable with $D^H_{f}\Psi = \Id_{C(\XC\times \YC)}$, one may select $U = C(\XC\times \YC)$ and the condition on the limit $\LC$ becomes vacuous. 

To conclude various types of useful regularity properties, as required in $(ii)$ of \eqref{eq:ConditionRegElMaps}, we thus define in the following subsections various maps such that the conditions of \Cref{thm:RegElevationAbstract} are met. Additionally, we provide suitable metric entropy bounds for $\FC^{\Psi(\tilde c) \Psi(\tilde c)}$ independent of $\tilde c\in C(\XC\times \YC)$.

\subsection{Regularity Elevation to Deterministic Boundedness}\label{subsec:RegELBdd}

Consider compact Polish spaces $\XC, \YC$ and let $c\in C(\XC\times \YC)$ be a continuous cost function such that $\norm{c}_\infty\leq 1$. We define the regularity elevation functional for boundedness as
\begin{align*}
  \Psi_{\textup{bdd}} \colon C(\XC\times \YC) \rightarrow C(\XC\times \YC), \quad \tilde c \mapsto \Big( (x,y) \mapsto \max(\min(\tilde c(x,y), 2),-2)\Big).
\end{align*}

\begin{prop}\label{prop:RegEl:Bdd}
	For the above setting, $\Psi= \Psi_{\textup{bdd}}$ fulfills $  \Psi(c) = c$, is continuous, and it is Hadamard differentiable at $c$ with $D^H_{|c}\Psi= \Id_{C(\XC\times \YC)}$. In particular, if $\XC$ is a finite space,  we obtain  for any uniformly bounded function class $\GC$ on $\YC$ that  $$\sup_{\tilde c\in C(\XC\times \YC)}\log\NC(\eps, \GC^{\Psi(\tilde c)}, \norm{\cdot}_\infty) \lesssim |\log(\eps)|.$$
\end{prop}

Hence, for our analysis of the empirical OT value under estimated cost functions we can assume without loss of generality that cost estimators are deterministically bounded by a constant that depends on the population cost. In the following we prove a similar insight for the modulus of continuity for cost estimators on compact (pseudo-)metric spaces. 

\subsection{Regularity Elevation to Concave Modulus of Continuity and Lipschitzianity}\label{subsec:RegELMod}

Consider compact Polish spaces $\XC, \YC$ and let $\tilde d_\XC$ be a continuous (pseudo-)metric on $\XC$. Denote by $\tilde \XC$ the space $\XC$ equipped with the topology induced by $\tilde d_\XC$ which is also compact (\Cref{lem:pseudoMetricSpaceCompact}) but potentially does not satisfy the Hausdorff property. Let $c\in C(\tilde \XC\times \YC)$ be a cost function such that $\norm{c}_\infty\leq 1$ and consider a concave modulus $w\colon \RR_+\rightarrow \RR_+$ with $w(\delta) > 0$ for $\delta>0$ such that 
\begin{align}\label{eq:modLipschitz}
  |c(x,y)  - c(x', y)| \leq  w(\tilde d_\XC(x,x')) \quad \text{for any }x,x'\in \tilde \XC, y\in \YC.
\end{align}
If $c(\cdot,y)$ is $1$-Lipschitz under $\tilde d_{\XC}$, then select $w(t) \coloneqq t$ and if $c(\cdot,y)$ is $(\gamma,1)$-H\"older for $\gamma\in (0,1]$ (recall footnote~\ref{foot:Holder}),  choose $w(t) \coloneqq t^{\gamma}$. The regularity elevation functional for $w\circ d_\XC$ is then given by

$$\Psi_{\textup{mod}}^{w \circ \tilde d_\XC} \colon C(\XC\times \YC) \rightarrow C(\tilde \XC\times \YC), \quad \tilde c \mapsto \left( (x,y) \mapsto \inf_{x' \in \XC} \tilde c(x',y) + 2w(\tilde d_\XC(x,x'))\right)$$

\begin{prop}\label{prop:RegEl:Mod}
For the above setting, $\Psi = \Psi_{\textup{mod}}^{w \circ \tilde d_\XC}\circ \Psi_{\textup{bdd}}$ fulfills $\Psi(c) = c$, it is continuous near~$c$, and it is Hadamard directionally differentiable at $c$ with $D^H_{|c}\Psi|_{C(\tilde \XC\times \YC)}= \Id_{C(\tilde \XC\times \YC)}$. Further, for any uniformly bounded function class $\GC$ on $\YC$ it holds that  $$\sup_{\tilde c\in C(\XC\times \YC)}\log\NC(\eps, \GC^{\Psi(\tilde c)}, \norm{\cdot}_\infty)  \lesssim \NC(\eps/8, \XC, w\circ \tilde d_\XC) |\log(\eps)|.$$
\end{prop}

An appealing consequence of the above considerations is that they allow us to construct a regularity elevated estimator $\tilde c_{n,m}$ from $c_{n,m}$ such that $\HC_{\tilde c_{n,m}} \subseteq \FC^{\tilde c_{n,m} \tilde c_{n,m}}$, for $\FC = \FC(2\norm{c}_\infty+1, 2w)$ defined in \eqref{eq:defF}, holds deterministically. 

\begin{corollary}\label{thm:RegElModulus}
Let $c\in C(\XC\times \YC)$, set $B\coloneqq \norm{c}_\infty+1/2$ and let $w\colon \RR_+ \rightarrow \RR_+$ be a concave modulus with $w(\delta)>0$ for $\delta>0$ such that \eqref{eq:modLipschitz} holds for a metric $d_\XC$ on $\XC$. Assume for a random sequence $ c_n \in C(\XC\times \YC)$ that $a_n( c_n - c)\dto \Gproc^c$ in  $C(\XC\times \YC)$ with $a_n\to \infty$. Then, the random sequence $$\overline c_n \coloneqq B\cdot \Psi_{\textup{mod}}^{w\circ d_\XC/B}\circ \Psi_{\textup{bdd}}(c_n/B)\in C(\XC\times \YC)$$ satisfies $a_n\norm{\overline c_n -  c_n}_\infty\pto 0$ for $n\rightarrow \infty$ and deterministically fulfills $\norm{\overline c_n}_\infty \leq 2B = 2\norm{c}_\infty+1$, relation \eqref{eq:modLipschitz}, and the inclusion $\HC_{\overline c_n} \subseteq \FC^{\overline c_n \overline c_n}(2\norm{c}_\infty+1, 2w)$. 
\end{corollary}

\subsection{Regularity Elevation to H\"older Functions of Order $\gamma \in (1,2]$}\label{subsec:RegElHol}

Since we are able to leverage for convergence rates of the empirical OT value (recall \Cref{prop:DonskerProperty}$(i)$) the regularity of the underlying cost function up to H\"older degree $\gamma \leq 2$, we provide in this subsection a corresponding regularity elevation map. As the setting for $\gamma\leq 1$ can be treated using the theory from previous subsection, we only focus on the regime of $\gamma \in (1,2]$. 

Consider a convex, compact set $\XC\subseteq \RR^d$  with non-empty interior. 
Let $c\in C(\XC\times \YC)$ be a cost function such that $\norm{c}_\infty\leq 1$ and assume $c$ is continuously differentiable in $x$, i.e., suppose that $\nabla_x c\colon \interior{\XC}\times \YC\rightarrow \RR^d$ can be continuously extended to $\XC\times \YC$. Further, suppose that $c(\cdot,y)$ is $(\gamma,1)$-H\"older for each $y\in \YC$ for $\gamma\in (1,2]$. We define the regularity elevation map for H\"older functions of order $\gamma\in (1,2]$ by
\begin{align*}
  \Psi^{c, \gamma}_{\textup{Hol}} \colon C(\XC\times \YC) \rightarrow C(\XC\times \YC), \tilde c \mapsto \left( (x,y) \mapsto \inf_{x' \in \XC} \tilde c(x',y) + \langle \nabla_x c(x',y), x-x'\rangle  + 2\sqrt{d}\norm{x-x'}^{\gamma}\right)
\end{align*}
Notably, it is crucial that the scalar product term involves the partial derivative of the respective (population) cost function $c$. 
Moreover, we like to point out that the image under $\Psi_{\textup{Hol}}^{c, \gamma}$ does not necessarily lead to $(\gamma,1)$-H\"older functions but nonetheless ensures suitable metric entropy bounds. 
\begin{prop}\label{prop:RegEl:Hol}
For the above setting with $\XC\subseteq \RR^d$ convex and compact, $\Psi = \Psi^{c, \gamma}_{\textup{Hol}}\circ \Psi_{\textup{bdd}}$ fulfills $\Psi(c) = c$, it is continuous near $c$, and it is Hadamard differentiable at $c$ with $D^H_{|c}\Psi= \Id_{C(\XC\times \YC)}$. Further, for any uniformly bounded function class $\GC$ on $\YC$ we obtain that  $$\sup_{\tilde c\in C(\XC\times \YC)}\log\NC(\eps, \GC^{\Psi(\tilde c)}, \norm{\cdot}_\infty) \lesssim \eps^{-d/\gamma}.$$
\end{prop}

\subsection{Combination of Regularity Elevations}\label{subsec:RegElCom}

Finally, we also outline a constructive way to combine regularity elevation maps defined on different spaces. This is important since it enables to leverage regularity properties of the population cost function for different regions of the domain. 

Hence, let $\XC, \YC$ be compact Polish spaces and assume existence of a collection of homeomorphisms $\zeta_i\colon \UC_i \rightarrow \zeta_i(\UC_i)$ for $1\leq i \leq I$ such that $\XC = \bigcup_{i = 1}^{I} \zeta_i(\UC_i)$. Further, assume there exists a partition on unity $\{\eta_i\}_{i=1}^{I}$ on $\XC$ with $\supp(\eta_i)\subseteq \zeta_i(\UC_i)$. Consider a continuous cost function $c\colon \XC\times \YC\rightarrow \RR$ and let $c_i\colon \UC_i \times \YC\rightarrow \RR, (u,y)\mapsto c(\zeta_i(u),y)$. Assume there exist maps $\Psi_i\colon C(\UC_i\times \YC)\rightarrow C(\UC_i\times \YC)$ such that $\Psi_i(c_i) = c_i$ and where $\Psi_i$ is continuous near $c_i$ and Hadamard differentiable at $c_i$ with derivative $D^H_{|c_i}\Psi_i= \Id$. Using these maps we define the combination of regularity elevations as\begin{align*}
  \Psi_{\textup{com}}\colon C(\XC\times \YC)\rightarrow C(\XC\times \YC), \quad \tilde c \mapsto \left((x,y)\mapsto \sum_{i = 1}^{I}\eta_i(x)\Psi_i\Big(\tilde c(\zeta_i(\cdot),\cdot)\Big)(\zeta_i^{-1}(x),y)\right).
\end{align*}
Indeed, by continuity of the partition of one $\eta_i$  as well as the functionals $\Psi_i$ and $\zeta_i$, $\zeta_i^{-1}$ for each $i\in \{1, \dots,I\}$ it follows that the range of this functional is indeed contained in $C(\XC\times \YC)$. 

\begin{prop}\label{prop:RegEl:Union}
For the above setting, $\Psi = \Psi_{\textup{com}}$ fulfills $\Psi(c) = c$, it is continuous near $c$, and it is Hadamard differentiable at $c$ with $D^H_{|c}\Psi= \Id_{C(\XC\times \YC)}$. Further, for any uniformly bounded function class $\GC$ on $\YC$ we obtain  $$\sup_{\tilde c\in C(\XC\times \YC)}\log\NC(\eps, \GC^{\Psi(\tilde c)}, \norm{\cdot}_\infty) \leq \sum_{i = 1}^{I} \sup_{\tilde c\in C(\XC\times \YC)}\log\NC(\eps, \GC^{\Psi_i(\tilde c(\zeta_i(\cdot),\cdot))}, \norm{\cdot}_\infty) .$$
\end{prop}


\section{Proofs of Main Results}
\label{sec: ProofMain}

In this section, we provide the full proofs of \Cref{lem:RelFCandConcave} for the dual representation of the OT value, \Cref{thm:AbstractMainResult} and \Cref{thm: BootConst} for the distributional limit of the empirical OT value under weakly converging costs, as well as Theorems \ref{thm:OTProcessCty}--\ref{thm:OTProcessSup} and \Cref{cor:OTInfimumNonCompactTheta} for empirical OT with extremal-type costs. The proofs for all auxiliary results of this subsection are deferred to \Cref{app:ProofSection2}. 

\subsection{Proof of Lemma~\ref{lem:RelFCandConcave}: Dual Representation of Optimal Transport Value}\label{subsubsec:RelFCandConcaveProof}
The subsequent auxiliary lemma establishes an important property of cost-transformations which is essential throughout this section. 
\begin{lem}[Lipschitz property of cost-transformation]\label{lem:CtrafoLipschitz}
	For arbitrary functions $f, \tilde f\colon \XC \to \RR$ and cost functions $c, \tilde c\colon \XC\times \YC\to \RR$, it follows that
	$\norm{f^c - \tilde f^{\tilde c}}_\infty \leq \norm{f- \tilde f}_\infty + \norm{c- \tilde c}_\infty$. In particular, upon selecting the constant functions $\tilde f, \tilde c \equiv 0$, it follows that $\norm{f^c}_\infty \leq \norm{f}_\infty + \norm{c}_\infty$. 
\end{lem}
\begin{proof}[Proof of \Cref{lem:RelFCandConcave}]
	For any $h\in \HC_c$ there exists $g\colon \YC\rightarrow [-\norm{  c}_\infty,\norm{  c}_\infty]$ with $h=g^{  c}$, and hence
	\begin{align*}
		- \norm{  c}_\infty - \sup_{y\in \YC}g(y) \leq h(x) = \inf_{y\in \YC}  c(x,y) - g(y) \leq \norm{  c}_\infty - \sup_{y\in \YC}g(y).
	\end{align*}
  In consequence, we find that $\norm{h}_\infty \leq 2\norm{c}_\infty \leq 2B$. 
	Further, for arbitrary $x,x'\in \XC$ and $\eps>0$, consider $y'\in\YC$  such that $h(x') \geq c(x', y') -  g( y') - \eps$. Then, it follows that 
	\begin{align*}
	h(x) -  h(x') =&\; \left[\inf_{y \in \YC} c(x,y) - g(y)\right] - \left[\inf_{y \in \YC}  c(x',y) - g(y)\right]\\
	\le& \; c(x, y') - g(y') - c(x', y') + g( y') + \eps\\
	\le& \; w(d_\XC(x,x')) + \eps. 
	\end{align*}
	Since $\eps>0$ can be chosen arbitrarily small, we obtain that $|h(x) - h(x')|\leq w(d_\XC(x,x'))$. This yields $\HC_c \subseteq \FC$ and thus $\HC^{  c}_{  c} \subseteq \FC^{  c}$. Further, by \citet[Proposition 1.34]{santambrogio2015optimal} we infer $\HC_c =\HC_c^{cc} \subseteq \FC^{  c  c}$. 
  To show  the remaining inclusions note for $f\in \FC$ that $$-\norm{  c}_\infty - \sup_{x\in \XC} f(x) \leq f^{  c} \leq \norm{  c}_\infty - \sup_{x\in \XC} f(x).$$ Hence, the function $g \coloneqq f^{  c}+\sup_{x\in \XC} f(x)$ fulfills $\norm{g}_\infty \leq \norm{  c}_\infty $, and since $\norm{f}_\infty \leq 2B$, we find that $$f^{  c   c}(x) = (f^{  c})^{  c} = (g)^{  c} +\sup_{x\in \XC} f(x) \in \HC_c + [-2B,2B],$$  
	which yields $ \FC^{  c  c} \subseteq  \HC_c  + [-2B,2B]$ as well as  $ \FC^{  c}=\FC^{  c  c  c} \subseteq  \HC_c^{  c}  + [-2B,2B]$.
  To show representation \eqref{eq:OTDualNice}, we combine the inclusions $\HC_c\subseteq \FC\subseteq C(\XC)$ with the alternative dual representations \eqref{eq: OTdual} and \eqref{eq:DualFormulationSpecificC}. 
	For the final claim, take a maximizing sequence $\{f_n\}_{n\in \NN}$ for \eqref{eq:OTDualNice} which admits by compactness of $\FC$ a converging subsequence $\{f_{n_k}\}_{k\in \NN}$ with uniform limit $f\in \FC$. Then by \Cref{lem:CtrafoLipschitz} it follows that $\{f_{n_k}^c\}_{k\in \NN}$ and $\{f_{n_k}^{cc}\}_{k\in \NN}$ also uniformly converge to $f^c$ and $f^{cc}$, respectively. We thus obtain that $\mu(f^{cc}) + \nu(f^c) = \lim_{k\rightarrow \infty} \mu(f^{cc}_{n_k}) + \nu(f^{c}_{n_k}) = OT(\mu, \nu,c)$ which shows that $f\in \FC$ is a maximizing element hence the set of optimizers $S_{\!c}(\mu,\nu)$ for \eqref{eq:OTDualNice} is non-empty. 
\end{proof}

\subsection{Proofs for Distributional Limits under Weakly Converging Costs}\label{subsec:ProofOfMainResult}

\subsubsection{Proof of Theorem~\ref{thm:AbstractMainResult}}\label{subsubsec:AbstractMainResultProof}

For the proof of \Cref{thm:AbstractMainResult} the following auxiliary results are crucial. We start with lower and upper bound on the difference between OT values for varying costs and probability measures which are a consequence of the OT problem having a representation in terms of an infimum over feasible couplings  as well as a supremum over feasible potentials. 

\begin{lem}[Lower and upper bounds]
\label{lem:LowerUpperBound}
Define for $B>0$ and a concave modulus of continuity $w\colon \RR_+\rightarrow \RR_+$ the collection \begin{equation*}C(B, w)\coloneqq \left\{ \bar c\in C(\XC\times \YC) \;\middle|\; \norm{\bar c}_\infty \leq B,  |\bar c(x,y) -\bar c(x',y)|\leq w(d_\XC(x,x')) \text{ for all }x,x'\in \XC, y\in \YC\right\}.\end{equation*}
Then, for costs $c, \tilde c \in C(B, w)$ and probability measures  $\mu, \tilde \mu\in \PC(\XC), \nu, \tilde \nu\in \PC(\YC)$ it holds that 
\begin{align*}
    &\inf_{\pi \in \Pi_{\tilde c}^\star(\tilde\mu,\tilde\nu)} \pi(\tilde c-c) + \sup_{f \in S\!_c(\mu,\nu)} (\tilde \mu-\mu ) f^{cc} +  (\tilde \nu-\nu ) f^c\\ %
\le& \;\;OT( \tilde \mu, \tilde \nu, {\tilde c}) -OT ( \mu, \nu,c)\\
\le& \min\bigg(\inf_{\pi \in \Pi_c^\star(\tilde\mu,\tilde\nu)} \pi (\tilde c-c) +\sup_{f \in S_{\!c}(\tilde \mu,\tilde \nu)} (\tilde \mu-\mu ) f^{cc} +  (\tilde \nu-\nu ) f^c,\\
&\inf_{\pi \in \Pi_c^\star(\mu,\nu)} \pi (\tilde c-c) +\sup_{f \in S_{\!\tilde c}(\tilde \mu,\tilde \nu)} (\tilde \mu-\mu ) f^{cc} +  (\tilde \nu-\nu ) f^c  + \sup_{f \in \FC}(\tilde \mu-\mu)(f^{\tilde c\tilde c} - f^{cc}) + (\tilde \nu-\nu)(f^{\tilde c} - f^{c}).\bigg)
\end{align*}
In particular, for fixed measures or fixed costs it follows that %
\begin{align*}
  |OT(\mu, \nu, {\tilde c}) -OT ( \mu, \nu, c)| \leq \norm{\tilde c - c}_\infty\!, \;\; |OT(\tilde \mu, \tilde\nu, {c}) -OT ( \mu, \nu,c)| \leq \!\!\sup_{f \in \FC^{cc}} \big|(\tilde \mu-\mu ) f\big| +  \!\!\sup_{f \in \FC^{c}} \big|(\tilde \nu-\nu ) f\big|. %
\end{align*}
\end{lem}

To employ the lower and upper bounds  of \Cref{lem:LowerUpperBound} for the proof of \Cref{thm:AbstractMainResult} we additionally require a number of continuity and measurability properties which are captured in the following lemma. Notably, we equip $\PC(\XC)\times \PC(\YC)$ with the bounded Lipschitz norm, which turns it into a Polish metric space and metrizes weak convergence of measures.

\begin{lem}[Continuity and Measurability]\label{lem:ContinuityResults}Let $\mu \in \PC(\XC), \nu \in \PC(\YC)$, and $c\in C(\XC\times \YC)$. 
Take a concave modulus of continuity $w\colon \RR_{+} \rightarrow \RR_{+}$ for $c$ and set $C\coloneqq C(2\norm{c}_\infty+1,2w)$ (for the definition of $C(2\norm{c}_\infty+1,2w)$ see \Cref{lem:LowerUpperBound}). Further, recall the function class $\FC= \FC(2\norm{c}_\infty+1,2w)$ introduced in \eqref{eq:defF} and define the functions
\begin{align}
  &T_1\colon \PC(\XC)\times \PC(\YC)\times \C\to \RR,\quad   \notag
  &&   (\mu', \nu', c')\mapsto OT( \mu', \nu',c'),\\
  &T_2\colon \PC(\XC)\times \PC(\YC)\times \C \times C(\XC\times \YC) \to \RR, \;\; \notag
  &&   (\mu', \nu', c', h_c)\mapsto \inf_{\pi \in \Pi_{ c'}^\star(\mu',\nu')} \pi(h_c),\\
  &T_3\colon  \PC(\XC)\times \PC(\YC)\times \C\times C_u(\FC)^2\to \RR,\quad  \notag 
   &&   (\mu', \nu', c', h_\mu, h_\nu)\mapsto\sup_{f \in S_{\!c'}(\mu',\nu')} h_\mu(f) +  h_\nu(f), \\ 
   &T_4 \colon C_u(\FC)^4\to \RR,\quad \notag
   &&  (h_\mu, \tilde h_\mu,h_\nu, \tilde h_\nu) \mapsto\sup_{f \in \FC} h_\mu(f)-\tilde h_\mu(f) +  h_\nu(f)-\tilde h_\nu(f).
\end{align}
Then, $T_1$ and $T_4$ are continuous, $T_2$ is lower semi-continuous, and $T_3$ is upper semi-continuous. If  $\Pi^\star_{c'}(\mu',\nu')$ is unique, $T_2$ is continuous at $(\mu', \nu', c', h_c)$. Moreover, for fixed $(\mu', \nu', c')$ the map $T_2$ is continuous in $h_c$ while $T_3$ is continuous in $(h_\mu, h_\nu)$. In particular, each function $T_i$ for  $1\leq i \leq 4$ is Borel measurable. 
\end{lem}

The previous two assertions fully deal with \emph{deterministic} statements on the OT functional and related terms that arise from corresponding bounds. The following two results provide the relevant tools to control the stochastic aspects. More precisely, for our proof of \Cref{thm:AbstractMainResult} we consider a Skorokhod representation of the random sequence detailed in \ref{ass:AThree} which additionally fulfills the property that $ \mu_n$ and $ \nu_n$ weakly converge to $\mu$ and $\nu$, respectively. For this purpose, we state the following measurability assertions and joint weak convergence statements. 

\begin{lem}[Measurability of empirical process]\label{lem:measurability}
  For a Polish space $\XC$ consider a  totally bounded function class $ \GC\subseteq C(\XC)$ under uniform norm. Then, the following assertions hold. 
  \begin{enumerate}[label=(\roman*)]
    \item  Any probability measure $\mu\in\PC(\XC)$ defines via evaluation a uniformly continuous functional on $ \GC$, i.e., $\mu \in C_u( \GC)$.
    \item A map $\omega \mapsto \mu(\omega) \in \PC(\XC)\subseteq C_u(\GC)$ is Borel measurable if and only if for any $g\in \GC$ the evaluation map $\omega \mapsto \mu(\omega)(g)$ is Borel measurable.
    \item The empirical process $\sqrt{n}(\mu_n - \mu)$ and the bootstrap empirical process $\sqrt{k}(\muboot - \mu_n)$ are both Borel measurable random variables in  $C_u(\GC)$. 
  \end{enumerate}
  \end{lem}

\begin{lem}[Joint weak convergence]\label{lem:JointWeakConvergence}For the setting of \Cref{thm:AbstractMainResult}, assume \ref{ass:AThree}. Then, for $n,m\rightarrow \infty$, weak convergence in the Polish space $C_u(\FC)^2 \times C(\XC\times \YC)\times\PC(\XC)\times \PC(\YC)$ to~a tight limit occurs
\begin{align}\label{eq:jointWeakConvergence1}&\left(\Big(\Gproc_n^\mu(f^{cc}), \Gproc_m^\nu(f^{c})\Big)_{f\in \FC}, \Gproc^c_{n,m},  \mu_n, \nu_m\right) \dto \left(\Big(\Gproc^\mu(f^{cc}),\Gproc^\nu(f^{c})\Big)_{f\in \FC}, \Gproc^c, \mu,\nu\right).
\end{align}
If \ref{ass:BTwo} of \Cref{thm:AbstractMainResult} is also valid, then, for $n,m\rightarrow \infty$, it follows in the Polish space $C_u(\FC)^4\times  C(\XC\times \YC) \times\PC(\XC)\times \PC(\YC)$ that 
 \begin{align}&\left(\Big(\Gproc_n^\mu(f^{cc}), \Gproc_n^\mu(f^{ c_{n,m} c_{n,m}}), \Gproc_m^\nu(f^{c}),  \Gproc_m^\nu(f^{ c_{n,m}})\Big)_{f\in \FC}, \Gproc^c_{n,m}, \mu_n, \nu_n,  \right)\notag \\
&\quad \quad \dto \left(\Big(\Gproc^\mu(f^{cc}),\Gproc^\mu(f^{cc}),\Gproc^\nu(f^{c}),\Gproc^\nu(f^{c})\Big)_{f\in \FC}, \Gproc^c,\mu,\nu,\right),\label{eq:jointWeakConvergence2}
\end{align}
Each sequence element for \eqref{eq:jointWeakConvergence1} and \eqref{eq:jointWeakConvergence2} as well as the weak limit are Borel measurable. 
\end{lem}

\begin{rmk}[Skorokhod representation]\label{rmk:SeparableSpace_SkorokhodWelldefined}
When dealing with weak convergence of empirical processes in non-separable spaces, special care is required due to potential measurability issues. However, since the different maps of interest are defined between Polish spaces and measurable, we circumvent such issues. In particular, since the random variables from \Cref{lem:JointWeakConvergence} converge weakly to a tight limit with separable support, the conditions of \citet[ Theorem~6.7]{billingsley1999convergence} are met and a (measurable) Skorokhod representation exists. 
\end{rmk}

With these tools at our disposal, we now proceed with the proof of \Cref{thm:AbstractMainResult}. 

\begin{proof}[Proof of \Cref{thm:AbstractMainResult}]
Invoking \Cref{thm:RegElModulus}, as  $\sqrt{nm/(n+m)}(c_{n,m} - c) =\colon \Gproc^{c}_{n,m}\dto \Gproc^c$ in the space $C(\XC\times \YC)$, there exists $\overline c_{n,m}$ such that the inclusion $\HC_{\overline c_{n,m}}\subseteq \FC^{\overline c_{n,m} \overline c_{n,m}}$ (recall the function classes from \Cref{subsec:preliminaries}) holds deterministically for $\FC = \FC(2\norm{c}_\infty+1, 2w)$ and $\sqrt{n}(c_{n,m} - \overline c_{n,m})\pto 0$. The latter implies by \Cref{lem:LowerUpperBound} that 
\[
\sqrt{\frac{nm}{n+m}}\left(OT( \mu_n, \nu_m, {\overline c_{n,m}}) - OT( \mu_n, \nu_m, { c_{n,m}})\right)\pto 0.
\]
Hence, to prove the assertion it suffices by Slutzky's lemma to show that 
\begin{align}\label{eq:desiredLimit}
  \sqrt{\frac{nm}{n+m}}\Big(OT( \mu_n, \nu_m, { \overline c_{n,m}})-OT( \mu, \nu,c)\Big) \dto \inf_{\pi \in \Pi_c^\star(\mu,\nu)} \!\!\pi(\Gproc^c) +\sup_{f \in S_{\!c}( \mu, \nu)}\!\! \sqrt{\lambda}\Gproc^\mu(f^{cc}) + \!\sqrt{1-\lambda} \Gproc^\nu(f^c).%
\end{align}
Without loss of generality, we may therefore assume $c_{n,m} = \overline c_{n,m}$. Further, set $\lambda_n \coloneqq m/(n+m)$.
Then, by \Cref{lem:LowerUpperBound}, the subsequent lower and upper bounds follow,  
\begin{align}
&\inf_{\pi \in \Pi_{ c_{n,m}}^\star(\mu_n,\nu_m)} \pi(\Gproc^c_{n,m}) + \sup_{f \in S_{\!c}(\mu,\nu)} \sqrt{\lambda_n}\; \Gproc^\mu_n(f^{cc}) +  \sqrt{1-\lambda_n}\;\Gproc^\nu_m(f^c)\notag\\
     &\leq\sqrt{\frac{nm}{n+m}}(OT( \mu_n, \nu_m,{ c_{n,m}})-OT( \mu, \nu,c))\label{eq:BoundsOT}\\
     &\leq \min\bigg(\inf_{\pi \in \Pi_c^\star(\mu_n,\nu_m)} \pi (\Gproc^c_{n,m}) +\sup_{f \in S_{\!c}( \mu_n, \nu_m)} \sqrt{\lambda_n}\; \Gproc^\mu_n(f^{cc}) + \sqrt{1-\lambda_n}\; \Gproc^\nu_m(f^c),\notag\\
& \quad\qquad\inf_{\pi \in \Pi_c^\star(\mu,\nu)} \pi(\Gproc^c_{n,m}) +\sup_{f \in S_{\!c_{n,m}}( \mu_n, \nu_m)} \sqrt{\lambda_n}\;\Gproc^\mu_n(f^{cc}) + \sqrt{1-\lambda_n}\; \Gproc^\nu_m(f^c)  \notag \\
&\quad\qquad+ \sup_{f \in \FC}\sqrt{\lambda_n}\; \left(\Gproc^\mu_n(f^{ c_{n,m}  c_{n,m}}) - \Gproc^\mu_n(f^{cc})\right)  +  \sqrt{1-\lambda_n}\;\left(\Gproc^\nu_m(f^{ c_{n,m}}) -   \Gproc^\nu_m(f^c)\right)\bigg).\notag
\end{align}
For each setting \ref{ass:BOne} and \ref{ass:BTwo} we show that the upper and lower bounds asymptotically converge in distribution to the limit in \eqref{eq:desiredLimit}, which then asserts that the empirical OT value also tends to this limit. To this end, we take for the random variables of \Cref{lem:JointWeakConvergence}  a Skorokhod representation on a probability space $(\Omega, \AC, P)$ \citep[p. 70]{billingsley1999convergence} which is well-defined by \Cref{rmk:SeparableSpace_SkorokhodWelldefined}. More precisely, under \ref{ass:BOne} we take the Skorokhod representation such that
\begin{align}\label{eq:SkorokhodB1}
    &\left(\Big(\tilde \Gproc_n^\mu(f^{cc}),\tilde  \Gproc_m^\nu(f^{c})\Big)_{f\in \FC}, \tilde \Gproc^c_{n,m}, \tilde  \mu_n, \tilde \nu_m\right) \as \left(\Big(\tilde \Gproc^\mu(f^{cc}),\tilde \Gproc^\nu(f^{c})\Big)_{f\in \FC}, \tilde \Gproc^c, \mu,\nu\right)
\end{align}
in $C_u(\FC)^2\times C(\XC\times \YC)\times\PC(\XC)\times \PC(\YC) $, whereas under \ref{ass:BTwo} we choose it such that
\begin{align}&\left(\Big(\tilde \Gproc_n^\mu(f^{cc}), \tilde \Gproc_n^\mu(f^{\tilde c_{n,m} \tilde c_{n,m}}), \tilde \Gproc_m^\nu(f^{c}),  \tilde \Gproc_m^\nu(f^{\tilde c_n})\Big)_{f\in \FC},\tilde \Gproc^c_{n,m}, \tilde \mu_n, \tilde\nu_m\right)\notag \\
&\quad \quad \as \left(\Big(\tilde \Gproc^\mu(f^{cc}),\tilde \Gproc^\mu(f^{cc}),\tilde \Gproc^\nu(f^{c}),\tilde \Gproc^\nu(f^{c})\Big)_{f\in \FC}, \tilde \Gproc^c, \mu,\nu\right)\label{eq:SkorokhodB2}
\end{align}
in $C_u(\FC)^4\times C(\XC\times \YC)\times\PC(\XC)\times \PC(\YC)$. We also set $\tilde c_{n,m} \coloneqq c + \tilde \Gproc_{n,m}^c/\sqrt{nm/(n+m)}$ which a.s.\ converges to $c$.

For the subsequent argument recall the functions $T_1, T_2, T_3,T_4$ from \Cref{lem:ContinuityResults} and their (semi-) continuity properties. Furthermore, note that an application of \Cref{lem:CtrafoLipschitz} in combination with the arguments of the proof of \Cref{lem:measurability} $(i)$ yields that the maps $\FC\to\RR$,
\[f\mapsto  \Gproc^\mu_n(f^{cc}),\quad   f\mapsto  \Gproc^\mu_n(f^{c_{n,m}c_{n,m}}), \quad f\mapsto \Gproc^\nu_m(f^{c}), \quad  f\mapsto  \Gproc^\nu_m(f^{c_{n,m}})\]
are uniformly continuous, i.e., elements in $C_u(\FC)$. For both settings \ref{ass:BOne} and \ref{ass:BTwo} it follows by measurability of the maps $T_1,T_2, T_3$ for each $n,m\in \NN$ that 
\begin{align*}
   & \sqrt{\frac{nm}{n+m}}(OT( \mu_n,  \nu_m,{ c_{n,m}}) - OT(\mu, \nu,c)) \eqd \sqrt{\frac{nm}{n+m}}(OT(\tilde \mu_n, \tilde \nu_m, {\tilde c_{n,m}}) - OT(\mu, \nu,c))\\
 &  \inf_{\pi \in \Pi_{ c_{n,m}}^\star(\mu_n,\nu_m)} \pi(\Gproc^c_{n,m}) + \sup_{f \in S\!_c(\mu,\nu)} \sqrt{\lambda_n}\;\Gproc^\mu_n(f^{cc}) + \sqrt{1-\lambda_n}\; \Gproc^\nu_m(f^c) \\
  & \qquad \qquad \eqd \inf_{\pi \in \Pi_{\tilde c_{n,m}}^\star(\tilde\mu_n,\tilde \nu_m)} \pi(\tilde \Gproc^c_{n,m}) + \sup_{f \in S\!_c(\mu,\nu)}\sqrt{\lambda_n}\; \tilde \Gproc^\mu_n(f^{cc}) + \sqrt{1-\lambda_n}\; \tilde \Gproc^\nu_m(f^c).
\intertext{Under \ref{ass:BOne}~we also notice that}
 &   \inf_{\pi \in \Pi_c^\star(\mu_n,\nu_m)} \pi (\Gproc^c_{n,m}) +\sup_{f \in S_{\!c}( \mu_n, \nu_m)}  \sqrt{\lambda_n}\;\Gproc^\mu_n(f^{cc}) +  \sqrt{1-\lambda_n}\;\Gproc^\nu_m(f^c) \\
 & \qquad \qquad   \eqd \inf_{\pi \in \Pi_c^\star(\tilde \mu_n,\tilde\nu_m)} \pi (\tilde\Gproc^c_{n,m}) +\sup_{f \in S_{\!c}(\tilde \mu_n,\tilde \nu_m)}  \sqrt{\lambda_n}\;\tilde\Gproc^\mu_n(f^{cc}) +  \sqrt{1-\lambda_n}\;\tilde\Gproc^\nu_m(f^c),
\end{align*}
whereas under \ref{ass:BTwo} we additionally employ measurability of $T_4$ to infer for each $n \in \NN$ that
\begin{align*}
    &\inf_{\pi \in \Pi_c^\star(\mu,\nu)} \pi(\Gproc^c_{n,m}) +\sup_{f \in S_{\!c_{n,m}}( \mu_n, \nu_m)} \sqrt{\lambda_n}\;\Gproc^\mu_n(f^{cc}) +  \sqrt{1-\lambda_n}\;\Gproc^\nu_m(f^c) \\
    & + \sup_{f \in \FC} \sqrt{\lambda_n}\;\left(\Gproc^\mu_n(f^{cc}) - \Gproc^\mu_n(f^{ c_{n,m}  c_{n,m}}) \right) +  \sqrt{1-\lambda_n}\;\left(\Gproc^\nu_m(f^c) -   \Gproc^\nu_m(f^{ c_{n,m}})\right)\\
  & \qquad \qquad  \eqd    \inf_{\pi \in \Pi_c^\star(\mu,\nu)} \pi(\tilde \Gproc^c_{n,m}) +\sup_{f \in S_{\!\tilde  c_{n,m}}(\tilde  \mu_n,\tilde  \nu_m)}\sqrt{\lambda_n}\; \tilde \Gproc^\mu_n(f^{cc}) + \sqrt{1-\lambda_n}\; \tilde \Gproc^\nu_m(f^c)  \\
  &  \qquad \qquad + \sup_{f \in \FC} \sqrt{\lambda_n}\;\left(\tilde \Gproc^\mu_n(f^{cc}) - \tilde \Gproc^\mu_n(f^{\tilde c_{n,m} \tilde c_{n,m}})\right)  + \sqrt{1-\lambda_n}\;\left( \tilde \Gproc^\nu_m(f^c) -   \tilde \Gproc^\nu_m(f^{\tilde c_{n,m}})\right).
\end{align*}
 Hence, it suffices to work with the Skorokhod representation to obtain the weak limit for the empirical OT value. 
Invoking \Cref{lem:LowerUpperBound}, identical lower and upper bounds on the quantity of interest, $\sqrt{nm/(n+m)}(OT(\tilde \mu_n, \tilde \nu_m, {\tilde c_{n,m}}) - OT(\mu, \nu, c))$, as for \eqref{eq:BoundsOT} can be concluded.

To obtain a suitable bound on the limit inferior of $\sqrt{nm/(n+m)}(OT(\tilde \mu_n, \tilde \nu_m, {\tilde c_{n,m}}) - OT(\mu, \nu, c))$ take for both \ref{ass:BOne} and \ref{ass:BTwo} a measurable set $A\in \AC$ of full measure such that the convergence from \eqref{eq:SkorokhodB1} and \eqref{eq:SkorokhodB2} is fulfilled thereon, respectively. Then, for each $\omega\in A$ it follows by lower semi-continuity of $T_2$ jointly with continuity of $T_3$ under fixed $(\mu, \nu, c)$ that 
\begin{align*}
    &\liminf_{n,m \rightarrow \infty} \inf_{\pi \in \Pi_{\tilde c_{n,m}}^\star(\tilde\mu_n,\tilde \nu_m)} \pi(\tilde \Gproc^c_{n,m}) + \sup_{f \in S\!_c(\mu,\nu)} \sqrt{\lambda_n}\; \tilde \Gproc^\mu_n(f^{cc}) + \sqrt{1-\lambda_n}\; \tilde \Gproc^\nu_m(f^c)\\ 
   & \qquad\geq \inf_{\pi \in \Pi_{c}^\star(\mu,\nu)} \pi(\tilde \Gproc^c) + \sup_{f \in S\!_c(\mu,\nu)} \sqrt{\lambda}\; \tilde \Gproc^\mu(f^{cc}) + \sqrt{1-\lambda}\; \tilde \Gproc^\nu(f^c).
\end{align*}

Under \ref{ass:BOne}, we find for each $\omega\in A$ by continuity of $T_2$ at  $(\mu, \nu, c, \Gproc^c_{n,m})$ as a consequence of \ref{ass:BOne} and upper semi-continuity of $T_3$ that 
\begin{align*}
    &\limsup_{n,m \rightarrow \infty}\inf_{\pi \in \Pi_c^\star(\tilde \mu_n,\tilde \nu_m)} \pi (\tilde \Gproc^c_{n,m}) + \sup_{f \in S_{\!c}(\tilde \mu_n,\tilde \nu_m)}\sqrt{\lambda_n}\;  \tilde\Gproc^\mu_n(f^{cc}) +  \sqrt{1-\lambda_n}\;\tilde\Gproc^\nu_m(f^c)\\
   & \qquad \leq \;\;\pi^\star(\tilde \Gproc^c) + \sup_{f \in S\!_c(\mu,\nu)} \sqrt{\lambda}\;\tilde \Gproc^\mu(f^{cc}) +\sqrt{1-\lambda}\;  \tilde \Gproc^\nu(f^c)\\
   & \qquad =  \inf_{\pi \in \Pi_{c}^\star(\mu,\nu)} \pi(\tilde \Gproc^c) + \sup_{f \in S\!_c(\mu,\nu)}\sqrt{\lambda}\; \tilde \Gproc^\mu(f^{cc}) + \sqrt{1-\lambda}\; \tilde \Gproc^\nu(f^c).
\end{align*}

Under \ref{ass:BTwo}, we note for each $\omega\in A$ by continuity of $T_2$ in $h_c$ for fixed $(\mu, \nu, c)$, upper semi-continuity of $T_3$ and continuity of $T_4$ that 
\begin{align*}
    &\limsup_{n,m\rightarrow \infty} \!\inf_{\pi \in \Pi_c^\star(\mu,\nu)} \!\!\!\!\pi(\tilde \Gproc^c_{n,m}) +\!\!\!\!\sup_{f \in S_{\!\tilde  c_{n,m}}(\tilde  \mu_n,\tilde  \nu_m)} \!\!\!\!\!\!\sqrt{\lambda_n}\;\tilde \Gproc^\mu_n(f^{cc}) + \sqrt{1-\lambda_n}\; \tilde \Gproc^\nu_m(f^c)  \\
    &+ \sup_{f \in \FC} \sqrt{\lambda_n}\; \left(\tilde \Gproc^\mu_n(f^{cc}) - \tilde \Gproc^\mu_n(f^{ c_{n,m}  c_{n,m}}) \right)  + \sqrt{1-\lambda_n}\; \left(\tilde \Gproc^\nu_m(f^c) -   \tilde \Gproc^\nu_m(f^{ c_{n,m}})\right)\\
     & \qquad \leq  \inf_{\pi \in \Pi_c^\star(\mu,\nu)} \pi(\tilde \Gproc^c) +\sup_{f \in S_{\!c}(\mu,\nu)} \sqrt{\lambda}\;\tilde \Gproc^\mu(f^{cc}) + \sqrt{1-\lambda}\; \tilde \Gproc^\nu(f^c)  \\
     &\qquad + \sup_{f \in \FC }\sqrt{\lambda}\; \left(\tilde \Gproc^\mu(f^{cc}) - \tilde \Gproc^\mu(f^{c c}) \right) 
     + \sqrt{1-\lambda}\;\left( \tilde \Gproc^\nu(f^c) -   \tilde \Gproc^\nu(f^{c})\right)\\
     &\qquad =  \inf_{\pi \in \Pi_c^\star(\mu,\nu)} \pi(\tilde \Gproc^c) +\sup_{f \in S_{\!c}(\mu,\nu)} \sqrt{\lambda}\;\tilde \Gproc^\mu(f^{cc}) + \sqrt{1-\lambda}\; \tilde \Gproc^\nu(f^c). 
\end{align*}

As the lower bound and the upper bounds for $\sqrt{nm/(n+m)}(OT(\tilde \mu_n, \tilde \nu_m, {\tilde c_{n,m}}) - OT(\mu, \nu,c))$ asymptotically match for all $\omega \in A$, it follows  under both \ref{ass:BOne} and \ref{ass:BTwo} that 
\[
\lim_{n,m\rightarrow \infty} \!\sqrt{\frac{nm}{n+m}}(OT(\tilde \mu_n, \tilde \nu_m, {\tilde c_{n,m}}) - OT(\mu, \nu, c)) = \!\! \inf_{\pi \in \Pi_c^\star(\mu,\nu)} \pi(\tilde \Gproc^c) +\!\!\sup_{f \in S_{\!c}(\mu,\nu)} \!\!\!\!\sqrt{\lambda}\; \tilde \Gproc^\mu(f^{cc}) +\!  \sqrt{1-\lambda}\;\tilde \Gproc^\nu(f^c).
\]
As the set $A$ has full measure we obtain that
\[\sqrt{\frac{nm}{n+m}}(OT(\tilde \mu_n, \tilde \nu_m, {\tilde c_{n,m}}) - OT(\mu, \nu, c)) \as \inf_{\pi \in \Pi_c^\star(\mu,\nu)} \pi(\tilde \Gproc^c) +\sup_{f \in S_{\!c}(\mu,\nu)}  \!\!\!\!\sqrt{\lambda}\;\tilde \Gproc^\mu(f^{cc}) +  \sqrt{1-\lambda}\; \tilde \Gproc^\nu(f^c),
\]
where the limit has by measurability of $T_2$ and $T_3$ the same Borel law as the limit in the assertion, which finishes the proof. 
\end{proof}

\subsubsection{Proof of \Cref{thm: BootConst}}\label{subsubsec:BootConstProof}
Before turning to the proof of the bootstrap consistency, i.e., \Cref{thm: BootConst}, we introduce an important result on the convergence of the bootstrap empirical measure.
\begin{lem}
\label{lem: asConvEmpMeas}
For  a Polish space $\XC$ let $\mu \in \PC(\XC)$. Consider i.i.d. random variables $\{X_i\}_{i =1}^{n}\sim \mu^{\otimes n}$ to define the empirical measure $\mu_n = n^{-1}\sum_{i=1}^{n}\delta_{X_i}$. 
Further, consider $k(n)$ i.i.d. random variables $\{\Xboot{k}\}_{i =1}^{n}\sim \mu_n^{\otimes k(n)}$ to define the bootstrap empirical measure $\muboot :=\frac{1}{k(n)}  \sum_{j=1}^{k(n)}\delta_{\Xib}$. 
Then, provided that $k(n) \to \infty$ as $n \to \infty$, it follows under $n\rightarrow \infty$ that $
\muboot$ weakly converges to $\mu$, in probability. 
\end{lem}

The above lemma is a corollary of Theorem~2 in \citep{beran1987convergence} and was added to ease further referencing. We can now prove \Cref{thm: BootConst} on bootstrap consistency under weakly converging costs.

\begin{proof}[Proof of \Cref{thm: BootConst}] 
By Assumptions \ref{ass:AThree} and \ref{ass:AThreeStar}, and by measurability of the empirical and bootstrap empirical processes (\Cref{lem:measurability}) we infer using Lemma~2.2$(c)\Rightarrow(a)$ in \citet{bucher2019note} for two bootstrap versions $(\mu_{n,k}^{(1)},
\nu_{n,k}^{(1)},c_{n,k}^{(1)})$, $(\mu_{n,k}^{(2)},\nu_{n,k}^{(2)},c_{n,k}^{(2)})$ based on independent bootstrap samples $\{X_i^{(1)}\}_{i=1}^k, \{X_i^{(2)}\}_{i=1}^k\sim\mu_n^{\otimes k}$ and $\{Y_i^{(1)}\}_{i=1}^k, \{Y_i^{(2)}\}_{i=1}^k\sim\nu_n^{\otimes k}$ for  $n, k \to \infty, k=\oh(n)$ that
\[
\begin{pmatrix}
\sqrt{n}
\begin{pmatrix}
\mu_n - \mu \\
\nu_n - \nu\\
c_n - c
\end{pmatrix}\hspace{1cm}\\
\sqrt{k}
\begin{pmatrix}
\mu_{n,k}^{(i)} - \mu_n \\
\nu_{n,k}^{(i)} - \nu_n \\
 c_{n,k}^{(i)} - c_n
\end{pmatrix}_{i = 1,2}
\end{pmatrix}
\dto 
\begin{pmatrix}
\begin{pmatrix}
\Gproc^\mu \\
\Gproc^\nu \\
\Gproc^c \\
\end{pmatrix}\hspace{0.3cm}\\
\begin{pmatrix}
\Gproc^{\mu,(i)} \\
\Gproc^{\nu,(i)} \\
\Gproc^{c,(i)} \\
\end{pmatrix}_{i = 1,2}
\end{pmatrix}
\]
in $\Big( C_u( \FC^{cc}) \times C_u(\FC^{c})\times C(\XC\times \YC)\Big)^3$. Since $k = \oh(n)$ we also obtain by Slutzky's lemma that 
\[
\begin{pmatrix}
\sqrt{n}
\begin{pmatrix}
\mu_n - \mu \\
\nu_n - \nu\\
c_n - c
\end{pmatrix}\hspace{1cm}\\
\sqrt{k}
\begin{pmatrix}
\mu_{n,k}^{(i)} - \mu \\
\nu_{n,k}^{(i)} - \nu \\
 c_{n,k}^{(i)} - c
\end{pmatrix}_{i = 1,2}
\end{pmatrix}
\eqqcolon 
\begin{pmatrix}
\begin{pmatrix}
\Gproc_n^\mu \\
\Gproc_n^\nu\\
\Gproc_n^c
\end{pmatrix}\hspace{0.3cm}\\
\begin{pmatrix}
\Gproc_{n,k}^{\mu, (i)} \\
\Gproc_{n,k}^{\nu, (i)} \\
\Gproc_{n,k}^{c,   (i)}
\end{pmatrix}_{i = 1,2}
\end{pmatrix}
\dto 
\begin{pmatrix}
\begin{pmatrix}
\Gproc^\mu \\
\Gproc^\nu \\
\Gproc^c \\
\end{pmatrix}\hspace{0.3cm}\\
\begin{pmatrix}
\Gproc^{\mu,(i)} \\
\Gproc^{\nu,(i)} \\
\Gproc^{c,(i)} \\
\end{pmatrix}_{i = 1,2}
\end{pmatrix}.
\]
Herein, the triples $(\Gproc^\mu,\Gproc^\nu,\Gproc^c)$, $(\Gproc^{\mu,(1)},\Gproc^{\nu,(1)},\Gproc^{c,(1)})$, and $(\Gproc^{\mu,(2)},\Gproc^{\nu,(2)},\Gproc^{c,(2)})$ are independent and have identical law. 
Notably, invoking \Cref{thm:RegElModulus} we may assume without loss of generality that the empirical and bootstrap cost function $c_n$ and $c_{n,k}^{(i)}$ for $i \in \{1,2\}$ deterministically satisfy the relation $\FC_{\bar c}\subseteq \FC^{\bar c\bar c}$, $\bar c\in\{c_n,c_{n,k}^{(i)}\}$. 
Moreover, by \citet{varadarajan58} we know that $\mu_n \dto \mu, \nu_n \dto \nu$ a.s. for $n\rightarrow \infty$, and by \Cref{lem: asConvEmpMeas} we infer for $i \in \{1,2\}$ that $\mu_{n,k}^{(i)}\dto \mu, \nu_{n,k}^{(i)}\dto \nu$ in probability for $n,k \rightarrow \infty, k = \oh(n)$.  Hence, Slutzky's lemma asserts that 
\begin{align}
\begin{pmatrix}
\begin{pmatrix}
\Gproc_n^\mu, 
\Gproc_n^\nu,
\Gproc_n^c,
\mu_n,
\nu_n
\end{pmatrix}^T\hspace{0.3cm}\\
\begin{pmatrix}
\Gproc_{n,k}^{\mu, (i)},
\Gproc_{n,k}^{\nu, (i)},
\Gproc_{n,k}^{c,   (i)},
\mu_{n,k}^{(i)},
\nu_{n,k}^{(i)}
\end{pmatrix}^T_{i = 1,2}
\end{pmatrix}
\dto 
\begin{pmatrix}
\begin{pmatrix}
\Gproc^\mu,
\Gproc^\nu,
\Gproc^c,
\mu,
\nu
\end{pmatrix}^T\hspace{0.3cm}\\
\begin{pmatrix}
\Gproc^{\mu,(i)},
\Gproc^{\nu,(i)},
\Gproc^{c,(i)},
\mu,
\nu
\end{pmatrix}^T_{i = 1,2}
\end{pmatrix}
\end{align}
in  $\big( C_u( \FC^{cc}) \times C_u(\FC^{c})\times C(\XC\times \YC)\times \PC(\XC)\times \PC(\YC)\big)^3$.
Moreover, using an analogous argument as for the proof of \Cref{lem:JointWeakConvergence} we conclude that 
\begin{align}\label{eq:JointBootstrapConvergence1}&
\begin{pmatrix}
	\left(\Big(\Gproc_n^\mu(f^{cc}), \Gproc_n^\nu(f^{c})\Big)_{f\in \FC}, \Gproc^c_n,  \mu_n, \nu_n\right)^T\\
	\left(\Big(\Gproc_{n,k}^{\mu,(i)}(f^{cc}), \Gproc_{n,k}^{\nu,(i)}(f^{c})\Big)_{f\in \FC}, \Gproc^{c,(i)}_{n,k},  \mu_{n,k}^{(i)}, \nu_{n,k}^{(i)}\right)^T_{i = 1,2}
\end{pmatrix}
 \dto 
 \begin{pmatrix}
 \left(\Big(\Gproc^\mu(f^{cc}),\Gproc^\nu(f^{c})\Big)_{f\in \FC}, \Gproc^c, \mu,\nu\right)^T\\
 \left(\Big(\Gproc^{\mu,(i)}(f^{cc}),\Gproc^{\nu,(i)}(f^{c})\Big)_{f\in \FC}, \Gproc^{c,(i)}, \mu,\nu\right)^T_{i = 1,2}
 \end{pmatrix}
\end{align}
in the Polish space $\big( C_u( \FC)^2\times C(\XC\times \YC)\times \PC(\XC)\times \PC(\YC)\big)^3$, and we use under Assumption \ref{ass:BOne} a Skorokhod representation for the process in \eqref{eq:JointBootstrapConvergence1}. 

Under Assumptions \ref{ass:BTwo} and \ref{ass:BTwoStar}, by measurability of $c_{n}$ and $c_{n,k}$ as maps to $C(\XC\times \YC)$, Lipschitzianity under $c$-transformation (\Cref{lem:CtrafoLipschitz}) and Slutzky's lemma we conclude weak convergence of the random variables 
\begin{align}
\begin{pmatrix}
	\left(\Big(\Gproc_n^\mu(f^{cc}), \Gproc_n^\mu(f^{c_nc_n}), \Gproc_n^\nu(f^{c}), \Gproc_n^\nu(f^{c_n})\Big)_{f\in \FC}, \Gproc^c_n,  \mu_n, \nu_n\right)^T\\
	\left(\Big(\Gproc_{n,k}^{\mu,(i)}(f^{cc}),\Gproc_{n,k}^{\mu,(i)}(f^{c^{(i)}_{n,k}}), \Gproc_{n,k}^{\nu,(i)}(f^{c}), \Gproc_{n,k}^{\nu,(i)}(f^{c_{n,k}^{(i)}})     \Big)_{f\in \FC}, \Gproc^{c,(i)}_{n,k},  \mu_{n,k}^{(i)}, \nu_{n,k}^{(i)}\right)^T_{i = 1,2}
\end{pmatrix}& \notag
 \\
 \dto 
 \begin{pmatrix}
 \left(\Big(\Gproc^\mu(f^{cc}),\Gproc^\mu(f^{cc}),\Gproc^\nu(f^{c}),\Gproc^\nu(f^{c})\Big)_{f\in \FC}, \Gproc^c, \mu,\nu\right)^T\\
 \left(\Big(\Gproc^{\mu,(i)}(f^{cc}),\Gproc^{\mu,(i)}(f^{cc}),\Gproc^{\nu,(i)}(f^{c}),\Gproc^{\nu,(i)}(f^{c})\Big)_{f\in \FC}, \Gproc^{c,(i)}, \mu,\nu\right)^T_{i = 1,2}
 \end{pmatrix}& \label{eq:JointBootstrapConvergence2}
\end{align}
in the Polish space $\big( C_u( \FC)^4\times C(\XC\times \YC)\times \PC(\XC)\times \PC(\YC)\big)^3$. For the random variables from \eqref{eq:JointBootstrapConvergence2} we now take a Skorokhod representation. 

To denote the random elements from the Skorokhod representation, we equip to the respective random variable with a tilde, e.g., we write $\tilde \mu_n$ for the representation of $\mu_n$. 
Following the same proof technique as \Cref{thm:AbstractMainResult} we thus conclude with \Cref{lem:LowerUpperBound} and \Cref{lem:ContinuityResults} that 
\begin{align*}
	&\begin{pmatrix}
\sqrt{n}\Big(OT( \tilde \mu_{n},  \tilde \nu_{n}, { \tilde c_{n}}) - OT(\mu, \nu, c)\Big)\\
\sqrt{k}\Big(OT( \tilde \mu^{(i)}_{n,k},  \tilde \nu^{(i)}_{n,k}, { \tilde c^{(i)}_{n,k}}) - OT(\mu, \nu, c)\Big)_{i = 1,2}
\end{pmatrix}\\
&\as
\begin{pmatrix}
\inf_{\pi \in \Pi_{c}^\star(\mu,\nu)}\pi(\tilde\Gproc^c) + \sup_{f \in S\!_c(\mu,\nu)} \tilde\Gproc^\mu(f^{cc})+ \tilde\Gproc^\nu(f^c)\\
\Big(\inf_{\pi \in \Pi_{c}^\star(\mu,\nu)}\pi(\tilde\Gproc^{c,(i)}) + \sup_{f \in S\!_c(\mu,\nu)}\tilde\Gproc^{\mu,(i)}(f^{cc}) + \tilde\Gproc^{\nu,(i)}(f^c)\Big)_{i = 1,2}
\end{pmatrix}.
\end{align*}
Consequently, we infer for the original random variables and using that $k = \oh(n)$ that 
\begin{align*}
&\begin{pmatrix}
\sqrt{n}\Big(OT(  \mu_{n},   \nu_{n}, {  c_{n}}) - OT(\mu, \nu, c)\Big)\\
\sqrt{k}\Big(OT(  \mu^{(i)}_{n,k},   \nu^{(i)}_{n,k}, {  c^{(i)}_{n,k}}) - OT(  \mu_{n},   \nu_{n}, {  c_{n}})\Big)_{i = 1,2}
\end{pmatrix}\\
&\dto
\begin{pmatrix}
\inf_{\pi \in \Pi_{c}^\star(\mu,\nu)}\pi(\Gproc^c) + \sup_{f \in S\!_c(\mu,\nu)} \Gproc^\mu(f^{cc}) + \Gproc^\nu(f^c)\\
\Big(\inf_{\pi \in \Pi_{c}^\star(\mu,\nu)}\pi(\Gproc^{c,(i)}) + \sup_{f \in S\!_c(\mu,\nu)}\Gproc^{\mu,(i)}(f^{cc})\! + \Gproc^{\nu,(i)}(f^c)\Big)_{i = 1,2}
\end{pmatrix}.
\end{align*}
 Since the three components in the limit have identical distributions and are independent, the assertion follows at once from \citet[Lemma 2.2 $(a)\Rightarrow (c)$]{bucher2019note}. 
\end{proof}

\subsection{Proofs for Distributional Limits under Extremal-Type Costs}
\label{subsec: ProofExtremalCosts}
Before we proceed with the proofs for the results from \Cref{subsec:OTProcessResults} which rely on an application of the functional delta method, we provide a simple result on the support of the limiting processes. Its proof is deferred to \Cref{subsec:SupportLimitingProcessProof}.
\begin{lem}\label{lem:SupportLimitingProcess}
	For a Polish space $\XC$ let $\mu \in \PC(\XC)$ and consider a bounded, measurable function class $\tilde \FC$  on $\XC$. 
	Then, the following assertions hold.
	\begin{enumerate}
		\item[$(i)$] The contingent cone of $\PC(\XC)$ at $\mu$ is given by $T_{\mu}\PC(\XC) = \closure\{\frac{\mu' - \mu}{t}| t > 0, \mu'\in \PC(\XC)\}\subseteq \ell^\infty(\tilde \FC)$. 
		\item[$(ii)$] For any $\Delta\in T_{\mu}\PC(\XC)$ and  $f,f'\in \tilde \FC$ with $f - f' \equiv \kappa$ for some $\kappa\in \RR$ it holds that $\Delta(f) = \Delta(f')$.
		\item[$(iii)$] If $\tilde \FC$ is $\mu$-Donsker, then the tight limit $\Gproc^\mu$ of the empirical process $\sqrt{n}(\mu_n - \mu)$ in $\ell^\infty(\tilde \FC)$ is a.s. contained in $T_{\mu}\PC(\XC)$. 
	\end{enumerate}	
\end{lem}

\subsubsection{Proof of \Cref{thm:OTProcessCty}}\label{subsubsec:OTProcessCtyProof}
The result follows by an application of the functional delta method \citep{Roemisch04}. Without loss of generality, we assume that $\XC = \supp(\mu)$ and $\YC= \supp(\nu)$. This ensures that Kantorovich potentials are by \ref{ass:KPU} unique on the full domains $\XC$ and $\YC$. 
Assumption \ref{ass:Don} in conjunction with  independence of the underlying random variables from $\mu$ and $\nu$ ensure by \citet[Example 1.4.6]{van1996weak} that the joint process $\sqrt{nm/n+m}(\mu_n -\mu,\nu_m -\nu)$ weakly converge in $\linfty(\cup_{\theta\in \Theta}\FC^{c_\theta c_\theta})\times \linfty( \cup_{\theta\in \Theta}\FC^{c_\theta})$. Further, by \Cref{lem:SupportLimitingProcess} the limit is a.s. contained in $T_\mu\PC(\XC)\times T_\nu\PC(\YC)$.
 It remains to show that the  map 
\begin{align*}
  (OT(\cdot,\cdot, {c_\theta}))_{\theta \in \Theta}\colon &\PC(\XC)\times \PC(\YC)\subseteq \ell^\infty(\cup_{\theta\in \Theta}\FC^{c_\theta c_\theta})\times \ell^\infty(\cup_{\theta\in \Theta}\FC^{c_\theta}) \rightarrow C(\Theta), \\&(\mu, \nu) \mapsto \left(\theta\mapsto \sup_{f\in \FC} \mu(f^{c_\theta,c_\theta}) + \nu(f^{c_\theta})  \right)
\end{align*}
is Hadamard directionally differentiable at $(\mu, \nu)$ tangentially to $\PC(\XC)\times \PC(\YC)$. In the language of \Cref{thm:DiffPsi}, take $\FC$ and $\Theta$ as they are and set 
\begin{align*}
V \coloneqq \linfty(\cup_{\theta\in \Theta} \FC^{c_\theta c_\theta})\times \linfty( \cup_{\theta\in \Theta}\FC^{c_\theta}),\quad U \coloneqq  \PC(\XC)\times \PC(\YC), \quad E((\mu, \nu),f, \theta )\coloneqq  \mu(f^{c_\theta c_\theta}) + \nu(f^{c_\theta}).
\end{align*}
Then, Assumption \textbf{(EC)} follows from \Cref{lem:CtrafoLipschitz}, while \textbf{(Lip)}  and \textbf{(Lin)} are simple to verify by definition of $V$ and $E$. Moreover, by Assumption~\ref{ass:KPU} the condition of point \textit{(ii)} in \Cref{lem:LSC_con} holds, since the evaluations of $E$ in $f$ with  $(\Delta^\mu,\Delta^\nu)\in T_\mu\PC(\XC)\times T_\nu\PC(\YC)$ are invariant under constant shifts (\Cref{lem:SupportLimitingProcess}), and since Kantorovich potentials are unique on $\XC$ and $\YC$ up to a constant shift. This establishes \textbf{(DC)}, and the proof is complete. \hfill\qed

\subsubsection{Proof of Theorem~\ref{thm:OTProcessInf}}\label{subsubsec:OTProcessInfProof}
Since $\Theta$ is a compact Polish space, it follows by \citet[Lemma S.4.9]{fang2019inference} (see also \citealt[Corollary 2.3]{carcamo2020directional}) that the infimal mapping, $$ I\colon C(\Theta) \rightarrow \RR, \quad h\mapsto \inf_{\theta\in \Theta} h(\theta),$$
is Hadamard directionally differentiable at  $OT(\mu, \nu, c(\cdot)) \in C(\Theta)$ with derivative given by 
\begin{align*}
	D^H_{OT(\mu, \nu, c(\cdot))} I \colon C(\Theta) \rightarrow \RR, \quad \Delta^h \mapsto  \inf_{\theta\in S_{\!-}( \Theta, \mu , \nu)}\Delta^h(\theta). 
\end{align*}
Hence, applying the functional delta method \citep{Roemisch04} for the infimal mapping $I$ onto the uniform weak limit for the empirical OT process from \Cref{thm:OTProcessCty} asserts the claim. \qed

\subsubsection{Proof of Theorem~\ref{thm:OTProcessSup}}\label{subsubsec:OTProcessSupProof}
From the dual formulation \eqref{eq:OTDualNice} the supremal OT value over $\Theta$ is given by 
\begin{align*}
  \sup_{\theta\in \Theta} OT({\cdot,\cdot,c_\theta})\colon &\PC(\XC)\times \PC(\YC)\subseteq \ell^\infty(\cup_{\theta\in \Theta}\FC^{c_\theta c_\theta})\times \ell^\infty(\cup_{\theta\in \Theta}\FC^{c_\theta}) \rightarrow \RR,\\  &(\mu, \nu) \mapsto \!\!\!\!\! \sup_{(f, \theta)\in \FC\times \Theta} \!\!\!\!\!\mu(f^{c_\theta c_\theta}) + \nu(f^{c_\theta}).
\end{align*}
The results of Appendix~\ref{ap: Had-Diff} readily apply, with the choices for $V$, $U$, and $E$ as in the proof of \Cref{thm:OTProcessCty}; the only difference being that the supremum is taken over  $\FC\times \Theta$ instead of~$\FC$. In particular, \textbf{(EC)}, \textbf{(Lip)}, and \textbf{(Lin)} are valid, whereas \textbf{(DC)} is now trivially fulfilled. 
Overall, \Cref{thm:DiffPsi} asserts that $\sup_{\theta\in \Theta} OT({\cdot,\cdot,c_\theta})$ is Hadamard directionally differentiable tangentially to $\PC(\XC)\times \PC(\YC)$ with derivative 
\begin{align*}
	D^H_{|(\mu, \nu)} \sup_{\theta\in \Theta} OT({\cdot,\cdot,c_\theta})\colon T_\mu\PC(\XC)\times T_\nu\PC(\YC)\rightarrow \RR, \quad (\Delta^\mu, \Delta^\nu)\mapsto \sup_{\substack{\theta\in S_{\!+}(\Theta, \mu, \nu)\\ f_\theta\in S_{\!c_\theta}(\mu, \nu)}} \Delta^\mu(f_\theta^{c_\theta c_\theta})+ \Delta^\nu(f_\theta^{c_\theta}).
\end{align*}
Combined with weak convergence of $\sqrt{nm/n+m}(\mu_n -\mu,\nu_m -\nu)$  in $\linfty(\cup_{\theta\in \Theta}\FC^{c_\theta c_\theta})\times \linfty( \cup_{\theta\in \Theta}\FC^{c_\theta})$ by \ref{ass:Don} in conjunction with the independence of the underlying samples \citep[Example 1.4.6]{van1996weak}, and the inclusion of the limit in $T_\mu\PC(\XC)\times T_\nu\PC(\YC)$ by \Cref{lem:SupportLimitingProcess}, the functional delta method \citep{Roemisch04} implies the claim. \qed

\begin{rmk}In addition to the proof presented above, it is also possible to show \Cref{thm:OTProcessSup} with similar arguments to those found in the proof of \citet[Lemma S.4.9]{fang2019inference} or \citet[Corollary 2.3]{carcamo2020directional}. However, their statements only provide sufficient conditions for Hadamard directional differentiability for tangentially to the space $C(\cup_{\theta\in \Theta}\FC^{c_\theta c_\theta})\times C(\cup_{\theta\in \Theta}\FC^{c_\theta})$, whereas the supremal OT value is defined only on the strict subset $\PC(\XC)\times \PC(\YC)$.	
\end{rmk}

\subsubsection{Proof of \Cref{cor:OTInfimumNonCompactTheta}}\label{subsubsec:ProofOfCorollaryNonCompactStuff}
  Define $\Delta(\mu_n, \nu_m,K)\coloneqq \inf_{\theta\in \Theta} OT( \mu_n, \nu_m,c_\theta) -  \inf_{\theta\in K} OT( \mu_n, \nu_m,c_\theta)$. Then, 
  \begin{align*}
    &\pOut \big(\Delta(\mu_n, \nu_m,K) \neq 0 \big) \le\pOut \big( \big\{ \Delta(\mu_n, \nu_m,K) \neq 0 \big\} \cap \big\{\theta_{n,m} \in K \big\} \big) + \pOut (\theta_{n,m} \notin K), 
  \end{align*}
  and as the first summand in the above display is null while $\lim_{n\rightarrow\infty}\pOut (\theta_{n,m} \notin K)= 0$, the right-hand side converges to zero. Hence, invoking Slutzky's Lemma \citep[Example 1.4.7]{van1996weak} it follows from \Cref{thm:OTProcessInf} that 
  \begin{align*}
    &\sqrt{\frac{nm}{n+m}}\left(\inf_{\theta\in \Theta} OT( \mu_n, \nu_m,c_\theta) -  \inf_{\theta\in \Theta} OT( \mu, \nu,c_\theta) \right) \\
    =\;&\sqrt{\frac{nm}{n+m}} \Delta(\mu_n, \nu_m,K) +  \sqrt{\frac{nm}{n+m}} \left(\inf_{\theta\in K} OT( \mu_n, \nu_m,c_\theta) -  \inf_{\theta\in K} OT( \mu, \nu,c_\theta) \right)\\
    \dto \;&0 + \inf_{\theta\in S_{\!-}(K, \mu, \nu)} \sqrt{\lambda}\Gproc^\mu(f_\theta^{c_\theta c_\theta})+ \sqrt{1-\lambda}\Gproc^\nu(f_\theta^{c_\theta}).
  \end{align*}
  The claim now follows at once after observing that $S_{\!-}(K, \mu, \nu) = S_{\!-}(\Theta, \mu, \nu)$.\qed

\begin{ackno}
  S. Hundrieser and C.A.\ Weitkamp gratefully acknowledge support from the DFG Research
  Training Group 2088 \textit{Discovering structure in complex data: Statistics meets optimization and inverse problems}. G. Mordant gratefully acknowledges support from the DFG CRC 1456 \textit{Mathematics of the Experiment
  A04} and A. Munk gratefully acknowledges support from the DFG CRC 1456 \textit{A04, C06} and the Cluster of Excellence 2067 \textit{MBExC Multiscale bioimaging–from molecular machines
  to networks of excitable cells}.
  \end{ackno}


\addcontentsline{toc}{section}{References}
\printbibliography
\newpage
\appendix


\section[Uniform Hadamard Directional Differentiability of Extremal-Type Functionals]{Uniform Hadamard Directional Differentiability of Extremal-\\Type Functionals}
\label{subsec:HadDiff}
\label{ap: Had-Diff}
A number of results in this work rely on the notion of Hadamard directional differentiability and the functional delta method. More precisely, both the result on the weak convergence of the empirical OT process from \Cref{subsec:OTProcessResults} and the formulation of regularity elevation functionals from \Cref{subsec:RegEl} rely on this approach. 
Although, these two findings are conceptually rather unrelated, their proof techniques are based on a more general insight which we lay out in this section. 

Let $(V, \norm{\cdot}_V)$ be a normed vector space  and consider sets $\FC$ and $\Theta$. Additionally, consider a real-valued function $E\colon V\times \FC\times \Theta\rightarrow \RR$ which assigns each triple $(v,f,\theta)$ to a some objective value $E(v,f,\theta)$. We are interested in sensitivity results for extremal-type functionals $$\Psi(v) \coloneqq \left( \sup_{f\in \FC} E(v,f,\theta) \right)_{\theta\in \Theta} \quad \text{ and } \qquad \tilde \Psi(v) \coloneqq \left( \inf_{f\in \FC} E(v,f,\theta)\right)_{\theta\in \Theta}.$$
Herein, $\Theta$ provides the collection of feasible parameters which affect the optimization problem, while $\FC$ represents the collection of feasible solutions. The space $V$ denotes another set of parameters that determine the optimization problem and exhibit a vector space structure. 
Overall, these optimization problems characterize the general structure of processes indexed over $\Theta$ which are pointwise defined as the supremum or infimum over a collection $\FC$ and depend on some parameter in $V$ with an additive structure. 

For our sensitivity analysis under perturbations of $v$ it suffices to focus only on $\Psi$ since 
$$\inf_{f\in \FC} E(v,f,\theta) = - \sup_{f\in \FC} -E(v,f,\theta)\quad \text{ for any }(v, \theta)\in V\times \Theta.$$
 In the following, we first establish sufficient conditions in terms of $E$ for the continuity properties of $\Psi$ and the underlying sets of optimizers. 

\begin{lem}[Continuity]\label{lem:CtyPsi}
	Let $(V, \norm{\cdot}_V)$ be a normed vector space, consider compact topological spaces $\FC$ and $\Theta$ whose topologies are generated by (pseudo-)metrics $d_\FC$ and $d_\Theta$, respectively, and assume that $E\colon V\times \FC\times \Theta\rightarrow \RR$ satisfies the following. 
	\begin{enumerate}
		\item[\textbf{(EC)}] For any $v\in V$ the functional $E(v, \cdot, \cdot)\colon \FC\times \Theta\rightarrow \RR$ is continuous.
		\item[\textbf{(Lip)}] There exists some $L\geq 0$ such that for any $(f,\theta)\in \FC\times \Theta$ the functional $E(\cdot,f,\theta)\colon V\rightarrow \RR$ is $L$-Lipschitz with respect to $\norm{\cdot}_V$.  
	\end{enumerate}
	Then, $\Range(\Psi)\subseteq C(\Theta)$ and the functional $\Psi \colon V \rightarrow C(\Theta)$ is $L$-Lipschitz.
	Further, for any $(v, \theta)\in V\times \Theta$ the set of optimizers $S(v, \theta)\coloneqq\{f\in \FC | \sup_{f'\in \FC} E(v,f',\theta)  = E(v,f,\theta)\}$ is non-empty, and for fixed $v \in V$ the set-valued map 
  $$(\theta, t) \in \Theta\times \RR_{+} \mapsto S(v, \theta;t)\coloneqq\left\{f\in \FC \;\Big|\; \sup_{f'\in \FC} E(v,f',\theta) \leq E(v,f,\theta) +t\right\}$$
	 is upper semi-continuous in terms of inclusion, i.e., for $\theta_n\rightarrow \theta$ and $t_n \rightarrow t$ any sequence $f_n\in S(\overline v, \theta_n;t_n)$ admits a converging subsequence $(f_{n_k})_{k\in \NN}$ in $\FC$ with limit $f \in S(\overline v, \theta;t)$.
\end{lem}

\begin{proof}[Proof of \Cref{lem:CtyPsi}]
	By Assumption \textbf{(EC)} and compactness of $\Theta\times\FC$ it follows for any $v \in V$ that $E(v, \cdot, \cdot)$ is uniformly continuous, hence the function $$w_{E,v}\colon \RR_{+}\rightarrow \RR_{+}, \quad t \mapsto \sup_{\substack{d_\Theta(\theta,\theta') \leq t\\ d_\FC(f,f')\leq t}} |E(v,f,\theta) - E(v,f',\theta')|$$
	is finite for all $t\geq 0$ and fulfills $\lim_{t\searrow 0} w_{E,v}(t) = 0$. 
	For $\theta, \theta' \in \Theta$ we thus find that \begin{align*}
		\left|\sup_{f\in \FC} E(v,f,\theta) - \sup_{f\in \FC} E(v,f,\theta')\right|\leq \sup_{f\in \FC}|E(v,f,\theta) - E(v,f,\theta')|\leq w_{E,v}(d_\Theta(\theta,\theta')),
	\end{align*}
	which implies for $v\in V$ that $\Psi(v)\in C(\Theta)$ and therefore $\Range(\Psi)\subseteq C(\Theta)$. For the Lipschitzianity of $\Psi$, note by Assumption \textbf{(Lip)} for any $v, v'\in V$ that 
	\begin{align*}
		\norm{\Psi(v) - \Psi(v')}_{C(\Theta)}&= \sup_{\theta\in \Theta}\left|\sup_{f\in \FC} E(v,f,\theta) - \sup_{f\in \FC} E(v',f,\theta)\right|\\
		&\leq \sup_{\substack{\theta\in \Theta\\f\in \FC}}|E(v,f,\theta) - E(v',f,\theta)|\leq L \norm{v - v'}_V.
	\end{align*}
	
	To see that $S(v, \theta) \neq \emptyset$, note that  the function $E(v, \cdot, \theta)\colon \FC\rightarrow \RR$ is continuous for any $(v, \theta)\in V\times \Theta$; hence, by compactness of $\FC$ the supremum over $\FC$ is attained. 
	
	It remains to prove the assertion on upper semi-continuity. Consider converging sequences $t_n \rightarrow t\geq 0$ and $\theta_n\rightarrow \theta\in \Theta$ and take a sequence $f_n\in S(v, \theta_n ; t_n)$. By compactness of $\FC$ a converging subsequence $(f_{n_k})_{k \in \NN}$ exists with limit $f \in \FC$. Hence, by Assumption \textbf{(EC)} and since $\sup_{f\in \FC} E(v,f,\cdot)= \Psi(v)(\cdot) \in C(\Theta)$ we obtain that $f\in S(v, \theta; t)$ since
\begin{align*}
	E(v,f, \theta) + t= \lim_{k \rightarrow \infty} E(v,f_{n_k},\theta_{n_k}) + t_{n_k}&\geq \lim_{k\rightarrow \infty} \sup_{f\in \FC} E(v,f,\theta_{n_k})  = \sup_{f\in \FC} E(v,f,\theta). \qedhere
\end{align*}
\end{proof}

With these tools at our disposal, we can state our general sensitivity result. 

\begin{thm}[Differentiability]\label{thm:DiffPsi}
	Assume in the setting of \Cref{lem:CtyPsi} Conditions \textbf{(EC)} and \textbf{(Lip)}. Let $\overline v\in V$ and consider a convex set $U\subset V$. Denote by $T_{\overline v}U \coloneqq \closure\{ \frac{v-\overline v}{t} \mid t>0, v\in U\}\subseteq V$ its contingent cone at $\overline v$. Further, assume the following:  
	\begin{enumerate}
		\item[\textbf{(Lin)}] For any $(f,\theta)\in \FC\times \Theta$  the function $\Delta_{|\overline v} E(\cdot, f, \theta) \colon V\rightarrow \RR, v\mapsto E(\overline v+v,f,\theta)-E(\overline v,f,\theta)$ is~linear. 
		\item[\textbf{(DC)}] For any  $h\in T_{\overline v}U$ the function $\theta\in \Theta\mapsto \sup_{f\in S(\overline v,\theta)}\Delta_{|\overline v} E(h, f, \theta)$ is lower semi-continuous. 
	\end{enumerate}	
	Then, the functional $$\Psi \colon V \rightarrow C(\Theta), \quad v \mapsto \left( \sup_{f\in \FC} E(v,f,\theta) \right)_{\theta\in \Theta}$$
	is Hadamard directionally differentiable   at $\overline v$ tangentially to $U$ with derivative given by 
	$$ D^H_{|\overline v} \Psi \colon T_{\overline v}U \rightarrow C(\Theta), \quad h \mapsto \left( \sup_{f\in S(\overline v,\theta)} \Delta_{|\overline v} E(h, f, \theta)  \right)_{\theta\in \Theta}.$$
\end{thm}

\Cref{thm:DiffPsi} can be viewed as an extension of \citet[Lemma S.4.9]{fang2019inference} to a uniform perturbation result over the parameter space $\Theta$. Additionally, our result does not require regularity properties on the full domain $V$ but only a convex set $U$, an appealing property which we exploit in the context of our analysis for the OT process (where we choose $U = \PC(\XC)\times \PC(\YC)$) as well as regularity elevations (see proof of \Cref{prop:RegEl:Mod}). 

	Assumptions \textbf{(EC)}, \textbf{(Lip)}, and \textbf{(Lin)} are fairly straightforward and often simple to verify. 
	The first two conditions also appear to be necessary to infer that  $\Range(\Psi)\subseteq C(\Theta)$ and Lipschitzianity of $\Psi\colon V \rightarrow C(\Theta)$. 
	Assumption \textbf{(DC)} is more technical and requires some knowledge on the set of optimizers $S(\overline v, \theta)$. As the proof of \Cref{thm:DiffPsi} reveals, is the functional $\theta\in \Theta\mapsto \sup_{f\in S(\overline v,\theta)} E(h,f,\theta)$ under the assumptions of \Cref{lem:CtyPsi} always upper semi-continuous. Hence, the sole purpose of \textbf{(DC)} is to ensure $\Range(D^H_{|\overline v}\Psi)\subseteq C(\Theta)$. Sufficient conditions for its validity are stated as follows.

\begin{lem}\label{lem:LSC_con}
	Assume the setting of \Cref{lem:CtyPsi} and \Cref{thm:DiffPsi}. Then under either of the following conditions Assumption \textbf{(DC)} of \Cref{thm:DiffPsi} is fulfilled. %
	\begin{enumerate}
	\item[$(i)$] For any $\theta\in \Theta$ and $h\in T_{\overline v}U$ there exists $f\in S(\overline v, \theta)$ with $\sup_{f'\in S(\overline v,\theta)}  \Delta_{|\overline v} E(h, f', \theta) = \Delta_{|\overline v} E(h, f, \theta)$ such that any converging sequence $\theta_{n}\rightarrow \theta$ admits a sub-sequence $(\theta_{n_k})$ and a converging sequence $f_{n_k}\in S(\overline v, \theta_{n_k})$ with $f_{n_k}\rightarrow f$ in $\FC$. 
	\item[$(ii)$] For any $\theta\in \Theta$ and $h\in T_{\overline v}U$ it holds that $  \Delta_{|\overline v} E(h, f, \theta)=  \Delta_{|\overline v} E(h, f', \theta)$ for any $f,f' \in S(\overline v, \theta)$.
	\end{enumerate}
\end{lem}

\begin{proof}[Proof of \Cref{lem:LSC_con}]
	Let $\theta_n \rightarrow \theta$ and consider an element $f\in S(\overline v, \theta)$ such that  $\Delta_{|\overline v} E(h, f, \theta) = \sup_{f'\in S(\overline v,\theta)}  \Delta_{|\overline v} E(h, f, \theta)$. 
	For setting $(i)$ take an arbitrary subsequence $\theta_{n_k}$ and take another subsequence $\theta_{n_{k_l}}$  such that $f_{n_{k_l}}\in S(\overline v, \theta_{n_{k_l}})$ converges to $f$ for $l\rightarrow \infty$. 
	Then, by \textbf{(EC)},  
\begin{align*}
  \lim_{l\rightarrow \infty}\Delta_{|\overline v} E(h, f_{n_{k_l}}, \theta_{n_{k_l}}) &=  
  \Delta_{|\overline v} E(h, f, \theta) = \sup_{f'\in S(\overline v, \theta)} \Delta_{|\overline v} E(h,f', \theta).
\end{align*}
This implies by monotonicity of the limit inferior and \Cref{lem:limsup} that 
\begin{align*}
  \liminf_{n \rightarrow \infty}  \sup_{f'\in S(\overline v, \theta)} \Delta_{|\overline v} E(h,f', \theta_n) \geq \liminf_{n \rightarrow \infty} \Delta_{|\overline v} E(h, f_{n}, \theta_{n}) \geq \sup_{f'\in S(\overline v, \theta)} \Delta_{|\overline v} E(h,f', \theta),
\end{align*} 
which asserts the validity of  Assumption \textbf{(DC)} of \Cref{thm:DiffPsi}. 
	For setting $(ii)$ take $f_n\in S(\overline v, \theta_n)$ and consider by \Cref{lem:CtyPsi} a converging subsequence $f_{n_k}$ with limit $f\in S(\overline v, \theta)$. Hence, it holds that $\Delta_{|\overline v} E(h, f, \theta)= \sup_{f'\in S(\overline v, \theta)}  \Delta_{|\overline v} E(h, f', \theta)$ and the assertion follows from $(i)$. 
\end{proof}

\begin{proof}[Proof of \Cref{thm:DiffPsi}]

The proof strategy is inspired by \citet{Roemisch04} who performs a sensitivity analysis for when $\Theta$ is a singleton. To extend the claim for a compact topological space $\Theta$ we employ the subsequent version of Dini's theorem. 

\begin{lem}[Dini's Theorem, {\citealt[Corollary 1]{toma1997strong}}]\label{prop:DiniTheorem}
Let $\Theta$ be a compact topological space and consider  a decreasing $f_n\colon \Theta \rightarrow \RR$ sequence (i.e., $f_{n} \geq f_{n+1}$ for all $n\in \NN$) of  upper semi-continuous functions. Further, assume that $f_n$ pointwise converges to a (lower semi-)continuous function $f\colon \Theta\rightarrow \RR$. Then, $f_n$ converges to $f$  uniformly on $\Theta$. 
\end{lem}

Take a positive null sequence $t_n\searrow 0$ with $t_n>0$ for all $n\in \NN$ and let $h\in T_{\overline v}U$. Further, take a sequence $h_n \in V$ such that $v_n \coloneqq \overline v + t_n h_n\in U$ for all $n \in \NN$ and $h_n \rightarrow h$ in $V$. 
For any $\theta\in \Theta$, we then observe by \textbf{(Lin)} and \textbf{(Lip)} for any $n\in \NN$ the lower bound 
\begin{align}
	\Psi(v_n)(\theta) - \Psi(\overline v)(\theta) & = \sup_{f \in \FC} E(v_n,f, \theta) - \sup_{f \in \FC} E(\overline v,f, \theta)\notag\\
	&\geq \sup_{f\in S(\overline v, \theta)} E(v_n, f, \theta) - E(\overline v, f, \theta)\notag\\
  &\geq \sup_{f\in S(\overline v, \theta)} \Delta_{|\overline v}E(t_n h_n, f, \theta)\notag\\
	& \geq t_n \sup_{f\in S(\overline v, \theta)} \Delta_{|\overline v}E(h,f,\theta) - 2 t_n L \norm{h - h_n}_V.\label{eq:lowerboundPsi_good}
\end{align}
Analogously, we obtain the upper bound
\begin{align}
	\Psi(v_n)(\theta) - \Psi(\overline v)(\theta) & = \sup_{f \in \FC} E(v_n,f, \theta) - \sup_{f \in \FC} E(\overline v,f, \theta)\notag\\
	&\leq \sup_{f\in S(v_n, \theta)} E(v_n, f, \theta) - E(\overline v, f, \theta)\notag\\
	& \leq t_n \sup_{f\in S(v_n, \theta)} \Delta_{|\overline v}E(h,f,\theta) + 2t_n L \norm{h-h_n}_V. \label{eq:upperboundPsi}
\end{align}
Note that $S(v_n, \theta)\subseteq S(\overline v, \theta ; 2L \norm{v_n- v}_V)$ since any $f^*\in S(v_n, \theta)$ fulfills by \textbf{(Lip)} the bound
\begin{align*}
	E(v,f^*, \theta) &\geq E(v_n, f^*, \theta) - L\norm{v_n - v}_V\\
	& = \sup_{f \in \FC} E(v_n, f, \theta) - L\norm{v_n - v}_V\\
	& \geq \sup_{f\in \FC} E(v, f, \theta) - 2L\norm{v_n - v}_V.
\end{align*}
Hence, it follows from \eqref{eq:upperboundPsi} upon defining $\eps_n \coloneqq \sup_{k\geq n}2L \norm{v_k- v}_V$ that 
\begin{align}
	\Psi(v_n)(\theta) - \Psi(\overline v)(\theta) &\leq t_n \sup_{f\in S(\overline v, \theta ; 2L \norm{v_n- v}_V )}\Delta_{|\overline v}E(h,f,\theta)+ 2t_n L \norm{h-h_n}_V\notag \\
	&\leq t_n \sup_{f\in S(\overline v, \theta ;\eps_n)} \Delta_{|\overline v}E(h,f,\theta) +2t_n L \norm{h-h_n}_V.\label{eq:upperboundPsi_good}
\end{align}
Combining \eqref{eq:lowerboundPsi_good} and \eqref{eq:upperboundPsi_good} we thus obtain for any $\theta\in \Theta$ that 
\begin{align*}\sup_{f\in S(\overline v, \theta)} \Delta_{|\overline v}E(h,f,\theta) - 2L \norm{h - h_n}_V
\leq \frac{\Psi(v_n)(\theta) - \Psi(\overline v)(\theta)}{t_n}&\leq \sup_{f\in S(\overline v, \theta;\eps_n)} \Delta_{|\overline v}E(h,f,\theta) + 2L \norm{h-h_n}_V.
\end{align*}
To conclude the claim we show that the lower and upper bound uniformly converge on $\Theta$ for $n\rightarrow \infty$ to the $D^H_{|\overline v} \Psi$. Since $\norm{h_n - h}_V\rightarrow \infty$, it suffices to prove for the functions 
\begin{align*}
	\Phi \coloneqq D^H_{|\overline v} \Psi \colon &  \Theta \rightarrow \RR, \;\; \theta \mapsto \sup_{f\in S(\overline v, \theta)} \Delta_{|\overline v}E(h,f,\theta),\qquad 	\Phi_n\colon  \Theta \rightarrow \RR, \;\;\theta \mapsto \sup_{f\in S(\overline v, \theta,\eps_n )}\Delta_{|\overline v}E(h,f,\theta),
\end{align*}
that $\lim_{n\rightarrow\infty}\norm{\Phi - \Phi_n}_{C(\Theta)}= 0$. For this purpose, we employ Dini's theorem (\Cref{prop:DiniTheorem}).

In this context note, since $(\eps_n)_{n \in \NN}$ is a decreasing null-sequence, for all $n \in \NN$ and any $\theta\in \Theta$ that $S(\overline v, \theta)\subseteq S(\overline v, \theta;\eps_{n+1}) \subseteq S(\overline v, \theta;\eps_{n})$
and consequently 
\begin{align}\label{eq:PhiDecreasing}
\Phi(\theta)\leq \Phi_{n+1}(\theta) \leq \Phi_n(\theta)\leq 2\sup_{\theta\in \Theta}\sup_{f\in\FC} E(h,f,\theta)<\infty,
\end{align}
where the upper bound is finite due to  Assumption \textbf{(EC)} and compactness of $\FC\times \Theta$. 

Further, let us show for any $\theta\in \Theta$ that $\lim_{n\rightarrow \infty} \Phi_n(\theta) = \Phi(\theta)$. Take a sequence $f_n\in S(\overline v, \theta; \eps_n)$ such that $\Phi_n(\theta) \leq \Delta_{|\overline v} E(h, f_n, \theta) + 1/n$. Consider a converging subsequence $(f_{n_k})_{k \in \NN}$ with limit $f_\infty\in  S(\overline v, \theta)$. Then, by \textbf{(EC)} it follows that 
\begin{align*}
  \limsup_{k\rightarrow\infty}\Phi_{n_k}(\theta)  \leq \lim_{k\rightarrow \infty}\Delta_{|\overline v} E(h, f_{n_{k}}, \theta) + 1/n_k = \Delta_{|\overline v} E(h, f_{\infty}, \theta) \leq  \sup_{f \in S(\overline v, \theta)}\Delta_{|\overline v}E(v,f,\theta)= \Phi(\theta).
\end{align*}
Recalling \eqref{eq:PhiDecreasing}, it thus follows that $\lim_{n \rightarrow \infty} \Phi_n(\theta) = \Phi(\theta)$. 

To conclude the assertion with Dini's theorem it remains to show upper-continuity of $\Phi_n$ and of $\Phi$; recall by Assumption \textbf{(DC)} that $\Phi$ is already lower semi-continuous. To this end, let $\eps\geq 0$ and consider a converging sequence $\theta_n \rightarrow \theta$. Select $f_n\in S(\overline v, \theta_n, \eps)$ such that $\sup_{f\in S(\overline v, \theta_n, \eps)}\Delta_{|\overline v}E(h,f, \theta_n) \leq \Delta_{|\overline v}E(h, f_n, \theta_n) + 1/n$. 
Take a subsequence $f_{n_k}$ and select by \Cref{lem:CtyPsi} another converging subsequence $f_{n_{k_l}}$ with limit $f_\infty\in S(\overline v, \theta; \eps)$.  Using Assumption \textbf{(EC)} it thus follows that
\begin{align*}
  \lim_{l\rightarrow\infty} \Delta_{|\overline v}E(h, f_{n_{k_l}}, \theta_{n_{k_l}}) + 1/n_{k_l} = \Delta_{|\overline v}E(h, f_\infty, \theta) \leq \sup_{f\in S(\overline v, \theta; \eps)}\Delta_{|\overline v}E(h,f, \theta)
\end{align*}
Invoking monotonicity of the limit superior and \Cref{lem:limsup} we thus obtain that 
\begin{align*}
  \limsup_{n\rightarrow\infty} \sup_{f\in S(\overline v, \theta_n; \eps)}\Delta_{|\overline v}E(h,f, \theta) \leq \limsup_{l\rightarrow\infty} \Delta_{|\overline v}E(h, f_{n}, \theta_{n}) + 1/n \leq \sup_{f\in S(\overline v, \theta; \eps)}\Delta_{|\overline v}E(h,f, \theta), 
\end{align*}
Hence, by \Cref{lem:limsup} we conclude that $\Phi_n$ is upper semi-continuous and that $\Phi$ is continuous. Dini's theorem (\Cref{prop:DiniTheorem}) thus implies $\lim_{n\rightarrow\infty}\norm{\Phi - \Phi_n}_{\infty}= 0$, asserting the Hadamard directional differentiability of $\Psi$ at $\overline v$ tangentially to $U$. Finally, note that the range of $D_{\overline v}^H \Psi$ is indeed contained in $C(\Theta)$. 
\end{proof}

\section{Proofs for Section \ref{subsec:Assumptions}: Sufficient Criteria for Assumptions}\label{app:SufficientCriteriaProofs}

\subsection{Proof of Proposition \ref{prop:DonskerProperty}} 
By \Cref{lem:RelFCandConcave} it follows that $\FC^{c}\subseteq \HC_c^c + [-2B, 2B]$ and $\FC^{cc}\subseteq \HC_c + [-2B, 2B]$ with $\HC_c$ defined in \eqref{eq:C-concaveFunctionsDef}.
Invoking \citet[Lemma 2.1]{hundrieser2022empirical} and \citet[Proposition~1.34]{santambrogio2015optimal} we obtain for any $\eps>0$ that 
\begin{align*}
  \NC(\eps, \FC^{c}, \norm{\cdot}_\infty) =   \NC(\eps, \FC^{cc}, \norm{\cdot}_\infty) \leq \left\lceil\frac{2B}{\eps}\right\rceil\NC(\eps/2, \HC^{c}_c, \norm{\cdot}_\infty) =  \left\lceil\frac{2B}{\eps}\right\rceil \NC(\eps/2, \HC_{c}, \norm{\cdot}_\infty).
\end{align*}
For the function class $\HC_c$, the asserted uniform metric entropy bounds are available in Section 3.1 and Appendix A of \citet{hundrieser2022empirical}. Note by uniform boundedness of the cost function that $\HC_c$ and $\HC_c^c$ are uniformly bounded. The assertion on the universal Donsker property then follows from \citet[Theorem 2.5.6]{van1996weak}. 
\hfill\qed

\subsection{Proof of Proposition \ref{prop: A3StarEx_NonPar}} 
By assumption the functional 
\begin{align*}
  \overline \Phi_c \colon \PC(\XC)\times \PC(\YC) &\rightarrow \ell^\infty(\FC^{cc})\times \ell^\infty(\FC^c)\times C(\XC\times \YC)\\
  (\mu, \nu) &\mapsto (\mu, \nu, \Phi_c(\mu, \nu)), 
\end{align*}
where the domain is viewed as a subset of $\ell^{\infty}(\FC_\XC\cup \FC^{cc})\times \ell^{\infty}(\FC_\YC\cup \FC^{c})$, 
is Hadamard differentiable at $(\mu, \nu)$ tangentially to $\PC(\XC)\times \PC(\YC)$. Moreover, since $\FC_\XC\cup \FC^{cc}$ is $\mu$-Donsker it follows that $\sqrt{n/2}(\mu_n - \mu) \dto \Gproc^\mu$  in $\ell^\infty(\FC_\XC\cup \FC^{cc})$. Likewise, since $\FC_\YC\cup \FC^{c}$ is $\nu$-Donsker it follows that $\sqrt{n/2}(\nu_n - \nu) \dto \Gproc^\nu$  in $\ell^\infty(\FC_\YC\cup \FC^{c})$. Further, by independence of the random variables $\{X_i\}_{i= 1}^{n}$ and $\{Y_i\}_{i = 1}^{n}$ it follows from \cite[Theorem 1.4.6]{van1996weak} that the joint empirical processes $\sqrt{n/2}(\mu_n - \mu, \nu_n - \nu)$ weakly converge in $\ell^{\infty}(\FC_\XC\cup \FC^{cc})\times \ell^{\infty}(\FC_\YC\cup \FC^{c})$ to $(\Gproc^\mu, \Gproc^\nu)$,  contained in $T_{\mu}\PC(\XC)\times T_{\nu}\PC(\YC)$ by \Cref{lem:SupportLimitingProcess}. 
We thus conclude by the functional delta method \citep{Roemisch04} for $\overline \Phi_c$ that \ref{ass:AThree} is fulfilled. 

Moreover, by the Donsker property and independence of the random variables, we also infer by \citet[Theorem 3.6.13]{van1996weak} in the space $\ell^{\infty}(\FC_\XC\cup \FC^{cc})\times \ell^{\infty}(\FC_\YC\cup \FC^{c})$ that 
\begin{align*}
  d_{BL}\left( \Law\left(\sqrt{k} \begin{pmatrix}
\muboot - \mu_n\\
\nuboot - \nu_n\\
\end{pmatrix} \middle| X_1, \dots, X_n, Y_1, \dots Y_n\right),
\Law\begin{pmatrix}
  \Gproc^\mu\\
  \Gproc^\nu
\end{pmatrix}\right)\ptoOut 0.
\end{align*}
Hence, by the functional delta method for conditionally weakly converging random variables \citet{dumbgen1993nondifferentiable}  for $\overline \Psi_c$ we infer  that  \begin{align*}\pushQED{\qed} 
  & d_{BL}\left( \Law\left(\sqrt{k} \begin{pmatrix}
\muboot - \mu_n\\
\nuboot - \nu_n\\
\cboot - c_n
\end{pmatrix} \middle| X_1, \dots, X_n, Y_1, \dots Y_n\right),\Law\left( \sqrt{n} \begin{pmatrix}
\mu_n - \mu\\
\nu_n - \nu\\
c_n - c
\end{pmatrix}\right) \right)\ptoOut 0. \qedhere
\popQED
\end{align*}

\subsection{Proof of Proposition \ref{prop:AbstractB2_CoveringNumbers}} 
Before, we start to prove \Cref{prop:AbstractB2_CoveringNumbers}, we establish an auxiliary lemma.
\begin{lem}\label{lem:auxiliary36}
	Let $\X$ and $\Y$ be compact Polish spaces and consider $c\in\C(\X\times\Y)$.
	\begin{itemize}
		\item[(i)] For any function $g:\X\to\RR$ and any constant $\kappa$, it holds that $(g+\kappa)^c=g^c-\kappa$.
		\item[(ii)]Let $B>0.$ Then, for any  $g:\X\to\RR$ and $\Delta^c\in\C(\X\times\Y)$ with $\norm{g}_\infty+2\norm{c+\Delta^c}_\infty\leq B$ it holds  that
		$g^{(c+\Delta^c)(c+\Delta^c)cc}\in\HC_c+[-B,B]$.
	\end{itemize}
\end{lem}
The proof of the above lemma can be found in \Cref{subsec:auxiliary36Proof}.
\begin{proof}[Proof of \Cref{prop:AbstractB2_CoveringNumbers}]
The proof is strongly inspired by \cite[Theorem 2.3]{wellner2007empirical} and employs standard empirical process arguments. In order to simplify the notation, we only consider the case $n = m$ and write $c_{n}$ instead of $c_{n,n}$. Note that the claim for $n\neq m$ follows by the analogous arguments.

To show $(i)$ first note by triangle inequality and using \Cref{lem:CtrafoLipschitz} that 
\begin{align}
  \sup_{f\in \FC}|\Gproc_n^\mu(f^{c_n c_n}- f^{c c})| &\leq \sup_{f\in \FC}|\Gproc_n^\mu(f^{c_n c_n}- f^{\tilde c_n \tilde c_n})| + \sup_{f\in \FC}|\Gproc_n^\mu(f^{\tilde c_n \tilde c_n}- f^{c c})|\notag\\
  &\leq 4\sqrt{n} \norm{c_n - \tilde c_n}_\infty + \sup_{f\in \FC}|\Gproc_n^\mu(f^{\tilde c_n \tilde c_n}- f^{c c})|.\label{eq:UpperBoundConvergenceInProb}
\end{align}
The first term converges by assumption for $n \rightarrow \infty$ in probability to zero. For the latter term note by \ref{ass:AThree} and the assumption on $\tilde c_n$ that $\sqrt{n/2}(\tilde c_n - c)\dto \Gproc^c$. By tightness of the law of $\Gproc^c$ there exists for any $\eps>0$ a compact set $K\subseteq C(\XC\times \YC)$ such that $\prob(\Gproc^{c}\in K)>1-\eps$; thus for any $\delta>0$ the set $K^\delta$ of elements in $C(\XC\times \YC)$ with distance less than $\delta>0$ to $K$ fulfills \begin{align}\label{eq:WeakConvergenceBoundCn}
  \liminf_{n \rightarrow \infty} \prob\left(\sqrt{n/2}(\tilde c_n - c)\in K^\delta\right) \geq \prob(\Gproc^c\in K^\delta)>1-\eps.
\end{align}
By compactness of $K$ there exists a finite $\delta/2$-covering $\{h_1, \dots, h_{p}\}$ which implies that $K^{\delta/2}\subseteq \bigcup_{i = 1}^{p}B(h_i, \delta)$, where $B(h, \delta)$ denotes the open ball of radius $\delta$ around $h$ in the space $C(\XC\times \YC)$. 
We thus obtain \begin{align*}
   \left\{\sqrt{n/2}(\tilde c_n - c) \in K^{\delta/2}\right\} \subset \bigcup_{i =1}^{p}\left\{\tilde c_n \in B(c+2n^{-1/2}h_i,\delta)\right\}.
\end{align*}
Moreover, by \citet[Proposition 1.34]{santambrogio2015optimal} it follows for any  $f\in \FC$ and $\bar c\in C(\XC\times \YC)$ that $f^{\bar c \bar c} = f^{\bar c\bar c\bar c\bar c}$. Therefore, by triangle inequality, 
\begin{align}
  \sup_{f\in \FC}|\Gproc_n^\mu(f^{\tilde c_n \tilde c_n}- f^{c c})|&= \sup_{f\in \FC}|\Gproc_n^\mu(f^{\tilde c_n \tilde c_n\tilde c_n\tilde c_n}- f^{ccc c})|\notag\\
  &\leq \sup_{f\in \FC}|\Gproc_n^\mu(f^{\tilde c_n \tilde c_n\tilde c_n\tilde c_n}- f^{\tilde c_n\tilde c_nc c})| + \sup_{f\in \FC}|\Gproc_n^\mu(f^{\tilde c_n \tilde c_n cc}- f^{c c c c})|\notag\\
  &\leq \sup_{f\in \FC^{\tilde c_n \tilde c_n}}|\Gproc_n^\mu(f^{\tilde c_n\tilde c_n}- f^{c c})| + \sup_{f\in \FC}|\Gproc_n^\mu(f^{\tilde c_n \tilde c_n cc}- f^{c c c c})|.\label{eq:twoTermsToBeControlledForSup}
\end{align}
Assuming $\sqrt{n/2}(\tilde c_n - c) \in K^{\delta/2}$, it follows for the first term in  \eqref{eq:twoTermsToBeControlledForSup} that 
\begin{align}
	\sup_{f\in \FC^{\tilde c_n \tilde c_n}}|\Gproc_n^\mu(f^{\tilde c_n \tilde c_n}- f^{c c})|\notag%
	 &\leq \sup_{f\in \FC^{\tilde c_n \tilde c_n}}\max_{i = 1,\dots,p} \sup_{\norm{h-h_i}_\infty < \delta} |\Gproc_n^\mu(f^{(c + 2h/\sqrt{n}) (c + 2h/\sqrt{n})}- f^{c c})|\notag\\
	&\leq \sup_{f\in \FC^{\tilde c_n \tilde c_n}}\max_{i = 1,\dots,p} \sup_{\norm{h-h_i}_\infty < \delta} |\Gproc_n^\mu(f^{(c + 2h/\sqrt{n}) (c + 2h/\sqrt{n})}- f^{(c + 2h_i/\sqrt{n}) (c + 2h_i/\sqrt{n})})|\notag\\
	&+ \sup_{f\in \FC^{\tilde c_n \tilde c_n}}\max_{i = 1,\dots,p} |\Gproc_n^\mu(f^{(c + 2h_i/\sqrt{n}) (c + 2h_i/\sqrt{n})}- f^{c c})|\notag\\
	&\leq  8 \delta + \sup_{f\in \FC^{\tilde c_n \tilde c_n}}\max_{i = 1,\dots,p} |\Gproc_n^\mu(f^{(c + 2h_i/\sqrt{n}) (c +2 h_i/\sqrt{n})}- f^{c c})|.\label{eq:boundTerm1}
\end{align}
Here, we used in the last inequality \Cref{lem:CtrafoLipschitz} to infer $$\norm{f^{(c + 2h/\sqrt{n}) (c + 2h/\sqrt{n})}- f^{(c + 2h_i/\sqrt{n}) (c + 2h_i/\sqrt{n})}}_\infty \leq 4\norm{h_i-h}_\infty/\sqrt{n} \leq 4\delta/\sqrt{n}$$ 
in conjunction with $\Gproc^\mu_n(g) = \sqrt{n}(\mu_n - \mu)(g) \leq 2\sqrt{n} \norm{g}_\infty$ for any measurable function $g$ on $\XC$. Now, define for 1$\leq i\leq p$ the function class
\begin{align*}
\tilde \GC^i_n\coloneqq \tilde \GC^i_n(h_i) \coloneqq \left\{ f^{(c + 2h_i/\sqrt{n}) (c +2 h_i/\sqrt{n})}- f^{c c} \big| f\in \FC^{\tilde c_n\tilde c_n} \right\}.
\end{align*}
For each $1\leq i\leq p$ and any $\eps>0$  we then observe that
\begin{align*}
	\log N(\eps,\tilde \GC^i_n,\norm{\cdot}_\infty)\leq& \log N(\eps,\FC^{\tilde c_n\tilde c_n(c + 2h_i/\sqrt{n}) (c +2 h_i/\sqrt{n})},\norm{\cdot}_\infty)+\log N(\eps,\tilde \FC^{\tilde c_n\tilde c_n c c},\norm{\cdot}_\infty)\\
	\leq & 2 \log N(\eps,\FC^{\tilde c_n\tilde c_n},\norm{\cdot}_\infty),
\end{align*}
where the last step follows by Lemma 2.1 in \cite{hundrieser2022empirical}. In consequence, it follows by Dudley's entropy integral (see, e.g.,\cite[Chapter~5]{wainwright2019high}) that 
\begin{align}
  \EE\left[ \sup_{f\in \FC^{\tilde c_n\tilde c_n}}\max_{i = 1,\dots,p} \left|\Gproc_n^\mu(f^{(c +2h_i/\sqrt{n}) (c + 2h_i/\sqrt{n})}- f^{c c})\right|\right]&\leq \sum_{i = 1}^p \EE\left[\sup_{g\in \tilde \GC^i_n} \left|\Gproc_n^\mu(g)\right|\right]\notag\\
  &\lesssim \sum_{i = 1}^p \int_{0}^{4\norm{h_i}_\infty/\sqrt{n}} \sqrt{\log\left(\NC(\eps, \FC^{\tilde c_n\tilde c_n}, \norm{\cdot}_\infty)\right)} d\eps\notag\\
  &\lesssim \sum_{i = 1}^p \int_{0}^{4\norm{h_i}_\infty/\sqrt{n}} \!\!\eps^{-\alpha/2} d\eps\notag\\
  & \lesssim \sum_{i = 1}^p (\norm{h_i}_\infty/\sqrt{n})^{1-\alpha/2},\notag
\end{align}
where by assumption the hidden constants do not depend on $n$. 
We thus infer conditionally on the event $\sqrt{n/2}(\tilde c_n - c) \in K^{\delta/2}$ for $n\rightarrow\infty$ that 
\begin{align}
  \sup_{f\in \FC^{\tilde c_n \tilde c_n}}\max_{i = 1,\dots,p}  \Gproc_n^\mu(f^{(c + 2h_i/\sqrt{n}) (c + 2h_i/\sqrt{n})}- f^{c c})\pto 0.\label{eq:ConvP_Term1}
\end{align}

For the second term in \eqref{eq:twoTermsToBeControlledForSup} we assume  $\sqrt{n/2}(\tilde c_n - c) \in K^{\delta/2}$ and obtain by similar arguments,
\begin{align}
  \sup_{f\in \FC}|\Gproc_n^\mu(f^{\tilde c_n \tilde c_n cc}- f^{c c c c})| &\leq 8\delta + \sup_{f\in \FC}\max_{i = 1, \dots, p}|\Gproc_n^\mu(f^{(c + 2h_i/\sqrt{n}) (c +2 h_i/\sqrt{n}) cc}- f^{c c c c})|. \label{eq:boundTerm2}
\end{align}
Upon defining the function class 
\begin{align}\label{eq:prop36bootfc}
  \tilde \GC_n\coloneqq \tilde \GC_n(h_1, \dots, h_p) \coloneqq \left\{ f^{(c + 2h_i/\sqrt{n}) (c +2 h_i/\sqrt{n}) cc}- f^{c c c c} \big| f\in \FC \right\}
\end{align}
we note again by \Cref{lem:CtrafoLipschitz} that any $g\in\GC_n$ fulfills $\norm{g}_{\infty}\leq \max_{i = 1, \dots, p}4\norm{h_i}_\infty/\sqrt{n}$. Further, for $n$ sufficiently large there exists a constant $B>0$ such that \Cref{lem:auxiliary36} is applicable for any $f\in\FC$, and we obtain $$ f^{(c + 2h_i/\sqrt{n}) (c +2 h_i/\sqrt{n}) cc} \in \HC_{c}+[-B,B].$$ 
Hence, for sufficiently large $n$, it follows by \Cref{lem:RelFCandConcave} for any $\eps>0$ that 
\begin{align*}
  \NC(\eps, \tilde \GC_n(h_1, \dots, h_p), \norm{\cdot}_\infty)\leq \left(\NC(\eps, \FC^{cc}+[-B,B], \norm{\cdot}_\infty)\right)^2.
\end{align*} 
Again invoking, Dudley's entropy integral asserts such for $n$ that 
\begin{align}
   \EE\Bigg[\sup_{f\in \FC}\max_{i = 1, \dots, p}\bigg|\Gproc_n^\mu(f^{(c + 2h_i/\sqrt{n}) (c +2 h_i/\sqrt{n}) cc}&- f^{c c c c})\bigg|\Bigg] = \EE\left[\sup_{f\in \tilde\GC_n}\left|\Gproc_n^\mu(\tilde f)\right|\right]\notag\\
  \lesssim &\int_{0}^{\max_{i = 1, \dots, p}4\norm{h_i}_\infty/\sqrt{n}} \sqrt{\log\left(\NC(\eps, \tilde \GC_n, \norm{\cdot}_\infty)\right)} d\eps\notag\\
  \leq &\int_{0}^{\max_{i = 1, \dots, p}4\norm{h_i}_\infty/\sqrt{n}} \sqrt{\log\left(\NC(\eps, \FC^{c c}+[-B,B], \norm{\cdot}_\infty)\right)} d\eps\notag\\
  \lesssim &\int_{0}^{\max_{i = 1, \dots, p}4\norm{h_i}_\infty/\sqrt{n}} \!\!\eps^{-\alpha/2} d\eps\notag\\
  \lesssim &\left(\max_{i = 1, \dots, p}\norm{h_i}_\infty/\sqrt{n}\right)^{1-\alpha/2}\notag.
\end{align}
This implies conditionally on the event $\sqrt{n/2}(\tilde c_n - c) \in K^{\delta/2}$ for $n\rightarrow \infty$ that 
\begin{align}
  \sup_{f\in \FC}\max_{i = 1, \dots, p}|\Gproc_n^\mu(f^{(c + 2h_i/\sqrt{n}) (c +2 h_i/\sqrt{n}) cc}- f^{c c c c})| \pto 0.\label{eq:ConvP_Term2}
\end{align}

Concluding, for any $\eps>0$ it follows for $\delta\coloneqq \eps/32>0$ from \eqref{eq:WeakConvergenceBoundCn}--\eqref{eq:ConvP_Term2} that
\begin{align*}
  \;&\limsup_{n\rightarrow \infty}\mathbb{P}\left(\sup_{f\in \FC}|\Gproc_n^\mu(f^{\tilde c_n \tilde c_n}- f^{c c})|>\eps\right) \\
   \leq \;&\limsup_{n\rightarrow \infty}\left(\mathbb{P}\left(\sup_{f\in \FC}|\Gproc_n^\mu(f^{\tilde c_n \tilde c_n}- f^{c c})|>\eps,\sqrt{n/2}(\tilde c_n - c) \in K^{\delta/2} \right) + \mathbb{P}\left(\sqrt{n/2}(\tilde c_n - c) \not\in K^{\delta/2}\right)\right)\\
   \leq \;& \limsup_{n\rightarrow \infty}\mathbb{P}\left(\sup_{f\in \FC}|\Gproc_n^\mu(f^{\tilde c_n \tilde c_n}- f^{c c})|>\eps,\sqrt{n/2}(\tilde c_n - c) \in K^{\delta/2} \right) + \eps\\
   \leq \;& \limsup_{n\rightarrow \infty}\mathbb{P}\left(\sup_{f\in \FC^{\tilde c_n \tilde c_n}}|\Gproc_n^\mu(f^{\tilde c_n\tilde c_n}- f^{c c})|>\eps/2,\sqrt{n/2}(\tilde c_n - c) \in K^{\delta/2} \right) \\
   +& \limsup_{n\rightarrow \infty}\mathbb{P}\left(\sup_{f\in \FC}|\Gproc_n^\mu(f^{\tilde c_n \tilde c_n cc}- f^{c c c c})|>\eps/2,\sqrt{n/2}(\tilde c_n - c) \in K^{\delta/2} \right) + \eps\\
   \leq \;& \limsup_{n\rightarrow \infty}\mathbb{P}\left(\sup_{f\in \FC^{\tilde c_n \tilde c_n}}\max_{i = 1,\dots,p} |\Gproc_n^\mu(f^{(c + 2h_i/\sqrt{n}) (c +2 h_i/\sqrt{n})}- f^{c c})|>\eps/4,\sqrt{n/2}(\tilde c_n - c) \in K^{\delta/2} \right) \\
   +& \limsup_{n\rightarrow \infty}\mathbb{P}\left(\sup_{f\in \FC}\max_{i = 1, \dots, p}|\Gproc_n^\mu(f^{(c + 2h_i/\sqrt{n}) (c +2 h_i/\sqrt{n}) cc}- f^{c c c c})|>\eps/4,\sqrt{n/2}(\tilde c_n - c) \in K^{\delta/2} \right) + \eps= \eps,
\end{align*}
which shows the convergence in probability of $\sup_{f\in \FC}|\Gproc_n^\mu(f^{\tilde c_n \tilde c_n}- f^{c c})|$ to zero. 
 We thus conclude the convergence in probability for both terms of \eqref{eq:UpperBoundConvergenceInProb}. 
An analogous argument yields the convergence $\sup_{f\in \FC}|\Gproc_n^\nu(f^{c_n}- f^{c})|\pto 0$ for $n\rightarrow \infty$, where we apply Lemma 2.1 of \cite{hundrieser2022empirical} to obtain
 \begin{align*}
  \sup_{n \in \NN} \log \NC(\eps, \FC^{\tilde c_n}\cup \FC^c, \norm{\cdot}_\infty)&\leq \sup_{n \in \NN}\left(\log \NC(\eps, \FC^{\tilde c_n}, \norm{\cdot}_\infty)+ \log\NC(\eps, \FC^{c}, \norm{\cdot}_\infty) \right)\\
  &= \sup_{n \in \NN}\left(\log\NC(\eps, \FC^{\tilde c_n\tilde c_n}, \norm{\cdot}_\infty) + \log\NC(\eps, \FC^{c c}, \norm{\cdot}_\infty)\right)\\
  & \leq \sup_{n \in \NN}2\log\NC(\eps, \FC^{\tilde c_n\tilde c_n}\cup \FC^{c c}, \norm{\cdot}_\infty)\lesssim \eps^{-\alpha} \quad \text{ for $\alpha<2$},
\end{align*}
which overall verifies \ref{ass:BTwo} of \Cref{thm:AbstractMainResult}. 

For $(ii)$ note by \citet{bucher2019note} and since $k= k(n) = \oh(n)$ for $n\rightarrow \infty$ that 
$$
	\sqrt{k}(\ctildeboot - c) =  \sqrt{k}(\ctildeboot - \cboot) + \sqrt{k}(\cboot - c_n) + \sqrt{\frac{k}{n}} \sqrt{n}(c_n- c) \dto \Gproc^c.
$$
Likewise, it follows for $n\rightarrow \infty$ that $\sqrt{k}(\muboot - \mu)\dto \Gproc^\mu$ in $\ell^\infty(\FC^{cc})$,  $\sqrt{k}(\nuboot - \nu)\dto \Gproc^\nu$ in $\ell^\infty(\FC^{c})$. This means that we can pursue a similar proof strategy as for $(i)$. Define $\Gproc_{n,k}^\mu \coloneqq \sqrt{k}(\muboot - \mu)$ and $\Gproc_{n,k}^\nu \coloneqq \sqrt{k}(\nuboot - \nu)$. %
 Then, we infer from \Cref{lem:CtrafoLipschitz} that
\begin{align}
   \sup_{f \in \FC}|\Gproc^\mu_{n,k}(f^{\cboot \cboot} - f^{cc})| &\leq   \sup_{f \in \FC}|\Gproc^\mu_{n,k}(f^{\cboot \cboot} - f^{\ctildeboot \ctildeboot})| +   \sup_{f \in \FC}|\Gproc^\mu_{n,k}(f^{\ctildeboot \ctildeboot} - f^{cc})|\notag\\
  &\leq 4\sqrt{k} \norm{\cboot - \ctildeboot}_\infty + \sup_{f\in \FC}|\Gproc_{n,k}^\mu(f^{\ctildeboot \ctildeboot}- f^{c c})|,\label{eq:UpperBoundConvergenceInProb2}
\end{align}
where the first term converges for $n\rightarrow \infty$ in probability to zero. By \citet[Proposition 1.34]{santambrogio2015optimal} we obtain that
\begin{align*}
\sup_{f\in \FC}|\Gproc_{n,k}^\mu(f^{\ctildeboot \ctildeboot}- f^{c c})|
&\leq \sup_{f\in \FC^{\ctildeboot \ctildeboot}}|\Gproc_{n,k}^\mu(f^{\ctildeboot \ctildeboot}- f^{c c})| + \sup_{f\in \FC}|\Gproc_{n,k}^\mu(f^{\ctildeboot  \ctildeboot  cc}- f^{c c c c})|.
\end{align*}
Moreover, by analogous arguments to those for $(i)$ we obtain with probability at least $1-\eps$ for $n$ sufficiently large that
\begin{align}\label{eq:prop36boot1}
\sup_{f\in \FC^{\ctildeboot \ctildeboot}}|\Gproc_{n,k}^\mu(f^{\ctildeboot \ctildeboot}- f^{c c})|
&\leq  8 \delta + \sup_{f\in \FC^{\ctildeboot \ctildeboot}}\max_{i = 1,\dots,p} |\Gproc_{n,k}^\mu(f^{(c + 2h_i/\sqrt{k}) (c +2 h_i/\sqrt{k})}- f^{c c})|
\end{align}
as well as
\begin{align}\label{eq:prop36boot2}
\sup_{f\in \FC}|\Gproc_{n,k}^\mu(f^{\ctildeboot \ctildeboot cc}- f^{c c c c})| &\leq 8\delta + \sup_{f\in \FC}\max_{i = 1, \dots, p}|\Gproc_{n,k}^\mu(f^{(c + 2h_i/\sqrt{k}) (c +2 h_i/\sqrt{k}) cc}- f^{c c c c})|. 
\end{align}
Next, we verify that the suprema on the right-hand sides of \eqref{eq:prop36boot1} and \eqref{eq:prop36boot2} converge (unconditionally with respect to the $\mu_n$ but conditionally on the set with probability at least $1-\eps$)  to zero. We note by Dudley's entropy integral for the bootstrap empirical process $\sqrt{k}(\muboot-\mu_n)$ and the empirical process $\sqrt{n}(\mu_n-\mu)$ as well as our previous considerations that
\begin{align*}
&\textcolor{white}{\lesssim} \EE\left[ \sup_{f\in \FC^{\ctildeboot \ctildeboot}}\max_{i = 1,\dots,p}  |\Gproc_{n,k}^\mu(f^{(c + 2h_i/\sqrt{k}) (c + 2h_i/\sqrt{k})}- f^{c c})|\right]\\
&= 
\sum_{i = 1}^p \EE_{\mu_n}\left[\EE_{\muboot}\left[\sup_{f\in \FC^{\ctildeboot \ctildeboot}} |\sqrt{k}(\muboot - \mu_{n})(f^{(c + 2h_i/\sqrt{k}) (c + 2h_i/\sqrt{k})}- f^{c c})|\Bigg|\ \mu_n\right]\right] \\
&\quad + \sqrt{\frac{k}{n}}\EE\left[\sup_{f\in \FC^{\ctildeboot \ctildeboot}}|\Gproc_n^\mu(f^{(c +2 h_i/\sqrt{k}) (c +2 h_i/\sqrt{k})}- f^{c c})|\right]\\
&\lesssim \sum_{i = 1}^p \EE_{\mu_n}\int_{0}^{4\norm{h_i}_\infty/\sqrt{k}}\! \! \!\!  \sqrt{\log\left(\NC(\eps,  \FC^{\ctildeboot \ctildeboot}, \norm{\cdot}_\infty)\right)} d\eps + \sqrt{\frac{k}{n}} \int_{0}^{4\norm{h_i}_\infty/\sqrt{k}}\! \! \! \! \sqrt{\log\left(\NC(\eps,  \FC^{\ctildeboot \ctildeboot}, \norm{\cdot}_\infty)\right)} d\eps \\
&\lesssim \sum_{i = 1}^p \left( 1 + \sqrt{\frac{k}{n}}\right)\int_{0}^{4\norm{h_i}_\infty/\sqrt{k}}\eps^{-\alpha/2} d\eps \lesssim \sum_{i = 1}^p \left( 1 + \sqrt{\frac{k}{n}}\right)\left(\norm{h_i}_\infty/\sqrt{k}\right)^{1-\alpha/2},
\end{align*}
which tends to zero for $n \rightarrow \infty$ with $k= k(n)= \oh(n)$   since the hidden constants do not depend on $n, k$. 
 Recalling the definition of the function class $\tilde{\GC}_k$ in \eqref{eq:prop36bootfc} with $n$ replaced by $k$, we obtain
\begin{align*}
&\textcolor{white}{\lesssim} \EE\left[ \sup_{f\in \FC}\max_{i = 1, \dots, p}|\Gproc_{n,k}^\mu(f^{(c + 2h_i/\sqrt{k}) (c +2 h_i/\sqrt{k}) cc}- f^{c c c c})| \right]=\EE\left[\sup_{f\in\tilde{\GC}_k}\left|\Gproc_{n,k}^\mu(f)\right|\right].
\end{align*}
Hence, Dudley's entropy integral in combination with our previous considerations yields
\begin{align*}
  \EE\left[\sup_{f\in\tilde{\GC}_k}\left|\Gproc_{n,k}^\mu(f)\right|\right]
\leq \;& \EE_{\mu_n}\left[\EE_{\muboot}\left[\sup_{f\in\tilde{\GC}_k} |\sqrt{k}(\muboot - \mu_{n})(f)|\Bigg|\ \mu_n\right] \right]
+\sqrt{\frac{k}{n}}\EE\left[\sup_{f\in\tilde{\GC}_k} |\Gproc_n^\mu(f)|\right]\\
\lesssim \;&\left(1+\sqrt{\frac{k}{n}}\right)\int_{0}^{\max_{i = 1, \dots, p}4\norm{h_i}_\infty/\sqrt{k}} \sqrt{\log\left(\NC(\eps, \FC^{c c}+[-B,B], \norm{\cdot}_\infty)\right)} d\eps\notag\\
\lesssim \;&\left(1+\sqrt{\frac{k}{n}}\right)\left(\max_{i = 1, \dots, p}\norm{h_i}_\infty/\sqrt{k}\right)^{1-\alpha/2},
\end{align*}
which goes to zero for $n,k(n) \rightarrow \infty$ with $k(n)= \oh(n)$ (the hidden constants are independent of $n, k$). 

Using the same arguments as in $(i)$, we conclude that $$\sup_{f\in \FC}|\Gproc_{n,k}^\mu(f^{\ctildeboot \ctildeboot}- f^{c c})|\pto 0.$$
Finally, analogous arguments yield that $ \sup_{f \in \FC}|\Gproc^\nu_{n,k}(f^{\cboot } - f^{c})|\pto 0$, thus showing $(ii)$.
\end{proof}

\subsection{Proof of Corollary \ref{cor:SufficientConditionsSupDeterministic}} 
Define the random variables $$\tilde c_n \coloneqq \begin{cases} c &\text{ if } n < N,\\ c_n & \text{ if } n \geq N,
\end{cases} \quad \text{ and } \quad \ctildeboot \coloneqq \begin{cases} c &\text{ if } n < N \text{ or } k<K,\\ \cboot & \text{ if } n \geq N \text{ and } k\geq K.
\end{cases}$$
  By \Cref{prop:DonskerProperty} the cost estimators $\tilde c_n$ and $\ctildeboot$ satisfy the entropy bounds in \eqref{eq:EntropyCondition} and   \eqref{eq:EntropyConditionBS}. Tightness of $N$ and $K$ implies that $\sqrt{n}\norm{\tilde c_n - c_n}_\infty\pto 0$  and $\sqrt{k}\|\ctildeboot - \cboot\|_\infty \pto 0$ for $n,k\rightarrow \infty$, which asserts the claim by \Cref{prop:AbstractB2_CoveringNumbers}. \qed

\subsection{Proof of Corollary \ref{cor:SufficientConditionsSupConditions}} 
  By Assumption \ref{ass:AThree} it follows that $\sqrt{nm/(n+m)}(c_{n,m} - c)\dto \Gproc^c$, whereas under \ref{ass:AThreeStar} we infer from \citet{bucher2019note} and $k = \oh(n)$ that $\sqrt{k}(\cboot - c)\dto \Gproc^c$ unconditionally. In what follows we state the arguments for \ref{ass:BTwo}; for \ref{ass:BTwoStar} a similar proof strategy applies by replacing the empirical costs process by the bootstrap cost process. 

  First, assume without loss of generality that the population cost function fulfills $\norm{c}_\infty\leq 1$. Then, for all three settings of \Cref{prop:DonskerProperty} it follows that $\log \NC(\eps, \FC^{cc}, \norm{\cdot}_\infty) \lesssim \eps^{-\alpha}$ with $\alpha<2$. 
  
  For setting $(i)$ we set $\tilde  c_{n,m}\coloneqq \Psi_{\textup{bdd}}(c_{n,m})$, for $\Psi_{\textup{bdd}}$ defined in \Cref{subsec:RegELBdd}. Since $\|\tilde c_{n,m}\|_\infty \leq 2$ and $\FC=\FC(2\norm{c}_\infty+1, 2w)$ is uniformly bounded by $6$, we obtain that $\FC^{\tilde c_{n,m}}$ is uniformly bounded by $8$. By \Cref{thm:RegElevationAbstract} and \ref{prop:RegEl:Bdd} both conditions of \Cref{prop:AbstractB2_CoveringNumbers}$(i)$ are met, asserting \ref{ass:BTwo}. 
  
  For setting $(ii)$ we take $\tilde c_{n,m}\coloneqq \Psi_{\textup{mod}}^{\tilde d_\XC}\circ \Psi_{\textup{bdd}}(c_{n,m})$ for $\Psi_{\textup{mod}}^{\tilde d_\XC}$ from \Cref{subsec:RegELMod}. Then, $\|\tilde c_{n,m}\|_\infty \leq 2$ and $\FC^{\tilde c_{n,m}}$ is uniformly bounded by $8$. Moreover, by Assumption $(ii)$' it follows with Proposition~\ref{thm:RegElevationAbstract} and \ref{prop:RegEl:Mod} that $\sqrt{nm/(n+m)}\|\tilde c_{n,m}-c_{n,m}\|_\infty\pto 0$ and that $$\sup_{n\in \NN}\log\NC(\eps, \FC^{\tilde c_{n,m} \tilde c_{n,m}}, \norm{\cdot}_\infty) \lesssim \NC(\eps/8, \XC, \tilde d_\XC)|\log(\eps)|\lesssim \eps^{-\beta}|\log(\eps)|\lesssim \eps^{-2 + (2-\beta)/2},$$
  where we used the covering number assumption on $\XC$. \ref{ass:BTwo} then follows from \Cref{prop:AbstractB2_CoveringNumbers}$(i)$.

  For setting $(iii)$ define $c_i \in C(\UC_i\times \YC)$ as  $c_i(u,y) \coloneqq c(\zeta_i(u),y)$. We consider $\tilde c_{n,m}\coloneqq \Psi_{\textup{com}}(c_{n,m})$ where $\Psi_{\textup{com}}$ denotes the combination (\Cref{subsec:RegElCom}) of regularity elevation functionals $\Psi_i\colon C(\UC_i\times \YC)\rightarrow C(\UC_i\times \YC)$ defined by $\Psi_i =  \Psi_{\textup{mod}}^{\norm{\cdot}^{\gamma_i}}\circ \Psi_{\textup{bdd}}$ from \Cref{subsec:RegELMod} if $\gamma_i \in (0,1]$, and $\Psi_i = \Psi_{\textup{Hol}}^{c_i, {\gamma_i}}\circ \Psi_{\textup{bdd}}$ from \Cref{subsec:RegElHol} if $\gamma_i \in (1,2]$, where we replace $\XC$ by $\UC_i$. 
  Then, by Propositions \ref{prop:RegEl:Mod}, \ref{prop:RegEl:Hol}, and \ref{prop:RegEl:Union} the functional $\Psi$ fulfills the assumptions of \Cref{thm:RegElevationAbstract} and therefore $\sqrt{nm/(n+m)}\|\tilde c_{n,m}-c_{n,m}\|_\infty\pto 0$. Moreover, since for any $\tilde c\in C(\XC\times \YC)$ it holds that  $\|\Psi(\tilde c)\|_{\infty}<C$ for a deterministic constant $C\geq 0$ that only depends on the functions $c_i$ and the spaces $\UC_i$, it follows that $\FC^{\Psi(\tilde c)}$ is uniformly bounded by $C+6$ and therefore %
  \begin{align*}
    \sup_{n\in \NN}\log\NC(\eps, \FC^{\tilde c_{n,m} \tilde c_{n,m}}, \norm{\cdot}_\infty) &\lesssim \sum_{i=1}^I \sup_{\tilde c_i\in C(\UC_i\times \YC)} \log \NC(\eps, \FC^{\tilde c_{n,m} \Psi_i(\tilde c_i)}, \norm{\cdot}_\infty)\lesssim \max_{i = 1, \dots, I} \eps^{-d_i/\gamma_i},
  \end{align*}
  where we use for the first inequality  \Cref{prop:RegEl:Union}, and for the second we employ the bounds from Proposition \ref{prop:RegEl:Mod} with $\NC(\eps, \UC_i, \norm{\cdot}^{\gamma_i})\lesssim \eps^{-d_i/\gamma_i}$ for $0<\gamma_i\leq 1$ and Proposition \ref{prop:RegEl:Hol} for $1<\gamma_i\leq 2$. The assertion then follows by an application of Proposition \ref{prop:AbstractB2_CoveringNumbers}$(i)$. \qed

\subsection{Proof of Lemma \ref{lem:ParameterBoundMetricEntropy}}\label{subsubsec:ParameterBoundMetricEntropyProof}
	For $\eps>0$ suppose that the right-hand side is finite since otherwise the claim is vacuous. 
	Set $k = \NC(\eps/4, \Theta, d_\Theta)$ and let $\{\theta_1, \dots, \theta_{k}\}$ be a minimal $\eps/4$-covering of $\Theta$. Further, for each $i= 1, \dots, k$ let $\{f^{i}_{1},\dots, f^{i}_{k_i}\}$ be a minimal $\eps/2$-covering of $\FC^{c_{\theta_i}c_{\theta_i}}$, i.e., $k_i = \NC\left(\eps/2, \FC^{c_{\theta_i} c_{\theta_i}}, \norm{\cdot}_\infty\right)$. Once we show that $\FC_\XC(\eps) \coloneqq \bigcup_{i=1}^{k}\{f^{i}_{1},\dots, f^{i}_{k_i}\}$ is an $\eps$-covering for $\bigcup_{\theta\in \Theta} \FC^{c_\theta c_\theta}$ and that $\FC_\YC(\eps) \coloneqq \bigcup_{i=1}^{k}\{(f^{i}_{1})^{c_{\theta_i}},\dots, (f^{i}_{k_i})^{c_{\theta_i}}\}$ is an $\eps$-covering for $\bigcup_{\theta\in \Theta} \FC^{c_\theta}$ the claim follows, since 
	\begin{align*}
	|\FC_\YC(\eps)|\leq   |\FC_\XC(\eps)| &= \sum_{i = 1}^{k}  \NC\left(\frac{\eps}{2},  \FC^{c_{\theta_i} c_{\theta_i}}, \norm{\cdot}_\infty\right) \leq \NC\left(\frac{\eps}{4}, \Theta, d_\Theta\right) \sup_{\theta\in \Theta}  \NC\left(\frac{\eps}{2}, \FC^{c_\theta c_\theta}, \norm{\cdot}_\infty\right).
	\end{align*}
	
	Hence, let $\theta\in \Theta$ and $f\in \FC^{c_\theta c_\theta}$, and choose $\tilde f\in \FC$ with $f = \tilde f^{c_\theta c_\theta}$. Select $\theta_i$ with $d_{\Theta}(\theta, \theta_i)\leq \eps/4$ and choose $f^{i}_{l_i}\in \FC_\XC(\eps)$ such that $\norm{f^{i}_{l_i} - \tilde f^{c_{\theta_i} c_{\theta_i}}}_\infty \leq \eps/2$. Now, by Lipschitzianity of the cost in $\theta$ and \Cref{lem:CtrafoLipschitz} we infer $\norm{\tilde f^{c_{\theta_i} c_{\theta_i}} - \tilde f^{c_{\theta} c_{\theta}}}_\infty\leq 2 d_\Theta(\theta, \theta_i) \leq \eps/2$, and it follows that \begin{align*}
	\norm{f^{i}_{l_i} - f}_\infty = \norm{f^{i}_{l_i} - \tilde f^{c_{\theta} c_{\theta}}}_\infty \leq \norm{f^{i}_{l_i} - \tilde f^{c_{\theta_i} c_{\theta_i}}}_\infty + \norm{\tilde f^{c_{\theta_i} c_{\theta_i}} - \tilde f^{c_{\theta} c_{\theta}}}_\infty \leq \eps,
	\end{align*}
	which verifies that $\FC_\XC(\eps)$ is an $\eps$-covering of $\cup_{\theta\in \Theta} \FC^{c_\theta c_\theta}$. 
	
	Moreover,  for any $f\in \FC^{c_\theta}$ there exists $\tilde f\in \FC$ with $f = \tilde f^{c_\theta}$ and by \citet[Proposition 1.34]{santambrogio2015optimal} it follows that $\tilde f^{c_\theta} = \tilde f^{c_\theta c_\theta c_\theta}$. Hence, upon selecting $f^{i}_{l_i}\in \FC_\XC(\eps)$ as above, we find by \Cref{lem:CtrafoLipschitz} that 
	$$\norm{(f^{i}_{l_i})^{c_{\theta_i}} - \tilde f^{c_{\theta_i}}}_\infty=\norm{(f^{i}_{l_i})^{c_{\theta_i}} - \tilde f^{c_{\theta_i}c_{\theta_i} c_{\theta_i}}}_\infty \leq \norm{f^{i}_{l_i} - \tilde f^{c_{\theta_i} c_{\theta_i}}}_\infty\leq \eps/2.$$
	Again invoking \Cref{lem:CtrafoLipschitz} yields $\norm{\tilde f^{c_{\theta_i}} - \tilde f^{c_{\theta}}}_\infty\leq d(\theta, \theta_i) \leq \eps/4$. 
	Consequently, we find that 
	\begin{align*}
	\norm{(f^{i}_{l_i})^{c_{\theta_i}} - f}_\infty = \norm{(f^{i}_{l_i})^{c_{\theta_i}} - \tilde f^{c_{\theta}}}_\infty \leq \norm{(f^{i}_{l_i})^{c_{\theta_i}}- \tilde f^{c_{\theta_i} }}_\infty + \norm{\tilde f^{c_{\theta_i}} - \tilde f^{c_{\theta}}}_\infty \leq \frac{3\eps}{4}\leq \eps,
	\end{align*}
	which proves that $\FC_\YC(\eps)$ is an $\eps$-covering of $\cup_{\theta\in \Theta} \FC^{c_\theta}$ and finishes the proof. \qed
\newpage

\section{Proofs for Section \ref{sec:Applications}: Applications}\label{sec:ProofsApplications}

\subsection{Proof of Lemma \ref{lem:PropRegularityOneSampleGOFCost}}
Select $U$ as the pre-image of $\{\tilde g\in C(\XC, \RR^d) \colon\!\|\tilde g -  g^{-1}_{\vartheta^o}\|_{\infty}< 1\}$ under $K_\Theta$, which is open (relative) in $\Theta$ due to continuity. Hence, by compactness of $\XC$, the collection $\{g^{-1}_{\vartheta}\}_{\vartheta\in U}$ is uniformly bounded on $\XC$. 
Invoking the Cauchy-Schwarz inequality and, due to compactness of $\YC$, we infer that $\{C_\Theta(\vartheta)(x,\cdot)\}_{\vartheta \in U, x\in \XC}$ is also uniformly bounded on $\YC$. Further, since $\nabla_y C_\Theta(\vartheta)(x,y) = 2 \big(g_\vartheta^{-1}(x) - y\big)$ for $y\in \interior{\YC}$ the collection $\{\nabla_y C_\Theta(\vartheta)(x,\cdot)\}$ is bounded on $\YC$ uniformly over $\vartheta\in U, x\in \XC$. Finally, note that  $\text{Hess}_y C_\Theta(\vartheta)(x,y)= -2\Id$, independent of $\vartheta\in U, x\in \XC$. Thus, by combining these observations, we conclude the existence of $\Lambda\geq 0$ such that the $(2, \Lambda)$-H\"older regularity is met. \qed

\subsection{Proof of Lemma \ref{prop: cstGoF}}
To establish the Hadamard differentiability of $C_\Theta$ at $\vartheta^o$ note that 
\begin{align*}
 &
 \left\lVert \frac{C_\Theta(\vartheta^o + t_n h_n)-C_\Theta(\vartheta^o)}{t_n} - D^H_{|\vartheta^o}{C_\Theta}(h) \right\rVert_{\infty}\\ 
 =  &\sup_{(x,y)\in \XC\times \YC}\left|\frac{1}{t_n }\left\langle g_{\vartheta^o+t_nh_n}^{-1}(x)-g_{\vartheta^o}^{-1}(x), g_{\vartheta^o+t_nh_n}^{-1}(x)+g_{\vartheta^o}^{-1}(x)-2y\right\rangle -2 \Big\langle D^H_{\vartheta^o}{K_\Theta}(h)(x),g_{\vartheta^o}^{-1}(x) - y\Big\rangle\right|\\
 \leq   &\sup_{(x,y)\in \XC\times \YC}\left(\left|2\left\langle \frac{1}{t_n} \left(g_{\vartheta^o+t_nh_n}^{-1}(x)-g_{\vartheta^o}^{-1}(x)\right) -  D^H_{\vartheta^o}{K_\Theta}(h)(x), g_{\vartheta^o}^{-1}(x)-y\right\rangle\right|+\frac{1}{t_n }\left\|  g_{\vartheta^o+t_nh_n}^{-1}(x)-g_{\vartheta^o}^{-1}(x)\right\|^2\right).
\end{align*}
For $n\rightarrow \infty$, the first term tends to zero by Hadamard differentiability of $K_\Theta$ whereas the second term tends to zero by \ref{assum: StabParamGroup}. Hence,  $C_\Theta$ is Hadamard differentiable at $\vartheta^o$. The second assertion follows from the functional delta method for Hadamard differentiable functionals \citep{Roemisch04}. \qed

\subsection{Proof of Proposition \ref{prop:gouplimit}}
First note that, in comparison to Sections \ref{sec:MainResults} and \ref{subsec:Assumptions}, the roles of $\XC$ and $\YC$ are interchanged. The universal Donsker property of $\FC^{C_{\Theta}(\vartheta^o)}$ follows from \Cref{prop:DonskerProperty}$(iii)$ since $d\leq 3$ and $C_\Theta(\vartheta^o)(x, \cdot)$ is $(2,\Lambda)$-H\"older for some $\Lambda\geq 0$ uniformly in $x\in \XC$ (\Cref{lem:PropRegularityOneSampleGOFCost}). Moreover, note by measurability of $\vartheta_n$ and continuity of $C_\Theta$ near $\vartheta^o$ that $c_n$ is also measurable. 
By joint weak convergence \eqref{eq: JointMeasParam} we infer from Hadamard differentiability of $C_\Theta$ at $\vartheta^o$ (\Cref{prop: cstGoF}) using the functional delta method that the one-sample version of \ref{ass:AThree} (recall \Cref{rmk:CommentsOTWeaklyConvergingCosts}\ref{rem:WeaklyConvergingCosts_OneSample}) is fulfilled. Further, since $\vartheta_n \pto \vartheta^o$, as $n$ tends to infinity,  we infer from \Cref{cor:SufficientConditionsSupConditions} and \Cref{lem:PropRegularityOneSampleGOFCost} that the one-sample version of \ref{ass:BTwo} is also met. The assertion now follows at once from \Cref{thm:AbstractMainResult}. \qed

\subsection{Proof of Proposition \ref{prop:OTInvariancesLimit}}
Note that by assumption, $(\mathcal{T},d_\mathcal{T})$ and $c$ fulfill the requirements of \Cref{thm:OTProcessInf}. Furthermore, note that Assumption \ref{ass:Don} can be established via \Cref{prop:DonskerProperty_ParameterClass} and Assumption \ref{ass:KPU} is implied by the assumptions on the support of $\mu$ and $\nu$ \citep[Corollary 2]{Staudt2022Unique}. Hence, the statement follows from \Cref{thm:OTProcessInf}. \qed

\subsection{Proof of Proposition \ref{prop:sliced OT}}
  Select  $\XC \subseteq \RR^d$ as a compact set which contains the supports of $\mu$ and $\nu$. 
  Note that  $\mathbb{S}^{d-1}$ is a compact Polish space and consider the Lipschitz map $c_{\mathbb{S}^{d-1}}:(\mathbb{S}^{d-1}, \norm{\cdot})\to C(\X\times\X)$, $\theta\mapsto c_\theta|_{\XC\times \XC}$ whose modulus  depends on $\XC$ and $p$. By compactness of $\XC$ and $\mathbb{S}^{d-1}$ it thus follows from the Theorem of Arzel\`a-Ascoli that $\{c_\theta|_{\XC\times \XC}\}_{\theta\in \mathbb{S}^{d-1}}$ is uniformly bounded and equicontinuous with a uniform modulus. Therefore, upon choosing the function class $\FC$ as in \Cref{thm:OTProcessCty},  Assertion $(i)$ follows by  \Cref{thm:OTProcessCty} once we verify that Assumption \ref{ass:Don} is fulfilled.
  To this end, note that $\log N(\eps,\mathbb{S}^{d-1},\norm{\cdot} )\lesssim |\log(\eps)|$. 
  Moreover, define for $\theta\in\mathbb{S}^{d-1}$ the pseudo metric $\tilde{d}_{\theta,\X}(x,y)=|\theta^Tx-\theta^Ty|$ on $\XC$ which fulfills  $\sup_{\theta\in\mathbb{S}^{d-1}}\NC(\eps,\XC,\tilde{d}_{\theta,\X})\lesssim\eps^{-1}$ and for any $x,x', y\in \XC$,
    \[\left|c_\theta(x,y)-c_\theta(x',y)\right|\leq p \, \diam(\mathfrak{p}_\theta(\XC))^{p-1}|\theta^Tx-\theta^Tx'|\leq  p \, \diam(\XC)^{p-1} \tilde{d}_{\theta,\X}(x,x').\] 
    Since the upper bound for the Lipschitz modulus does not depend on $\theta$, \Cref{prop:DonskerProperty_ParameterClass}$(ii)$ is applicable and we conclude that  $\bigcup_{\theta\in \mathbb{S}^{d-1}} \FC^{c_\theta c_\theta}$ and $\bigcup_{\theta\in \mathbb{S}^{d-1}} \FC^{c_\theta}$ are universal Donsker. By applying the continuous mapping theorem \citep[Theorem 1.11.1]{van1996weak} for the integration operator over $\mathbb{S}^{d-1}$ we obtain Assertion~$(ii)$. Finally, Assertion $(iii)$ follows from \Cref{thm:OTProcessSup}. \qed

    \subsection{Proof of Proposition \ref{prop:gateaux}}
    Since $\X\times \YC$ is compact and by continuity of $c$ and $\Delta^c$ there exists a common modulus of continuity $w$ for $\{c_t(\cdot, y)\}_{y\in \YC, t \in [0,1]}$. Hence, for any $t \in [0,1]$ we have $c,c_t\in C(\norm{c}_\infty + \norm{\Delta^c}_\infty+1,w)$ (see \Cref{lem:LowerUpperBound} for the definition of $C(\cdot,\cdot)$). Consequently, we infer by \Cref{lem:LowerUpperBound} the inequalities, %
    \begin{align*}
    &\frac{1}{t}\left( \inf_{\pi\in \Pi^\star_{c_t}(\mu_t, \nu_t)}\pi(t\Delta^{c}) + \sup_{f \in S\!_c(\mu,\nu)} t\Delta^\mu(f^{cc}) +  t\Delta^\nu(f^c)\right)\notag\\
    \leq&\frac{1}{t}(OT(\mu_t, \nu_t,c_{t})-OT( \mu, \nu,c))\\
    \leq&\frac{1}{t}\left(\inf_{\pi \in \Pi_c^\star(\mu,\nu)} \pi (t\Delta^c) +\sup_{f \in S_{\!c_t}(\mu_t, \nu_t)} t\Delta^\mu (f^{cc}) +   t\Delta^\nu (f^c)  + \sup_{f \in \FC} t\Delta^\mu(f^{ c_t c_t} - f^{cc}) +  t\Delta^\nu(f^{ c_t} - f^{c})\right).
    \end{align*}
    Next, we observe that $\Delta^\mu=\tilde{\mu}-\mu$ for some $\tilde{\mu}\in\PC(\X)$. This yields using Lipschitzianity under cost transformations with respect to the cost function (\Cref{lem:CtrafoLipschitz}) that
     \[\sup_{f\in \FC}\left|\Delta^\mu(f^{ c_tc_t} - f^{cc})\right|=\sup_{f\in \FC}\left|(\tilde{\mu}-\mu)(f^{ c_tc_t} - f^{cc})\right|\leq 4\norm{c_t-c}_\infty=4t\norm{\Delta^c}_\infty\xrightarrow{t\to 0}0.\]
    Likewise, it follows that  $|\sup_{f\in \FC}\Delta^\nu(f^{ c_t} - f^{c})|\to 0$ for $t\to0$. Finally, since the pair $(\mu_t, \nu_t)$ weakly converges for $t\searrow0$ to $(\mu, \nu)$ it follows by \Cref{lem:ContinuityResults} that
    \begin{align*}
      &\liminf_{t\searrow 0} \!\inf_{\pi \in \Pi_{c_t}^\star(\mu_t,\nu_t)} \pi (\Delta^c) +\sup_{f \in S_{\!c}(\mu, \nu)} \Delta^\mu(f^{cc}) +  \Delta^\nu (f^c)\\ %
   \geq & \inf_{\pi \in \Pi_c^\star(\mu,\nu)} \pi(\Delta^c) +\sup_{f \in S_{\!c}(\mu,\nu)} \Delta^\mu(f^{cc}) +  \Delta^\nu(f^c) \\ 
   \intertext{ as well as}
       &\limsup_{t\searrow 0} \!\inf_{\pi \in \Pi_c^\star(\mu,\nu)} \pi (\Delta^c) +\sup_{f \in S_{\!c_t}(\mu_t, \nu_t)} \Delta^\mu(f^{cc}) +  \Delta^\nu (f^c)\\ %
    \leq & \inf_{\pi \in \Pi_c^\star(\mu,\nu)} \pi(\Delta^c) +\sup_{f \in S_{\!c}(\mu,\nu)} \Delta^\mu(f^{cc}) +  \Delta^\nu(f^c), 
    \end{align*}
    which yields the claim.\qed
\newpage

\section{Proofs for Section \ref{subsec:RegEl}: Regularity Elevation Functionals}
\label{ap: RegElev}
\subsection{Proof of Proposition \ref{thm:RegElevationAbstract}}
  By the functional delta method \citep{Roemisch04} and the assumptions on $\Psi$ and $\LC$ it follows that $$
  a_{n}
  \begin{pmatrix}
  (f_{n} - f)\\
  (\Psi(f_{n}) - f)
  \end{pmatrix}
  \dto 
  \begin{pmatrix}
  \LC\\
  D^H_{f}\Psi(\LC)
  \end{pmatrix}
  \eqd
  \begin{pmatrix}
  \LC\\
  \LC
  \end{pmatrix}
  \quad \text{for } n \rightarrow \infty.$$
  The continuous mapping theorem  \citep[Theorem 1.11.1]{van1996weak} in combination with measurability of the random elements $f_{n}$ and $\Psi(f_{n})$ (due to continuity $\Psi$ near $f$) thus asserts  
  \[
  \pushQED{\qed} 
    a_{n}\big(\Psi(f_{n}) - f_{n}\big)\pto 0 \quad \text{for } n\rightarrow \infty.\qedhere
  \popQED
\]

\subsection{Proof of Proposition \ref{prop:RegEl:Bdd}}
  First note that $\Psi(\tilde c) \in C(\XC\times \YC)$ for any $\tilde c\in C(\XC\times \YC)$ as a concatenation of continuous functions and under $\norm{\tilde c}_\infty<2$ that $\Psi(\tilde c) = \tilde c$, which yields $\Psi(c) = c$.
  In particular, this shows that $\Psi\colon C(\XC\times \YC)\rightarrow C(\XC\times \YC)$ is continuous near $c$. 
   For Hadamard differentiability at $c$ consider a positive sequence $t_n \searrow 0$ and take a converging sequence $(h_n)_{n\in \NN}\subseteq C(\XC\times \YC)$ with limit $h$. Since $h$ is bounded and $\norm{c}_\infty\leq 1$, for $n$ sufficiently large we have  $\norm{c + t_n h_n}_\infty < 2$ and therefore $\Psi(c + t_n h_n) = c +t_n h_n$. We then obtain  \begin{align*}
          & \norm{\frac{\Psi(c + t_n h_n) - \Psi(c)}{t_n}  - h}_\infty = \norm{h_n -h}_\infty  \rightarrow 0. 
    \end{align*}
Finally, since for any $\tilde c\in C(\XC\times \YC)$ it holds that $\|g^{\Psi(\tilde c)}\|_\infty\leq B+2$ where $B\coloneqq \sup_{g\in \GC}\norm{g}_\infty$  we find for a finite space  $\XC$ that 
\[
\pushQED{\qed} 
  \sup_{\tilde c\in C(\XC\times \YC)}\log\NC(\eps, \GC^{\Psi(\tilde c)}, \norm{\cdot}_\infty)\leq |\XC|(\log(B+2)+|\log(\eps)|)\lesssim |\log(\eps)|. \qedhere
  \popQED
\]
\subsection{Proof of Proposition \ref{prop:RegEl:Mod}}
  By condition \eqref{eq:modLipschitz} it follows for $x,x'\in \XC$ with $\tilde d_\XC(x,x')=0$ that $c(x,y)=c(x',y)$, whereas under $\tilde d_\XC(x, x')>0$ we have by $w(\delta)>0$ for  $\delta>0$ that 
\begin{align*}
  c(x,y) \leq c(x',y) + w(\tilde d_\XC(x,x'))<c(x',y) + 2w(\tilde d_\XC(x,x')).
\end{align*}
This asserts for any $(x,y) \in \XC\times \YC$ that
\begin{align*}
  S(c, (x,y)) &\coloneqq \argmin_{x'\in \XC}c(x',y) + 2w\left(\tilde d_\XC(x,x')\right)= \{x'' \in \XC \;|\; \tilde d(x,x'') = 0\},
\end{align*}
and overall yields by $\norm{c}_\infty\leq 1$ that $\Psi(c) = c$. 

For the second and third claim, recall from \Cref{prop:RegEl:Bdd} that $\Psi_{\textup{bdd}}\colon C(\XC\times \YC)\rightarrow C(\XC\times \YC)$ is continuous near $c$ and Hadamard differentiable at $c$ with derivative $\Id_{C(\XC\times \YC)}$. Hence, it  suffices to verify that $\Psi_{\textup{mod}}^{w\circ \tilde d_\XC}$ is continuous near $c$ and Hadamard directionally differentiable with $D^H_{|c}\Psi|_{C(\tilde \XC\times \YC)}= \Id_{C(\tilde \XC\times \YC)}$ for which we rely on \Cref{lem:CtyPsi} and \Cref{thm:DiffPsi}. 
Define the spaces $V\coloneqq C(\XC\times \YC)$, $\FC =  \XC$, $\Theta = \tilde \XC\times \YC$ and the functional $$E^{w\circ \tilde d_\XC}\colon V \times \FC\times \Theta = C(\XC\times \YC)\times   \XC\times ( \tilde \XC\times \YC)\mapsto \RR, \quad (\tilde c, x', (x,y)) \mapsto -\tilde c(x',y) - 2 w(\tilde d_\XC(x,x')).$$ 
For any $\tilde c\in C(\XC\times \YC)$ the function $E^{w\circ \tilde d_\XC}(\tilde c, \cdot, \cdot)\colon  \XC\times ( \tilde \XC\times \YC)\rightarrow \RR$ is continuous as a sum of continuous functions. Further, for any $(x', (x,y))\in \XC\times (\tilde \XC\times \YC)$ note that the function $E^{w\circ \tilde d_\XC}(\cdot,x', (x,y))\colon C(\XC\times \YC)\rightarrow \RR$ is $1$-Lipschitz under uniform norm and that 
$$\Delta_{c}E^{w\circ \tilde d_\XC}(\tilde c, x', (x,y)) \coloneqq  E^{w\circ \tilde d_\XC}(\tilde c  + c,x', (x,y)) - E^{w\circ \tilde d_\XC}(c,x', (x,y))= -\tilde c(x',y)$$ 
is linear in $\tilde c \in C(\XC\times \YC)$. Hence, by \Cref{lem:CtyPsi} we obtain continuity of the functional 
\begin{align*}
\Psi_{\textup{mod}}^{w\circ \tilde d_\XC} \colon& C(\XC\times \YC) \rightarrow C(\tilde \XC\times \YC), \\&\;\tilde c \mapsto \left( (x,y) \mapsto \inf_{x' \in \XC} \tilde c(x',y) + 2w(\tilde d_\XC(x,x')) = -\sup_{x' \in \XC} E^{w\circ \tilde d_\XC}(\tilde c,x',(x,y))\right).\end{align*} 
Consider the closed sub-vector space $U\coloneqq C(\tilde \XC\times \YC)\subseteq C(\XC\times \YC)$, cf. \Cref{lem:pseudoMetricSpaceCompact}. It remains to show Assumption \textbf{(DC)} of \Cref{thm:DiffPsi}. To this end, note for $h\in C(\tilde \XC\times \YC)$ that $$h(\overline x,y) + 2w(\tilde d_\XC(\overline x,x')) = h(\overline x',y) + 2w(\tilde d_\XC(\overline x',x')) \quad \text{ for any }\overline x, \overline x' \in S(c, (x,y))$$  since $\tilde d_\XC(\overline x, \overline x') = 0$. This implies by \Cref{lem:LSC_con} that \textbf{(DC)} is fulfilled. 
 \Cref{thm:DiffPsi} thus asserts that $\Psi_{\textup{mod}}^{w\circ \tilde d_\XC}$ is Hadamard directionally differentiable at $c$ with derivative given by 
\begin{align*}
 D^H_{| c}\Psi^{w\circ \tilde d_\XC}_{\textup{mod}} \colon& C(\XC\times \YC) \rightarrow C(\tilde \XC\times \YC),\\ &h\mapsto\left( (x,y) \mapsto \!\!\!\!\!\!\inf_{x'\colon \tilde d_\XC(x',x) = 0} \!\!\!\!\!\!h(x',y)  = -  \!\!\!\!\!\!\sup_{x'\colon \tilde d_\XC(x',x) = 0}\!\!\!\!\!\!-\Delta_{c}E^{w\circ \tilde d_\XC}(h, x', (x,y))  \right).\end{align*}
Hence, if $h\in C(\tilde \XC\times \YC)$, then $D^H_{| c}\Psi^{w\circ \tilde d_\XC}_{\textup{mod}}(h) = h$, which yields $D^H_{| c}\Psi^{w\circ \tilde d_\XC}_{\textup{mod}}|_{C(\tilde \XC\times \YC)} = \Id_{C(\tilde \XC\times \YC)}$.

For the last claim note that any $\tilde c\in C(\XC\times \YC)$ fulfills for $(x,y) \in \XC\times \YC$ that 
\begin{align*}
  -\norm{\tilde c}_\infty\leq\Psi^{w\circ \tilde d_\XC}_{\textup{mod}}(\tilde c)(x,y) \leq \tilde c(x,y)\leq \norm{\tilde c}_\infty,
\end{align*}
and hence $\|\Psi(\tilde c)\|_\infty = \|\Psi^{w\circ \tilde d_\XC}_{\textup{mod}}\circ \Psi_{\textup{bdd}}(\tilde c)\|_\infty \leq 2$.
Further, for any $x,x'\in\XC, y\in \YC$ we have
\begin{align}
  		\Psi^{w\circ \tilde d_\XC}_{\textup{mod}}(\tilde c)(x,y)  - \Psi^{w\circ \tilde d_\XC}_{\textup{mod}}(\tilde c)(x', y) &\leq \inf_{x''\in \XC} c(x'',y) + 2w(\tilde d_\XC(x'', x)) - c(x'',y) - 2w(\tilde d_\XC(x'', x')) \notag \\
 		 &\leq 2w(\tilde d_\XC(x,x')),\label{eq:RegElModBounded}
\end{align}
where we used the reverse triangle inequality since $w\circ \tilde d_\XC$ defines a (pseudo-)metric on $\XC$. We thus conclude for any $\tilde c\in C(\XC\times \YC)$ and a bounded function class $\GC$ with $B\coloneqq \sup_{g\in \GC}\norm{g}_\infty<\infty$ from $\norm{\Psi(\tilde c)}_\infty \leq 2$ and \eqref{eq:RegElModBounded} that the elements of $\GC^{\Psi(\tilde c)}$ are bounded by $B+2$ and $2$-Lipschitz under $w\circ \tilde d_\XC$ as an infimum over such $2$-Lipschitz functions. Hence, $\GC^{\Psi(\tilde c)}\subseteq \BL_{(B+2),2}(\XC, w\circ \tilde d_\XC)$ where for the latter class uniform metric entropy bounds are available by \citet[Section~9]{Kolmogorov1961}, asserting for any $\eps>0$
\begin{align*}\pushQED{\qed} 
  \NC(\eps,\BL_{(B+2),2}(\XC, w\circ \tilde d_\XC), \norm{\cdot}_\infty) &= \NC(\eps/2,\BL_{(B+2)/2,1}(\XC, w\circ \tilde d_\XC), \norm{\cdot}_\infty) \\&\lesssim \NC(\eps/8, \XC, w\circ \tilde d_\XC) |\log(\eps)|.\qedhere
  \popQED
\end{align*}

\subsection{Proof of Corollary \ref{thm:RegElModulus}}
  We infer from \Cref{prop:RegEl:Mod} that $\Psi \coloneqq B \cdot \Psi_{\textup{mod}}^{w\circ d_\XC/B}\circ \Psi_{\textup{bdd}}(\cdot/B)$ is continuous near $c$ and Hadamard differentiable at $c$ with derivative $D^H_{|c}\Psi = \Id_{C(\XC\times \YC)}$. Hence, invoking \Cref{thm:RegElevationAbstract} it follows that  $a_n(\overline c_n - c_n)\pto 0$ for $n\rightarrow\infty$. 
  Moreover, by definition of $\Psi_{\textup{bdd}}$ and $\Psi_{\textup{mod}}^{w/B, d_\XC}$ it follows that $\norm{\overline c_n}_\infty \leq 2B$ and that $\overline c_n$ fulfills \eqref{eq:modLipschitz} with $w$ replaced by $2w$. The inclusion now follows at once from \Cref{lem:RelFCandConcave}.\qed

	\subsection{Proof of Proposition \ref{prop:RegEl:Hol}}
  Since $c$ is $(\gamma,1)$-H\"older it follows for any $x,x'\in \XC$ and $y \in \YC$ as in Lemma A.4 of \citet{hundrieser2022empirical} by convexity of $\XC$ that 
  \begin{align*}
    c(x,y) = c(x',y) + \langle \nabla_x c(x',y), x-x'\rangle + R_{x'}(x) \quad \text{ with } \quad |R_{x'}(x)| \leq \sqrt{d}\norm{x-x'}^\gamma,
  \end{align*}
  and consequently, for $x\neq x'$ we obtain
  \begin{align*}
    c(x,y) &< c(x',y) + \langle \nabla_x c(x',y), x-x'\rangle + 2\sqrt{d}\norm{x-x'}^\gamma.
  \end{align*}
  This asserts for any $(x,y)\in \XC\times \YC$ that 
  \begin{align*}
   S(c,(x,y)) &\coloneqq \argmin_{x'\in \XC}c(x',y) + \langle \nabla_x c(x',y), x-x'\rangle + 2\sqrt{d}\norm{x-x'}^\gamma =\{x\}
  \end{align*}
  and yields by $\norm{c}_\infty\leq 1$ that $\Psi(c) = c$.

To show the claim on continuity and  Hadamard differentiability it suffices to verify that $\Psi_{\textup{Hol}}$ is continuous near $c$ and that is Hadamard differentiable at $c$ with derivative $D^H_{|c}\Psi_{\textup{Hol}} = \Id_{C(\XC\times \YC)}$ for which we rely on \Cref{lem:CtyPsi} and \Cref{thm:DiffPsi}. Set $V\coloneqq C(\XC\times \YC)$, $\FC \coloneqq\XC$ and $\Theta \coloneqq \XC\times \YC$ and define the functional 
\begin{align*}
  E_{\textup{Hol}}\colon &V\times \FC\times \Theta = C(\XC\times \YC)\times \XC\times (\XC\times \YC) \rightarrow \RR, \\
  &(\tilde c,x',(x,y))\mapsto -\left(\tilde c(x',y) +  \langle \nabla_x c(x',y), x-x'\rangle + 2\sqrt{d}\norm{x-x'}^\gamma\right).
\end{align*}
For any $\tilde c\in C(\XC\times \YC)$ the functional  $E_{\textup{Hol}}(\tilde c, \cdot, \cdot)\colon \XC\times (\XC\times \YC)\rightarrow \RR$ is continuous by continuity of $\nabla_{x}c(\cdot, \cdot)$ and for any $(x',(x,y))\in \XC\times (\XC\times \YC)$ the functional $E_{\textup{Hol}}(\cdot,x',(x,y))\colon C(\XC\times \YC)\rightarrow \RR$ is $1$-Lipschitz under uniform norm while $$\Delta_{c}E_{\textup{Hol}}(\tilde c, x',(x,y)) = E_{\textup{Hol}}(\tilde c + c,x',(x,y)) - E_{\textup{Hol}}(c,x',(x,y)) = -\tilde c(x',y)$$ is linear in $\tilde c \in C(\XC\times \YC)$. Finally, condition \textbf{(DC)} follows by \Cref{lem:LSC_con} since $S(c,(x,y))=\{x\}$ is a singleton. 
Hence, by \Cref{lem:CtyPsi} and \Cref{thm:DiffPsi} the functional 

$$\Psi_{\textup{Hol}} \colon C(\XC\times \YC) \rightarrow C(\XC\times \YC), \;\tilde c \mapsto \left( (x,y) \mapsto -\sup_{x' \in \XC}  E_{\textup{Hol}}(\tilde c,x',(x,y))\right)$$
is continuous near $c$ and Hadamard differentiable at $c$ with derivative 
$$ D^H_{| c}\Psi_{\textup{Hol}} \colon C(\XC\times \YC) \rightarrow C(\tilde \XC\times \YC), \quad h\mapsto\Big( (x,y) \mapsto h(x,y)  = - \Delta_{c}E_{\textup{Hol}}(h, x',(x,y)) \Big).$$

For the claim on the uniform metric entropy bound let $\tilde c\in C(\XC\times \YC)$, and assume (after application of $\Psi_{\textup{bdd}}$) that $\norm{\tilde c}_\infty \leq 2$. Define the collection of functions $(\tilde E_{x',y})_{x'\in \XC, y\in \YC}$ with 
\begin{align*}
  \tilde  E_{x',y}\colon \XC\rightarrow \RR, \quad x\mapsto \tilde c(x',y) +  \langle \nabla_x c(x',y), x-x'\rangle + 2\sqrt{d}\norm{x-x'}^\gamma,
\end{align*}
which is $(\gamma,2)$-H\"older on $\XC$. Hence, by \citet[Lemma A.5]{hundrieser2022empirical}  there exists another collection $(\tilde E_{x',y}^\sigma)_{x'\in \XC, y\in \YC, \sigma\in (0,1]}$ of smooth functions on $\XC$ such that 
\begin{align}\label{eq:HolderMoreRegularlyBehavingFunction}
  \sup_{\substack{x'\in \XC\\y\in \YC}} \norm{ \tilde  E_{x',y} -  \tilde  E_{x',y}^\sigma}_\infty \leq K \sigma^{\gamma} \quad \text{ and } \quad \sup_{\substack{x'\in \XC\\y\in \YC}} \norm{  \tilde  E_{x',y}^\sigma}_{C^2(\XC)} \leq K \sigma^{\gamma-2},
\end{align} for all $\sigma >0$ and some independent $K>0$. Here, the $C^2(\XC)$-norm of a twice continuously differentiable function $g : \XC\subset\RR^d \to \RR$ is defined as
\[\norm{g}_{C^2(\XC)}\coloneqq \max_{|\beta|\leq2}\norm{D^\beta g}_\infty,~~\text{ where }~~D^\beta g=\partial^{|\beta|} g/\partial x_1^{\beta_1}\cdots x_d^{\beta_d}~\text{ for }\beta\in\NN^d_0.\]
Note that a function with $\norm{g}_{C^2(\XC)}\leq \Gamma$ for $\Gamma>0$ is absolutely bounded by $\Gamma$, it is $\Gamma$-Lipschitz,
and $d\Gamma$-semi-concave (for a formal definition see \cite{albano2002some} or \cite{hundrieser2022empirical}), since the Eigenvalues of its Hessian are bounded by $d \cdot \Gamma$. 
Upon defining $\overline c(x,y)\coloneqq \Psi(\tilde c)(x,y) = \inf_{x'\in \XC}\tilde  E_{x',y}(x)$ and $\overline c^\sigma(x,y)\coloneqq \inf_{x'\in \XC}\tilde  E^\sigma_{x',y}(x)$ we thus obtain from \eqref{eq:HolderMoreRegularlyBehavingFunction} that $\norm{\overline c - \overline c^\sigma}_\infty \leq K \sigma^\gamma$ and that $\overline c^\sigma$ is semi-concave of order $\Gamma(\sigma)\coloneqq d K\sigma^{\gamma - 2}$.
Hence, following along the lines the proofs of Lemma A.4 in \cite{hundrieser2022empirical} we obtain for any $\eps>0$ and $\sigma(\eps) \coloneqq (\eps/4K)^{1/\gamma}$ that 
\begin{align*}
  \!\!\!\!\NC(\eps,\GC^{\overline c}, \norm{\cdot}_\infty)&\leq \NC(\eps/2,\GC^{\overline c^{\sigma(\eps)}}, \norm{\cdot}_\infty)= \NC\left(\frac{\eps}{2\Gamma(\sigma(\eps))},\frac{\GC^{ \overline c^{\sigma(\eps)}}}{2\Gamma(\sigma(\eps))}, \norm{\cdot}_\infty\right)\lesssim \left(\frac{\eps}{2\Gamma(\sigma(\eps))}\right)^{-d/2} \!\!\!\lesssim \eps^{-d/\gamma}.
\end{align*}
Here, we used in the second inequality that $\GC^{\overline c^{\sigma(\eps)}}/2\Gamma(\sigma(\eps))$ is contained in the collection of functions on $\XC$ which are absolutely bounded by $B\geq 0$, Lipschitz with modulus $L\geq 0$ and $1$-semi-concave, where $B$ depends on $\GC$ and $L$ depends on $\XC$, in conjunction with uniform metric entropy bounds by \citet{bronshtein1976varepsilon,guntuboyina2012covering} for convex functions. In particular, since the hidden constants do not depend on $\tilde c$,  the claim follows. \qed

\subsection{Proof of Proposition \ref{prop:RegEl:Union}}

  For the first claim note that $\Psi_i(c_i)= c_i$ for each $i \in \{1, \dots, I\}$ and consequently, it follows for $x\in \zeta_i(\UC_i)$ that $\Psi_i(c_i)(\zeta_i^{-1}(x),y) = c(x,y)$. Hence, since $\sum_{i = 1}^I \eta_i(x) \equiv 1$ it follows that $\Psi(c) = c$. 
  
 	The claim on continuity of $\Psi$ near $c$ follows by continuity of the functionals $\Psi_i\colon C(\UC_i\times \YC)\rightarrow C(\UC_i\times \YC)$ near $c_i$ for each $1 \leq i \leq I$. 
  
  For the claim on Hadamard differentiability of $\Psi$ define for each $i\in \{1, \dots, I\}$ the functionals 
  \begin{align*}
    \Psi_{\textup{com},i}^{1}&\colon C(\XC\times \YC)\rightarrow C(\UC_i\times \YC), \!\!\!\!\!\!&& \tilde c\mapsto \left((u,y)\mapsto \tilde c(\zeta_i(u),y)\right),\\
    \Psi_{\textup{com},i}^{2}&\colon C(\UC_i\times \YC)\rightarrow C(\zeta_i(\UC_i)\times \YC), \!\!\!\!\!\!&& \tilde c\mapsto \left((x,y)\mapsto \tilde c(\zeta_i^{-1}(x),y)\right),
  \end{align*} 
  where both maps assign to the respective spaces of continuous functions since $\zeta_i^{-1}$ and $\zeta_i$ are both continuous. Further, note for any $\tilde c\in C(\XC\times \YC)$ that $\Psi(\tilde c)=\sum_{i = 1}^{I}\eta_i\cdot \Psi_{\textup{com},i}^{2}\circ \Psi_i \circ \Psi_{\textup{com},i}^{1}(\tilde c)$. Both functionals $\Psi_{\textup{com},i}^{1}$ and $\Psi_{\textup{com},i}^{2}$ are Hadamard differentiable at $c_i$ with derivative 
  \begin{align*}
    D^H_{|c_i}\Psi_{\textup{com},i}^{1}&\colon C(\XC\times \YC)\rightarrow C(\UC_i\times \YC), \quad h\mapsto \left( (u,y)\mapsto h(\zeta_i(u),y) \right),\\
    D^H_{|c_i}\Psi_{\textup{com},i}^{2}&\colon C(\UC_i\times \YC)\rightarrow C(\zeta_i(\UC_i)\times \YC), \quad h\mapsto \left( (x,y)\mapsto h(\zeta_i^{-1}(x),y) \right).
  \end{align*}
  By assumption on $\Psi_i$ and chain rule we infer that $\Psi$ is Hadamard differentiable at $c$ with derivative 
  \begin{align*}
    \!\!\!D^H_{c}\Psi\colon C(\XC\times \YC)& \to C(\XC\times \YC), \\
   h& \mapsto \left((x,y) \mapsto \sum_{i = 1}^{I}\eta_i(x) h(\zeta_i^{-1}(\zeta_i(x)),y) = \sum_{i = 1}^{I}\eta_i(x)  h(x,y) = h(x,y)\right)
  \end{align*}
  and conclude that $D^H_{|c}\Psi = \Id_{C(\XC\times \YC)}$. 
  
  Finally, the bound on the covering numbers is a consequence of Lemma 3.1 and Lemma A.1 in \cite{hundrieser2022empirical} as they assert for arbitrary $\tilde c\in C(\XC\times \YC)$ that 
  \begin{align*}
    \pushQED{\qed}
    \log\NC(\eps, \GC^{\Psi(\tilde c)}, \norm{\cdot}_\infty) &
    \leq \sum_{i = 1}^{I} \log\NC(\eps, \GC^{ \Psi(\tilde c)}|_{\zeta_i(\UC_i)}, \norm{\cdot}_\infty)\\  
    & \leq \sum_{i = 1}^{I} \log\NC(\eps,\GC^{ \Psi_i(\tilde c)}\circ \zeta_i, \norm{\cdot}_\infty)\\
    &  = \sum_{i = 1}^{I} \log\NC(\eps, \GC^{ \Psi_i(\tilde c(\zeta_i(\cdot),\cdot))}, \norm{\cdot}_\infty).\qedhere
    \popQED
  \end{align*}

\section{Proofs for Section \ref{sec: ProofMain}: Lemmata of Distributional Limits}\label{app:ProofSection2}

\subsection{Proof of Lemma \ref{lem:CtrafoLipschitz}}
Assume $\|f- \tilde f\|_\infty + \norm{c- \tilde c}_\infty<\infty$ since otherwise the claim is vacuous. 
For $\tilde f$ and $\tilde c$ there exists for $y \in \YC$ and $\eps>0$ some $x'\in\XC$  such that $\tilde f^{\tilde c}(y) \geq  \tilde c(x',y) - \tilde f(x') - \eps$. Hence, 
\begin{align*}
f^c(y) - \tilde f^{\tilde c}(y) =&\; \left[\inf_{x \in \XC} c(x,y) - f(x)\right] - \left[\inf_{x \in \XC} \tilde c(x,y) - \tilde f(x)\right]\\
\le& \; c(x',y) - f(x') - \tilde c(x',y) + \tilde f(x') + \eps\\
\le& \; \norm{f- \tilde f}_\infty + \norm{c- \tilde c}_\infty + \eps. 
\end{align*}
As $\eps>0$ can be chosen arbitrarily small, we obtain for any $y\in \YC$ the inequality 
\begin{align*}  
f^c(y) - \tilde f^{\tilde c}(y) \leq \norm{f- \tilde f}_\infty + \norm{c- \tilde c}_\infty.
\end{align*} 
Repeating the argument for $f$ and $c$ asserts the converse inequality and proves the claim.\qed
\subsection{Proof of Lemma \ref{lem:LowerUpperBound}}
Let us start by splitting the problem in two different ways,  
\begin{align*}
OT( \tilde\mu, \tilde\nu, {\tilde c}) -OT ( \mu, \nu,c) 
    &= (OT( \tilde\mu, \tilde\nu, {\tilde c}) - OT (\tilde \mu, \tilde\nu,c ))+ (OT ( \tilde\mu, \tilde\nu, c )- OT ( \mu, \nu, c))\\
    &= (OT( \tilde\mu, \tilde\nu, {\tilde c}) - OT(\mu, \nu, {\tilde c}))+ (OT (\mu,\nu, {\tilde c})- OT ( \mu, \nu, c)).
\end{align*}
Since $c,\tilde c\in C(2\norm{c}_\infty+1, 2w)$, we can employ the dual representation of the OT value from \Cref{lem:RelFCandConcave} with $\FC= \FC(2\norm{c}_\infty+1, 2w)$. 
Hence, for each bracket in the display above, one can choose to plug-in a feasible plan in the primal formulation or a potential from $\FC$ in the dual formulation to obtain upper and lower bounds. 
Doing so, we obtain
\begin{align*}
    \inf_{\pi \in \Pi_{\tilde c}^\star(\tilde\mu,\tilde\nu)} \pi(\tilde c-c)\leq OT( \tilde\mu, \tilde\nu, {\tilde c}) - OT (\tilde \mu, \tilde\nu, c )&\leq \inf_{\pi \in \Pi_c^\star(\tilde\mu,\tilde\nu)} \pi (\tilde c-c),\\
    \sup_{f \in S\!_c(\mu,\nu)} (\tilde \mu-\mu ) f^{cc} +  (\tilde \nu-\nu ) f^c \leq OT ( \tilde\mu, \tilde\nu,c )- OT ( \mu, \nu,c) &\leq \sup_{f \in S\!_c(\tilde\mu,\tilde\nu)} (\tilde \mu-\mu ) f^{cc} +  (\tilde \nu-\nu ) f^c,\\
    OT(\mu,\nu, {\tilde c} )- OT ( \mu, \nu, c) &\leq \inf_{\pi \in \Pi_c^\star(\mu,\nu)} \pi (\tilde c-c),\\
    OT( \tilde\mu, \tilde\nu, {\tilde c}) - OT(\mu, \nu, {\tilde c}) &\leq \sup_{f \in S_{\!\tilde c}(\tilde \mu,\tilde \nu)} (\tilde \mu-\mu ) f^{\tilde c\tilde c} +  (\tilde \nu-\nu ) f^{\tilde c}.
\end{align*}
In particular, for the last upper bound we further note that $$  \sup_{f \in S_{\!\tilde c}(\tilde \mu,\tilde \nu)} (\tilde \mu-\mu ) f^{\tilde c\tilde c} +  (\tilde \nu-\nu ) f^{\tilde c} \leq  \sup_{f \in S_{\!\tilde c}(\tilde \mu,\tilde \nu)} (\tilde \mu-\mu ) f^{cc} +  (\tilde \nu-\nu ) f^c  + \sup_{f \in \FC}(\tilde \mu-\mu)(f^{\tilde c\tilde c} - f^{cc}) + (\tilde \nu-\nu)(f^{\tilde c} - f^{c}),$$
which overall yields the lower and upper bounds for the OT cost under varying measures and costs. 

Finally, the bound under fixed measures $\mu, \nu$ it follows by H\"older's inequality for any $\pi\in \Pi(\mu, \nu)$ that $|\pi(\tilde c -c)| \leq \norm{\tilde c-c}_\infty$, whereas under a fixed cost function $c$ we have    
\begin{align*}
  \pushQED{\qed}  \sup_{f \in S\!_c(\tilde \mu,\tilde \nu)\cup S\!_c(\mu,\nu)} \big|(\tilde \mu-\mu ) f^{cc} +  (\tilde \nu-\nu ) f^c\big| &\leq
\sup_{f \in \FC} \big|(\tilde \mu-\mu ) f^{cc}\big| +  \sup_{f \in \FC} \big|(\tilde \nu-\nu ) f^c\big| \\
 &= \sup_{f \in \FC^{cc}} \big|(\tilde \mu-\mu ) f\big| +  \sup_{f \in \FC^{c}} \big|(\tilde \nu-\nu ) f\big|.  \qedhere
 \popQED
\end{align*}

\subsection{Proof of Lemma \ref{lem:ContinuityResults}}
 The continuity of $T_1$ is a consequence of \citet[Theorem 5.20]{villani2008optimal}. Indeed, any converging sequence $(\mu_n, \nu_n, c_n)$ with limit $(\mu_\infty, \nu_\infty, c_\infty)$ admits a sequence of OT plans $\pi_n\in \Pi_{c_n}^\star(\mu_n, \nu_n)$ which  converges weakly along a subsequence, say $(\pi_{n_k})_{k\in \NN}$, to an OT plan $\pi_\infty\in \Pi_{c_\infty}^\star(\mu_\infty, \nu_\infty)$. Hence, 
 \begin{align*}
  \limsup_{k\rightarrow \infty}  |T_1(\mu_{n_k}, \nu_{n_k}, c_{n_k}) - T_1(\mu_\infty, \nu_\infty, c_\infty)| &= \limsup_{k\rightarrow \infty} |OT(\mu_{n_k}, \nu_{n_k}, {c_{n_k}}) - OT(\mu_\infty,\nu_\infty, {c_\infty})| \\
    &=\limsup_{k\rightarrow \infty} |\pi_{n_k}(c_{n_k}) - \pi_\infty(c_\infty)|\\
     &\leq  \limsup_{k\rightarrow \infty} |(\pi_{n_k} - \pi_\infty)(c_\infty)| + \norm{c_{n_k} - c_\infty}_\infty = 0.
 \end{align*}
Since this holds for any sequence of converging OT plans, continuity of $T_1$ follows from \Cref{lem:limsup}.
 
 For the lower semi-continuity of $T_2$ take a sequence $(h_{c,n})_{n\in\NN}$ with limit $h_{c,\infty}$ and consider OT plans $\pi_n\in \Pi_{c_n}^\star(\mu_n, \nu_n)$ such that $$\inf_{\pi \in \Pi_{ c_n}^\star(\mu_n,\nu_n)} \pi(h_{c,n}) \geq \pi_n(h_{c,n}) - 1/n.$$
 Then, by \citet[Theorem 5.20]{villani2008optimal} a converging subsequence $(\pi_{n_k})_{k\in \NN}$ with limit $\pi_\infty\in \Pi_{c_\infty}^\star(\mu_\infty, \nu_\infty)$ exists and it follows that
  \begin{align*}
  \liminf_{k \rightarrow \infty} T_2(\mu_{n_k}, \nu_{n_k}, c_{n_k},h_{c,{n_k}})&= \liminf_{k\rightarrow \infty} \inf_{\pi \in \Pi_{ c_{n_k}}^\star(\mu_{n_k},\nu_{n_k})} \pi(h_{c,{n_k}}) \\
    &\geq \liminf_{k\rightarrow \infty} \pi_{n_k}(h_{c,{n_k}}) - 1/{n_k}\\
    &\geq  \liminf_{k\rightarrow \infty} \pi_{n_k}(h_{c,{\infty}}) - \norm{h_{c,{\infty}} - h_{c,{n_k}}}_\infty - 1/{n_k}\\
    &=  \pi_{\infty}(h_{c,{\infty}}) \geq T_2(\mu_{\infty}, \nu_{\infty}, c_{\infty},h_{c,{\infty}}).
 \end{align*}
 Consequently, by \Cref{lem:limsup}, lower semi-continuity of $T_2$ follows. 
To infer upper semi-continuity of $T_2$, and thus continuity, at $(\mu_\infty, \nu_\infty, c_\infty, h_{c,\infty})$ under the assumption of a unique OT plan $\pi^\star \in \Pi^\star_{c_\infty}(\mu_\infty, \nu_\infty)$ note by \citet[Theorem 5.20]{villani2008optimal} that for any sequence of OT plans $\pi_n\in \Pi^\star_{c_n}(\mu_n, \nu_n)$ there exists a weakly converging subsequence $\pi_{n_k}$ which tends to $\pi^\star$ for $k\rightarrow \infty$. Hence, we conclude that 
\begin{align*}
    \limsup_{k \rightarrow \infty} T_2(\mu_{n_k}, \nu_{n_k}, c_{n_k},h_{c,{n_k}})&= \limsup_{k\rightarrow \infty} \inf_{\pi \in \Pi_{ c_{n_k}}^\star(\mu_{n_k},\nu_{n_k})} \pi(h_{c,{n_k}}) \\
    &\leq \limsup_{k\rightarrow \infty} \pi_{n_k}(h_{c,{n_k}})\\
    &\leq  \limsup_{k\rightarrow \infty} \pi_{n_k}(h_{c,{\infty}}) - \norm{h_{c,{\infty}} - h_{c,{n_k}}}_\infty\\
    &=  \pi_{\infty}(h_{c,{\infty}}) = T_2(\mu_{\infty}, \nu_{\infty}, c_{\infty},h_{c,{\infty}}).
\end{align*}
This implies by \Cref{lem:limsup} the upper semi-continuity of $T_2$.  Moreover, for fixed $(\mu', \nu', c')$ the map $T_2$ is continuous in $h_c$ since for any $\tilde h_{c}$ it holds that $$|T_2(\mu, \nu, c, h_c) - T_2(\mu, \nu, c, \tilde h_c)| \leq \norm{h_c - \tilde h_c}_\infty.$$
 
 To show upper semi-continuity of $T_3$ take a sequence $(h_{\mu,n}, h_{\nu,n})_{n\in\NN}$ with limit $(h_{\mu,\infty}, h_{\nu,\infty})$. 
 Further, by definition of $C$, it follows from \Cref{lem:RelFCandConcave} that any $c_n \in C$ fulfills $\HC_{c_n}\subseteq \FC^{c_nc_n}\subseteq \FC$. 
 Take a sequence $f_n\in S_{\!c_n}(\mu_n, \nu_n)\subseteq \FC$ such that $$T_3(\mu_n, \nu_n, c_n, h_{\mu,n}, h_{\nu,n}) \leq h_\mu(f_n) + h_{\nu}(f_n) + 1/n.$$ By compactness of $\FC$ there exists a uniformly converging subsequence, say $(f_{n_k})_{k \in \NN}$, with limit $f_\infty\in \FC$. %
Next, we demonstrate that $f_\infty\in S_{\!c_\infty}(\mu_\infty, \nu_\infty)$. To this end, we note  \begin{align*}
     OT(\mu_\infty, \nu_\infty, {c_\infty}) &\geq \mu_\infty(f_{\infty}^{c_\infty c_\infty}) + \nu_\infty(f_{\infty}^{c_\infty}) \\
     &= \lim_{k \rightarrow \infty} \mu_{n_k}(f_{\infty}^{c_\infty c_\infty}) +  \nu_{n_k}(f_{\infty}^{c_\infty}) \\
     &\geq \lim_{k \rightarrow \infty} \mu_{n_k}(f_{n_k}^{c_{n_k} c_{n_k}}) +  \nu_{n_k}(f_{n_k}^{ c_{n_k}}) - \norm{f_{\infty}^{c_\infty c_\infty} - f_{n_k}^{c_{n_k} c_{n_k}}}_\infty-\norm{f_{\infty}^{c_\infty} - f_{n_k}^{c_{n_k}}}_\infty\\
     &= \lim_{k \rightarrow \infty} OT(\mu_{n_k}, \nu_{n_k}, {c_{n_k}})- \norm{f_{\infty}^{c_\infty c_\infty} - f_{n_k}^{c_{n_k} c_{n_k}}}_\infty-\norm{f_{\infty}^{c_\infty} - f_{n_k}^{c_{n_k}}}_\infty\\
     &= OT(\mu_\infty, \nu_\infty, {c_\infty}),
 \end{align*}
 where the last equality follows by continuity of $T_1$. Hence, we get $f_\infty\in S_{\!c_\infty}(\mu_\infty, \nu_\infty)$. 

By continuity of $h_{\mu}$ and $h_{\nu}$ on $\FC$ and upon denoting the norm on $C_u(\FC)$ by $\norm{\cdot}_\FC$, we infer that 
 \begin{align*}
     \limsup_{k \rightarrow \infty} T_3(\mu_{n_k}, \nu_{n_k}, &c_{n_k},h_{\mu,n_k}, h_{\nu,n_k}) = \limsup_{k\rightarrow \infty} \sup_{f\in S_{\!c_{n_k}}(\mu_{n_k}, \nu_{n_k})} h_{\mu,n_k}(f) + h_{\nu,n_k}(f)\\
     &\leq \limsup_{k\rightarrow \infty} h_{\mu,n_k}(f_{n_k}) + h_{\nu,n_k}(f_{n_k}) + 1/n_k\\
     &\leq \limsup_{k\rightarrow \infty} h_{\mu,\infty}(f_{n_k}) + h_{\nu,\infty}(f_{n_k}) + \norm{h_{\mu,\infty} - h_{\mu,n_k}}_{\FC} + \norm{h_{\nu,\infty} - h_{\nu,n_k}}_{\FC} + 1/n_k\\
     &= h_{\mu,\infty}(f_{\infty}) + h_{\nu,\infty}(f_{\infty}) \leq T_3(\mu_{\infty}, \nu_{\infty}, c_{\infty},h_{\mu,{\infty}}, h_{\nu,\infty})
 \end{align*}
 and consequently, by \Cref{lem:limsup}, upper semi-continuity of $T_3$ follows. Further, for fixed $(\mu', \nu', c')$ the map $T_3$ is continuous in $(h_\mu, h_{\nu})$ since for another $(\tilde h_\mu,\tilde  h_{\nu})$ it holds that $$|T_3(\mu, \nu, c, h_\mu, h_{\nu}) - T_3(\mu, \nu, c, \tilde h_\mu,\tilde  h_{\nu})| \leq \norm{\tilde h_\mu-\tilde h_\mu}_{\FC^{cc}} + \norm{\tilde h_\nu-\tilde h_\nu}_{\FC^{c}}.$$
  
 Finally, for $T_4$ take $(h_{1,\mu}, \tilde h_{1,\mu},h_{1,\nu}, \tilde h_{1,\nu})$, $(h_{2,\mu}, \tilde h_{2,\mu},h_{2,\nu}, \tilde h_{2,\nu})\in  C_u(\FC)^4$ and note that \begin{align*}
&|T_4(h_{1,\mu}, \tilde h_{1,\mu},h_{1,\nu}, \tilde h_{1,\nu}) - T_4(h_{2,\mu}, \tilde h_{2,\mu},h_{2,\nu}, \tilde h_{2,\nu})|\\\leq &\norm{h_{1,\mu} - h_{2,\mu}}_{\FC} + \norm{\tilde h_{1,\mu} - \tilde h_{2,\mu}}_{\FC} +\norm{h_{1,\nu} - h_{2,\nu}}_{\FC} + \norm{\tilde h_{1,\nu} - \tilde h_{2,\nu}}_{\FC},
 \end{align*}
 which asserts continuity. \qed
\subsection{Proof of Lemma~\ref{lem:measurability}}
  For $(i)$ take $f,g\in \GC$, then $|\mu(f) - \mu(g)|\leq \norm{f-g}_\infty$ and hence $\mu\colon \GC\rightarrow \RR$ defines a Lipschitz 
  map under uniform norm which asserts $\mu \in C_u(\GC)$. 
  Assertion $(ii)$ follows from \citet[p.\ 17]{gine2016mathematical}.
  Finally,  $(iii)$ follows from $(ii)$ since for any $g\in \GC$ the evaluations $\mu_n(g) = n^{-1}\sum_{i = 1}^{n} g(X_i)$ and $\muboot(g) = k^{-1}\sum_{i = 1}^{k} g(\Xib)$ are Borel measurable.\qed %
\subsection{Proof of Lemma~\ref{lem:JointWeakConvergence}}
We first prove that Assumption \ref{ass:AThree} implies for $n,m\rightarrow \infty$ with $m/(n+m)\rightarrow \lambda\in(0,1)$ that   \begin{align}\label{eq:LargerSpaceConvergence}
\begin{pmatrix}
\sqrt{n}\Big((\mu_n - \mu)(f^{cc})\Big)_{f\in \FC}\\
\sqrt{m}\Big((\nu_m - \nu)(f^{c}) \Big)_{f\in \FC}\\
\sqrt{\frac{nm}{n+m}}(c_{n,m} - c)
\end{pmatrix} = 
 \begin{pmatrix}
\Big(\Gproc_n^\mu(f^{cc})\Big)_{f\in \FC}\\
\Big(\Gproc_m^\nu(f^{c}) \Big)_{f\in \FC}\\
\Gproc_{n,m}^c
\end{pmatrix}\dto
\begin{pmatrix}
\Big(\Gproc^\mu(f^{cc})\Big)_{f\in \FC}\\
\Big(\Gproc^\nu(f^{c})\Big)_{f\in \FC}\\
\Gproc^c
\end{pmatrix}
\end{align}
in the Polish space $C_u(\FC)\times C_u(\FC)\times C(\X\times \Y)$. To this end, consider the map
\begin{align*}
\Psi\colon C_u(\FC^{cc})\times  C_u(\FC^{c})&\times C(\XC\times \YC) \rightarrow C_u(\FC)\times C_u(\FC)\times C(\XC\times \YC),\\
(\alpha,  \beta, \gamma) &\mapsto \left( \big(\alpha(f^{cc})\big)_{f\in \FC}, \big(\beta(f^{c})\big)_{f\in \FC}, \gamma \right).
\end{align*}
This map is well-defined (i.e., its range is correct) since for any $(\alpha,  \beta)\in C_u(\FC^{cc})\times  C_u(\FC^{c})$ there exist moduli of continuity $w_\alpha, w_\beta\colon \RR_+\rightarrow \RR_+$ such that for $f,\tilde f\in \FC$ it follows by \Cref{lem:CtrafoLipschitz} that  \begin{align*}
    |\alpha(f^{cc}) - \alpha(\tilde f^{cc})| &\leq w_\alpha(\norm{f^{cc} - \tilde f^{cc}}_\infty) \leq  w_\alpha(\norm{f - \tilde f}_\infty),\\
    |\beta(f^c) - \beta(\tilde f^c)| &\leq w_\beta(\norm{f^c - \tilde f^c}_\infty) \leq  w_\beta(\norm{f - \tilde f}_\infty),
\end{align*}
which assert that $( (\alpha(f^{cc}))_{f\in \FC}, (\beta(f^{c}))_{f\in \FC}, \gamma )\in C_u(\FC)\times C_u(\FC)\times C(\XC\times \YC)$. 
Moreover, for any $(\alpha,  \beta), (\tilde \alpha, \tilde \beta)\in C_u(\FC^{cc})\times  C_u(\FC^{c})$ we have 
\begin{align*}
\sup_{f\in \FC}|\alpha(f^{cc}) - \tilde \alpha(f^{cc})| =  \sup_{\tilde f\in \FC^{cc}}|\alpha(\tilde f) - \tilde \alpha(\tilde f)|  \quad \text{ and } \quad \sup_{f\in \FC}|\beta(f^{c}) - \tilde \beta(f^{c})| =  \sup_{\tilde f\in \FC^{c}}|\beta(\tilde f) - \tilde \beta(\tilde f)|,
\end{align*}
hence the map $\Psi$ is continuous. Consequently, Assumption \ref{ass:AThree} and the continuous mapping theorem \citep[Theorem 1.11.1]{van1996weak} assert weak convergence \eqref{eq:LargerSpaceConvergence}. 

Moreover, by \citet{varadarajan58} the empirical measures $( \mu_n, \nu_n)$ weakly converge a.s.\ in $\PC(\XC)\times\PC(\YC)$ to $(\mu, \nu)$. Note that $\PC(\XC)\times\PC(\YC)$ is by compactness of $\XC$ and $\YC$ a separable, complete metric space \citep{bolley2008separability}. Invoking Slutzky's lemma \citep[Example 1.4.7]{van1996weak} in conjunction with \eqref{eq:LargerSpaceConvergence} we thus obtain the first claim. In particular, by measurability of $c_{n,m}$ and \Cref{lem:measurability}, all involved quantities are Borel measurable.

For the second claim note by \Cref{lem:CtrafoLipschitz} that any realization of $\mu_n,\nu_m$ and $c_{n,m}$ leads the processes $\Gproc_n^\mu(f^{ c_{n,m} c_{n,m}})$ and $\Gproc_m^\nu(f^{ c_{n,m}})$ to be $2\!\sqrt{n}$-Lipschitz and $2\!\sqrt{m}$-Lipschitz in $f$, respectively. Thus, they are uniformly continuous in $f$. Moreover, for fixed $f\in \FC$ we can show %
 that the function 
\begin{align*}
    \tilde \Gproc_n^\mu\colon \PC(\XC)\times C(\XC\times\YC)\rightarrow \RR, \quad (\tilde\mu, \tilde c) \mapsto \sqrt{n}(\tilde\mu - \mu)(f^{\tilde c \tilde c})
\end{align*}
is upper semi-continuous (i.e., in particular measurable). Indeed, for $\tilde \mu_k\dto \tilde \mu$ in $\PC(\XC)$ and $\tilde c_k\rightarrow \tilde  c$ in $C(\XC\times\YC)$ it follows by \Cref{lem:CtrafoLipschitz}, upper semi-continuity of $f^{\tilde c\tilde c}$ and the Portmanteau  Theorem \citep[Theorem 1.3.4]{van1996weak} that 
\begin{align*}
    \limsup_{k\rightarrow \infty}\sqrt{n}(\tilde\mu_k - \mu)(f^{\tilde c_k \tilde c_k}) &\leq   \limsup_{k\rightarrow \infty} \sqrt{n}(\tilde\mu_k - \mu)(f^{\tilde c \tilde c}) + 2\sqrt{n}\norm{f^{\tilde c_k \tilde c_k} - f^{\tilde c \tilde c}}_\infty \leq  \sqrt{n}(\tilde\mu - \mu)(f^{\tilde c \tilde c}).
\end{align*}
Hence, by \Cref{lem:measurability}(ii)
 we conclude that $(\Gproc_n^\mu(f^{ c_{n,m} c_{n,m}}))_{f\in \FC}$ is Borel measurable. Likewise, we conclude  $(\Gproc_m^\nu(f^{ c_{n,m}}))_{f\in \FC}$ is Borel measurable. 

Consequently, by {\ref{ass:BTwo}} we infer, for $n,m\rightarrow \infty$, that
$$\Big(\Gproc_n^\mu(f^{cc}) -  \Gproc_n^\mu(f^{ c_{n,m} c_{n,m}}), \Gproc_m^\nu(f^{c}) - \Gproc_m^\nu(f^{ c_n})\Big)_{f\in \FC} \pto (0,0) \quad \text{in } C_u(\FC)^2.$$
The claim now follows by a combination of Slutzky's lemma and the continuous mapping theorem \citep[Example 1.4.7, Theorem 1.11.1]{van1996weak}. \qed

\subsection{Proof of Lemma \ref{lem:SupportLimitingProcess}}\label{subsec:SupportLimitingProcessProof}  
  The first claim follows by an observation in \citet{Roemisch04} since the set of probability measures $\PC(\XC)$ is convex. For additional insights see \citet[Proposition 4.2.1]{aubin2009setValued} 
  
  For the second claim consider a sequence $\Delta_n= (\tilde \mu_n - \mu)/t_n $ with $t_n>0$ and $\tilde \mu_n\in \PC(\XC)$ such that $\norm{\Delta_n - \Delta}_{\tilde \FC}= \sup_{f\in \tilde \FC}|\Delta_n(f) - \Delta(f)|\rightarrow 0.$ Then, it follows from triangle inequality that \begin{align*}
  	\left|\Delta(f) - \Delta(f')\right| &= \left|\Delta_n(f) - \Delta_n(f')  + (\Delta -\Delta_n)(f)  + (\Delta -\Delta_n)(f') \right|\\
		& \leq\left|\Delta_n(f-f')\right|  + 2\norm{\Delta -\Delta_n}_{\tilde \FC}
  	  \end{align*}
Herein, the first term vanishes since $\Delta_n(f-f') = (\tilde \mu_n- \mu)(\kappa)/t_n = 0$, whereas the second term converges for $n\rightarrow \infty$ to zero. Hence, $\Delta(f) = \Delta(f')$. 
  
  The third claim relies on Portemanteau's theorem \citep[Lemma 1.3.4]{van1996weak} which asserts using the notion of outer probabilities $\pOut$ that 
  \[
  \pushQED{\qed}
    \Prob\left( \Gproc^\mu\in T_{\mu}\PC(\XC)\right) \geq \limsup_{n\rightarrow \infty}  \pOut\left( \sqrt{n}(\mu_n - \mu) \in T_{\mu}\PC(\XC)\right) = 1. \qedhere\popQED
\]

\subsection{Proof of Lemma \ref{lem:auxiliary36}}\label{subsec:auxiliary36Proof}

	We start by proving $(i)$. Note for $\kappa \in \RR$ that
	\[(f+\kappa)^c(y)=\inf_{x\in\X}c(x,y)-f(x)-\kappa=f^c(y)-\kappa,\]
	which yields the claim. 
	To show assertion $(ii)$, observe by \Cref{lem:CtrafoLipschitz} that 
	\begin{equation}\label{eq:prop36trickone}\norm{g^{(c+\Delta^c)(c+\Delta^c)}}_\infty\leq \norm{g}_\infty+2\norm{c+\Delta^c}_\infty\leq B.\end{equation} 
	Further, we find that
	\begin{equation}\label{eq:prop36tricktwo}-\norm{c}_\infty-\sup_{x \in \X}g^{(c+\Delta^c)(c+\Delta^c)}(x)\ \leq g^{(c+\Delta^c)(c+\Delta^c)c}(y)\leq\norm{c}_\infty- \sup_{x \in \X}g^{(c+\Delta^c)(c+\Delta^c)}(x).\end{equation}
	Using part $(i)$ of this lemma, we obtain 
	\[g^{(c+\Delta^c)(c+\Delta^c)cc}=\left(\left(g^{(c+\Delta^c)(c+\Delta^c)}\right)^{c}+\sup_{x \in \X}g^{(c+\Delta^c)(c+\Delta^c)}(x)\right)^c+\sup_{x \in \X}g^{(c+\Delta^c)(c+\Delta^c)}(x).\]
	Combining \eqref{eq:prop36trickone} and \eqref{eq:prop36tricktwo} with the above equation demonstrates that  $g^{(c+\Delta^c)(c+\Delta^c)cc}\in\HC_c+[-B,B]$ and hence yields the claim.\qed

\section{Elementary Analytical Results}

\begin{lem}\label{lem:limsup}
	Consider a real-valued sequence $(a_n)_{n\in \NN}$ and let $K\in \RR$. 
	\begin{itemize}
		\item[(i)] If for any subsequence $(a_{n_k})_{k \in \NN}$ there exists a subsequence $(a_{n_{k_l}})_{l \in \NN}$ with $\limsup_{l\to\infty}a_{n_{k_l}}\leq K$, then it follows that  $\limsup_{n\to\infty}a_n\leq K$.
		\item[(ii)] If for any subsequence $(a_{n_k})_{k \in \NN}$ there exists a subsequence $(a_{n_{k_l}})_{l \in \NN}$ with $\liminf_{l\to\infty}a_{n_{k_l}}\geq K$, then it follows that $\liminf_{n\to\infty}a_n\geq K$.
		\item[(iii)] If for any subsequence $(a_{n_k})_{k \in \NN}$ there exists a subsequence $(a_{n_{k_l}})_{l \in \NN}$ with $\lim_{l\to\infty}a_{n_{k_l}}= K$, then it follows that $\lim_{n\to\infty}a_n= K$.
	\end{itemize}
\end{lem}

\begin{proof}
We only prove $(i)$ and note that $(ii)$ and $(iii)$ can be shown analogously. Assume that $\limsup_{n \rightarrow \infty} a_n = \inf_{n\in \NN} (\sup_{m \geq n} a_m) \geq K+\eps$ for some $\eps>0$. Since $(\sup_{m \geq n} a_m)_{n\in \NN}$ is decreasing in $n$, this would imply that $\sup_{m \geq n} a_m\geq K + \eps$ for all $n \in \NN$. Hence, there would exist a subsequence of $(a_n)_{n \in \NN}$, say $(a_{n_l})_{l \in \NN}$, with $a_{n_l} \geq K+\eps/2$ for all $l \in \NN$. However, this would assert $\liminf_{l \rightarrow \infty} a_{n_{l}} \geq K+\eps/2 > K$, contradicting the assumption. Thus, $\limsup_{n \rightarrow \infty} a_n \leq K$.
\end{proof}

\begin{lem}\label{lem:pseudoMetricSpaceCompact}
Let $(\XC,d_\XC)$ be a compact metric space and consider a continuous (pseudo-)metric $\tilde d_\XC$ on $\XC$. Then, $(\XC, \tilde d_\XC)$ is a compact (pseudo-)metric space. Moreover, given a Polish space $\YC$ it follows that $C((\XC, \tilde d_\XC)\times \YC)\subseteq C((\XC, d_\XC)\times \YC)$. 
\end{lem}

\begin{proof}
  The (pseudo-)metric properties are clearly fulfilled for $(\XC, \tilde d_\XC)$. 
  By continuity of $\tilde d_\XC$ under $d_\XC$ the canonical inclusion $\iota\colon (\XC, d_\XC) \rightarrow (\XC, \tilde d_\XC), x\mapsto x$ is continuous.  As the image of a compactum under a continuous map is again compact the first claim follows. 
  For the second claim, take $h\in C((\XC, \tilde d_\XC)\times \YC)$. Then, the composition map $\XC\times \YC\rightarrow \RR, (x,y) \mapsto h(\iota(x),y)$  is continuous and therefore the canonical embedding $h\circ(\iota, \Id_\YC)$ of $h$ is included in $C((\XC, d_\XC)\times \YC)$. 
\end{proof}

\end{document}